\documentclass[11pt]{article}

\usepackage{graphicx}
\usepackage{marvosym}

\usepackage{amsmath}
\usepackage{amssymb}
\usepackage{amsthm}
\usepackage{tikz}
\usepackage{tikz-cd}
\usepackage{geometry}
\usepackage[utf8]{inputenc}
\usepackage[english]{babel}
\usepackage[backend = biber, style = alphabetic, maxnames = 99, maxbibnames = 99]{biblatex}
\usepackage{hyperref}
\usepackage[dvipsnames]{xcolor}
\usepackage{mathabx}
\usepackage{pdflscape}

\newcommand{\reals}{\mathbb{R}}
\newcommand{\intgr}{\mathbb{Z}}
\newcommand{\ntrl}{\mathbb{N}}

\renewcommand{\subset}{\subseteq}
\renewcommand{\phi}{\varphi}
\renewcommand{\epsilon}{\varepsilon}

\newtheorem{lem}{Lemma}

\renewcommand{\hat}{\widehat}
\renewcommand{\bar}{\overline}

\newcommand{\supp}{\text{supp}}

\newcommand{\M}{\mathcal{M}}

\newcommand{\prob}{\mathbb{P}}

\newcommand{\N}{\mathcal{N}}
\newcommand{\Z}{\mathbb{Z}}

\newcommand{\diam}{\textnormal{diam}}

\renewcommand{\S}{\mathcal{S}}

\newcommand{\llangle}{\langle\langle}
\newcommand{\rrangle}{\rangle\rangle}

\newcommand{\Ax}{\textrm{Ax}}
\newcommand{\EE}{\mathcal{E}}
\newcommand{\A}{\mathcal{A}}

\newtheorem{theorem}{Theorem}[section]
\newtheorem{lemma}[theorem]{Lemma}
\newtheorem{corollary}[theorem]{Corollary}

\theoremstyle{definition}
\newtheorem{definition}[theorem]{Definition}
\theoremstyle{remark}
\newtheorem{remark}[theorem]{Remark}

\theoremstyle{hypothesis}
\newtheorem{hypothesis}[theorem]{Hypothesis}

\title{Random Quotients of Hyperbolic Spaces with Respect to WPD Actions and Linear Connectivity of the Boundary}

\author{Inhyeok Choi and Henry Talbott}

\begin{document}
\maketitle
\begin{abstract}
We study random quotients of acylindrically hyperbolic groups and hierchically hyperbolic groups, whose (partial) hyperbolicity is captured by their actions on Gromov hyperbolic spaces. We show that if an acylindrically hyperbolic group $G$ acts on a Gromov hyperbolic space $X$ coarsely transitively with a WPD element, then the quotient of $X$ by the normal closure of a random element of $G$ has linearly connected boundary. If $G$ is a group with a (relative) hierarchically hyperbolic group structure, with a quasi-tree top domain $X$, then in fact the quotient of $X$ by the normal closure of a random element of $G$ has boundary homeomorphic to the N{\"o}beling curve.

Our method relies on a careful study of the configuration of translates of typical periodic geodesics in $X$. 
\end{abstract}

\section{Introduction}

We consider \emph{few-relators random quotients} of groups acting on hyperbolic spaces by isometries. We are especially interested in the action $G\curvearrowright X$ of a \emph{hierarchically hyperbolic group} (HHG) on the maximal domain of the corresponding hierarchically hyperbolic space (HHS) structure, such as a mapping class group acting on its curve graph, or a right-angled Artin group (RAAG) acting on its coned-off contact graph. Our model of random quotients of HHGs is largely equivalent to the model defined by Abbott-Berlyne-Mangioni-Ng-Rasmussen, for which the authors obtained the result that random quotients of HHGs are hierarchically hyperbolic \cite{ABMNR26}.

Specifically, let $G\curvearrowright X$ be a cobounded, non-elementary, WPD, and \emph{action path-connected} (see Definition \ref{def:actiopathconnected}) group action by isometries on a Gromov hyperbolic space; WPD and action-path-connectedness are immediate in the HHG setting. Given admissible random walks $(w^{(1)}_n,...,w_n^{(k)})_{n>0}$ on $G$, we draw the $n$-step elements, define
$$N:=\langle\langle w^{(1)}_n,...,w_n^{(k)}\rangle\rangle,$$
and consider the induced action $G/N\curvearrowright X/N$. By the aforementioned work, $X/N$ is Gromov hyperbolic, so we may consider the geometry of its boundary equipped with the visual metric.

We prove the following results. All asymptotic probabilities are taken with respect to the limit $n\to\infty$. Recall that a space is linearly connected if points may be joined by paths with diameters coarsely bounded above by the distances between those points.

\begin{theorem}
\label{thm:maintheoremintro}

Let $G/N\curvearrowright X/N$ be as above. Then a.a.s. $\partial (X/N)$ with its intrinsic visual metric is linearly connected.

\end{theorem}

While many known examples of HHGs have quasi-tree maximal domains (e.g. cubical HHGs such as RAAGs) \cite{BHS17b}, the canonical maximal domains for mapping class groups -- the curve graphs -- are un-tree-like in the sense of possessing highly connected boundaries \cite{DG09}. We find this dichotomy striking, especially in light of the role of mapping class groups as prototypical HHGs. For general acylindrically hyperbolic groups (a class including all HHGs), the foundational work of Bestvina-Bromberg-Fujiwara on projection complexes gives a general recipe for constructing cobounded WPD actions on quasi-trees for any such group \cite{BBF15}. In contrast, to the best of the authors' knowledge, we lack a general recipe to construct cobounded WPD actions on hyperbolic spaces that are not quasi-trees. 

Theorem \ref{thm:maintheoremintro} shows that acylindrically hyperbolic groups acting on hyperbolic spaces that are not quasi-trees are highly abundant. In fact, such actions may be obtained by applying a quotient to \emph{any} action of an HHG on its maximal domain, so long as the quotient is sufficiently `generic' or `random'. As a consequence, we obtain a recipe for constructing many HHGs with maximal domains that are not quasi-trees.

In some cases, we can pin down the topology of the quotient boundary more precisely. We first obtain a general result controlling the dimension of the boundary.

\begin{theorem}
\label{thm:boundaryincrement}

Let $G/N\curvearrowright X/N$ be as above.  Then a.a.s.
$$1\leq \dim(\partial (X/N))\leq \dim(\partial X)+1.$$

\end{theorem}

Note that the lower bound follows immediately from Theorem \ref{thm:maintheoremintro}.

For our next result, note that the \emph{dense unbounded domains} property required below (see Definition \ref{def:largedense}) is expected to hold for most relatively hierarchically hyperbolic groups (RHHGs). In particular, it holds for mapping class groups, for any right-angled Artin group with a defining graph that is connected and not a join, and many right-angled Coxeter groups. (Although note that Theorem \ref{thm:nobeling} does not itself apply to mapping class groups due to the quasi-tree assumption.)

\begin{theorem}
\label{thm:nobeling}

Let $G/N$ and $X/N$ be as above. Additionally, assume that $G\curvearrowright X$ is the action of a relatively hierarchically hyperbolic group on its maximal domain, that the RHHG structure for $G$ satisfies the \emph{dense unbounded domains} property, and that $X$ is a quasi-tree. Then a.a.s. $\partial (X/N)$ is homeomorphic to the Nöbeling curve.

\end{theorem}

Recall that the \emph{Nöbeling curve} is a topological space of dimension $1$ that may be constructed by deleting from $\reals^3$ all points with at least two rational coordinates. To give a sense for this space, note that despite having topological dimension $1$, it remains connected after deleting any countable subset.

We believe Theorem \ref{thm:nobeling} is particularly interesting in light of analogous results on curve graph boundaries, discussed below.

\subsection{Applications}

We identify two classes of (relatively) hierarchically hyperbolic groups satisfying the hypotheses of Theorem \ref{thm:nobeling}. 
\begin{enumerate}
\item Graphs of groups with edges representing commutation, as defined in \cite{BR22}, satisfying certain simple conditions. This class includes all right-angled Artin groups $G$ such that the defining graph of $G$ is connected and is not a join, and many right-angled Coxeter groups.
\item Fundamental groups of non-geometric $3$-dimensional graph manifolds, using the hierarchical hyperbolicity structure developed in \cite{HRSS25}.

\end{enumerate}

Two examples are illustrated in Figure \ref{fig:raag}. In both examples, we assign a group to each vertex, take a free product of these groups, and then declare the relations that generators of groups corresponding to adjacent vertices commute. For the left pentagon graph, we assign $\intgr$ to each vertex; this is a right-angled Artin group. For the right graph, we arrange five pentagon graphs in a super-pentagon, and add edges between all pairs of vertices in adjacent pentagons. We then assign the group $\intgr/2\intgr$ to each vertex; this is a right-angled Coxeter group.

\begin{figure}
\centering
    \includegraphics[width=0.8\textwidth]{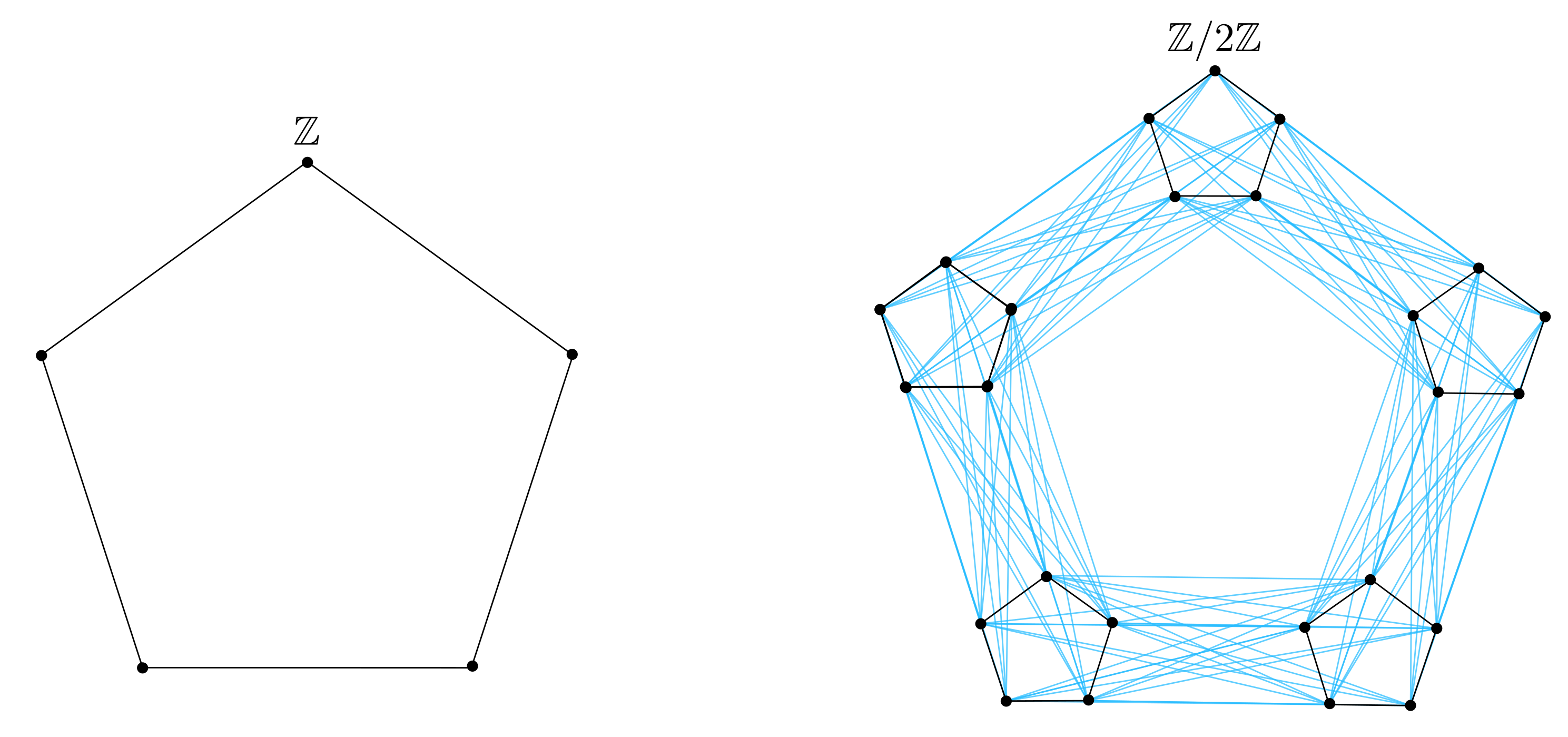} 
 \caption{Two defining graphs that provide examples of groups satisfying the hypotheses of Theorem \ref{thm:nobeling}. The left graph is taken as the defining graph of a right-angled Artin group. The right graph is constructed by taking five pentagon graphs, placing these graphs in a pentagonal configuration, and adding edges between all pairs of vertices in adjacent graphs. This graph is taken as the defining graph of a right-angled Coxeter group.}
 \label{fig:raag}
\end{figure}

\subsection{Background}

Many groups in geometric group theory exhibit hyperbolic behavior through actions on metric spaces; \emph{hierarchical hyperbolicity} is a way of axiomitizing a specific strong form of hyperbolicity observed across many important groups \cite{BHS17b, BHS19}, while acylindrical hyperbolicity likewise axiomitizes a somewhat weaker form observed in groups such as $\text{Out}(F_n)$ \cite{DO16}. The prototypical example of hierarchical hyperbolicity is the Masur-Minsky machine for mapping class groups, in which mapping classes may be described via coordinates on Gromov hyperbolic spaces, the \emph{curve graphs}, obtained from subsurfaces \cite{MM99, MM00}. Curve graphs have proved to be extremely useful tools in their own right \cite{MM99, YM06, KR07, BM08, BKMM12, BCM12, BBF15}. The HHG machinery has since seen application to  many other notable groups, including extra-large Artin groups \cite{HMS22} and (very recently) the fundamental group of the moduli space of smooth complex cubic surfaces \cite{AW26}. 

The Gromov boundary and its generalizations are central tools in the study of Gromov hyperbolic spaces, and are key to applications such as rigidity \cite{CT07, BM91, BB98, KK00, BK07, CH17, DMS25}. The boundary of a curve graph is the corresponding space of \emph{ending laminations} \cite{EK22}; boundary path connectivity of curve graphs was proven for special cases by Leininger-Schleimer and Leininger-Mj-Schleimer \cite{LS09, LMS11}, and in full generality by Gabai \cite{DG09}. Gabai further proved that curve graphs of punctured spheres have Nöbeling space boundaries \cite{DG09}, generalizing prior work of Hensel-Przytycki \cite{HP11}. In turn, Wright showed that all curve graphs have linearly connected boundaries \cite{AW24}.

Gromov proposed random groups (specifically, random quotients of free groups) as models for the expected properties of `typical' hyperbolic groups \cite{MG93}. Our results may be viewed as extensions of previous results on random hyperbolic groups to the HHG setting; Champetier showed few-relator random quotients of free groups have boundaries homeomorphic to the Menger curve \cite{CC95}, while Dahmani-Guirardel-Przytycki showed the same for random quotients of free groups in the density model \cite{DGP11}. We also note recent work of Ioannou on hyperbolic groups with Menger curve boundaries \cite{RI26}. One may likewise study the statistical properties of random elements of groups; in our setting of acylindrically hyperbolic group actions $G\curvearrowright X$, we rely substantially on deep results on random walks due to Maher--Tiozzo \cite{MT18, MT21}.

\subsection{Methods}

Our methods have their earliest antecedents in the work of Champetier \cite{CC95}, who studied random quotients of free groups and showed that these quotients have boundaries homeomorphic to the Menger curve (also called the Menger sponge). In order to adapt this approach to general acylindrically hyperbolic groups and hierarchically hyperbolic groups, one has to understand the structure of the normal closures of WPD elements. This structure was first systematically understood by Dahmani-Guirardel-Osin \cite{DGO17}, and further developed by Clay--Mangahas--Margalit \cite{CMM21} and Clay--Mangahas \cite{CM22}. These works introduce and develop the \emph{rotating/spinning families} machinery, and while this machinery does not appear directly in our work, the core ideas underlie many of our arguments. Another important input is the projection complex machinery due to Bestvina--Bromberg--Fujiwara \cite{BBF15}, which was further refined by Bestvina--Bromberg--Fujiwara--Sisto \cite{BBFS19}. 

To illustrate our methods for Theorem \ref{thm:maintheoremintro}, we sketch the case of the free group. Let $2\le k \le n$, let $G= \langle a_{1}, a_{2}, \ldots, a_{n}\rangle$ be a free group of rank $n$, and let $X$ be its standard Cayley graph.

Let us now consider a probability measure $\mu$, the support of which is $\{a_{1}^{\pm}, \ldots, a_{n}^{\pm}\}$. Then the random walk generated by $\mu$ is admissible with bounded steps on $X$. By the work of Maher-Tiozzo, there exists $\lambda>0$ such that the following property holds for asymptotic probability 1: the element $w$ obtained at step $n$ of the random walk satisfies \[
0.9999\lambda n < d_{X}(e, w) <1.0001\lambda n,\quad (w^{-1} \cdot w)_{e} \le \log n.
\]
Moreover, $[e, w]$ and $g[e, w]$ will overlap for length $<0.0001 n$ for each $g \in G \setminus \{e\}$. 

Now consider the quotient $X / \llangle w \rrangle$. What is the geometry of this space? It is no longer a tree, as each left coset of $\langle w \rangle$ gives rise to loops of length $\approx n$. However, $X / \llangle w \rrangle$ retains Gromov hyperbolicity almost surely \cite{MG93}; this result in the general non-proper setting is studied in \cite{ABMNR26}.

We now describe our proof of linear connectivity. The basic idea, as in \cite{CC95} and also \cite{AW24}, is that given a path lying in the $n$-sphere $S_n$ centered at some basepoint $x_0$, we may `push it out' to a finite-length path lying in $S_{n+1}\cup S_{n+2}\cup...\cup S_{n+K}$. Furthermore, we may do this pushing in a well-controlled way. Specifically, given a path $\gamma_n= (v_{0}, \ldots, v_{k})$ on the $n$-sphere $S_n$, for each $0 \le i \le k$ there exists a ``prong outward" at $v_{i}$. Namely, there exists a vertex $w_{i}$ at distance  from $v_{i}$ such that the directed paths $\overrightarrow{v_{i} w_{i}}$'s are all directed away from the center $x_{0}$ and are repelling each other; in numeric terms, both $(w_i\cdot x_0)_{v_i}$ and $(w_i\cdot w_{i+1})_{v_i}$ are at most, say, $4$. As $w$ is a generic word, it contains all $w_{i}^{-1}(v_{i}^{-1}v_{i+1}) w_{i+1}$ (of length $\le 21$) as subwords. Hence, in $X/\llangle w \rrangle$, we have loops attached at each of $\{v_{0}, \ldots, v_{n}\}$, where adjacent loops overlap for length $10$. We assign the label $u_i$ to the point at the apex of the $i$th loop, and join $v_0$ and $v_k$ through a path $\gamma_{n+1}$ that `skips' along the loops, passing through each $u_i$. This path is illustrated in Figure \ref{fig:skip}. As the prongs are directed away from $x_0$, this path lies in $S_{n+1}\cup S_{n+2}\cup...\cup S_{n+K}$. We note that the overlapping construction of the loops is essential to keep the path from re-entering $S_n$; as a particular consequence, the Gromov products between consecutive $w_{1}, \ldots, w_{m}$ are strictly larger than the products between consecutive $v_{1}, \ldots, v_{k}$. 

We think of $\gamma_{n+1}$ as a refined version of $\gamma_n$, and by inductively pushing paths $\gamma_{n}$ deeper and deeper, we can draw limiting paths on $\partial X / \llangle w \rrangle$ connecting arbitrary boundary points. While the pushing out procedure increases the length of the path at each step, we retain sufficient geometric control that the visual-metric diameter of the resultant boundary path is linearly bounded.

Our other results follow from further analysis of this basic scheme.

\begin{figure}[htbp]
    \centering 
    \includegraphics[width=0.8\textwidth]{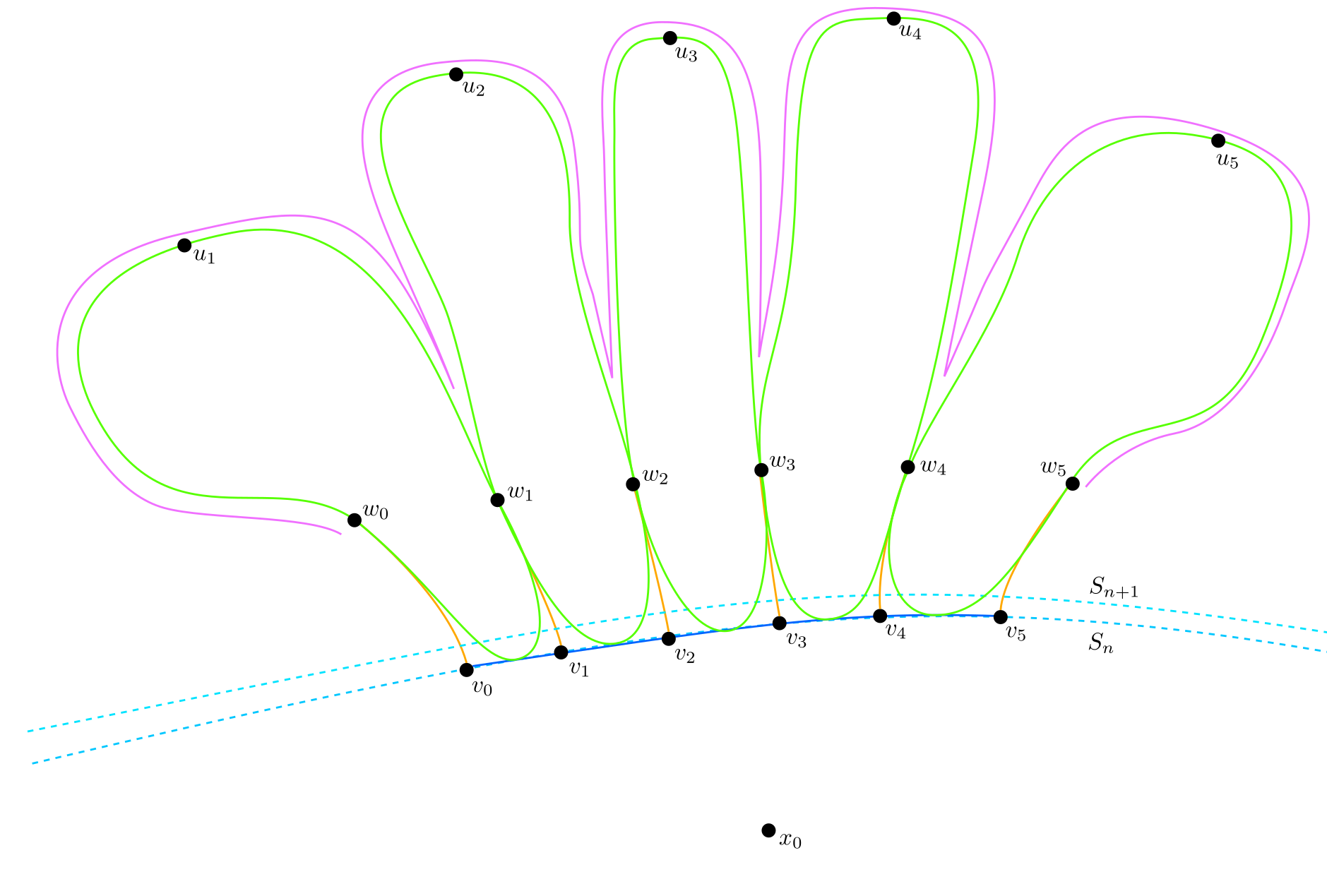} 
    \caption{The heuristic idea for the inductive pushing out procedure. The pink curve represents the pushed-out path, which skips over the attached loops.}
    \label{fig:skip}
\end{figure}

\subsection{Acknowledgments}

The authors would like to thank Alex Wright for suggesting the problem and Carolyn Abbott for further helpful discussions at the onset of the project. Furthermore, the authors would like to thank Giorgio Mangioni for his help with identifying applications of Theorem \ref{thm:nobeling}, in particular the $3$-manifold group example. The authors would also like to thank George Domat and Alessandro Sisto for helpful conversations.

Talbott was supported by the National Science Foundation Graduate Research Fellowship under Grant. No. DGE 2241144. Choi was supported by the KIAS Individual Grant (MG091901) at KIAS. Choi was also supported by the National Science Foundation under Grant No. DMS-2424139, while the authors were in residence
at the Simons Laufer Mathematical Sciences Institute in Berkeley, California, during the Spring 2026 semester. 

\subsection{Artificial Intelligence Declaration}

Talbott briefly experimented with using LLMs. The sole outcome of this usage was that the authors became aware of the Urysohn-Menger Sum Formula as a potential tool. Aside from this instance, no AI was used in the preparation of this paper. The text and images of the paper are entirely human-authored.

\section{Preliminaries}

\subsection{Notation}

Throughout the text, we will use the notation $A=_C B$ for real numbers $A,B,C$ with $C\geq 0$ to denote $|A-B|\leq C$.

\subsection{Hyperbolic Spaces and Projections}

We review some hyperbolic geometry and introduce the key hypotheses we will need. All metric spaces we will consider in this paper are geodesic metric spaces, and in fact are simplicial graphs. As a consequence, the distance function is integer-valued. This assumption gives essentially no loss of generality, as whenever a group $G$ acts on a metric space $X$ by isometries, $X$ is $G$-equivariantly quasi-isometric to a simplicial graph $X_S$, such that $G$ also acts on $X_S$ by isometries. We will further assume that $G$ is finitely generated, so that $G$ may be equipped with a word metric.

Given two points $x,y\in X$, we will use the notation $[x,y]_X$ to represent a geodesic from $x$ to $y$. Recall that a geodesic metric space $X$ is $\delta$\emph{-hyperbolic} (or \emph{Gromov hyperbolic} if for all geodesic triangles with sides $[x,y]_X$, $[y,z]_X$, and $[z,x]_X$, each side is contained in the $\delta$-neighborhood of the union of the other two sides. It is often useful to consider \emph{quasi-geodesics} instead of geodesics.

Throughout this section, we will assume a fixed $\delta$-hyperbolic simplicial graph $X$.

\begin{definition}

A $(\lambda,c)$-quasi-geodesic is a map $\gamma$ from a (possibly unbounded) subinterval $I$ of $\intgr$, such that for $s,t\in I$,
$$\frac{1}{\lambda}d_\intgr(s,t)-c\leq d_X(\gamma(s),\gamma(t))\leq \lambda d_\intgr(s,t)+c$$
A $c$\emph{-coarse geodesic} is a $(1,c)$-quasi-geodesic.

A $c$\emph{-coarse geodesic at speed }$L$ is map $\gamma$ from a subinterval $I\subset\intgr$ to $X$ such that for $s,t\in I$,
$$Ld_{\ntrl}(s,t)-c\leq d_X(\gamma(s),\gamma(t))\leq Ld_\ntrl(s,t)+c$$

\end{definition}

\begin{remark}

We will often use $\gamma$ to refer interchangeably to a quasi-geodesic map and its image.

\end{remark}

\begin{definition}

A subset $A\subset X$ is $K$-quasi-convex if for all $x,y\in A$ and all $[x,y]_X$, $[x,y]_X\subset \N_K(A)$.

\end{definition}

We need a few standard lemmas and definitions.

\begin{lemma}[Morse Lemma]\label{lem:morse}

For all $\lambda\geq 0$ and $c\geq 0$, there is a constant $\Phi\geq 0$ dependent only on $\delta$, $\lambda$, and $c$ satisfying the following property: for any two $(\lambda,c)$-quasigeodesics $\gamma_1$ and $\gamma_2$ with the same endpoints in $X$, the Hausdorff distance between $\gamma_1$ and $\gamma_2$ is at most $\Phi$.

\end{lemma}

\begin{definition}[Closest Point Projections]

Assume $A\subset X$ is $K$-quasi-convex for some $K$. Given $x\in X$, we define
$$\pi_A(x)=\{y\in A:d(x,A)=d(x,y)\}.$$
For $B\subset X$, we define $\pi_A(B)=\bigcup_{x\in B}\pi_A(x)$. Lastly, for $A,B\subset X$, we define
$$d_A(B,C)=\diam(\pi_A(B)\cup\pi_A(B))$$

\end{definition}

\begin{remark}

Note that there are two quite different forms of projection that will be denoted with $\pi$: we will use $\pi_A$ to denote closest-point projection \emph{in} $X$ to a quasi-convex subset of $X$, as well as $\pi_{X/N}$ to denote the quotient projection from $X$ to $X/N$.

\end{remark}

\begin{lemma}

Assume $A\subset X$ is $K$-quasi-convex. Then there is a constant $D$ dependent only on $\delta$ and $K$ so that for any $x\in X$, $\diam(\pi_A(x))\leq D$.

\end{lemma}

For the next lemma, the reader should note that the name `Bounded Geodesic Image' is also frequently applied to a conceptually related but somewhat different axiom for hierarchically hyperbolic spaces.

\begin{lemma}[Bounded Geodesic Image]
\label{lem:BGI}

Let $A\subset X$ be $K$-quasi-convex, and fix $\lambda,c\geq 0$. Then there exists $B,B'>0$ dependent only on $\delta$, $K$, $\lambda$, and $c$ so that if $\gamma$ is a $(\lambda,c)$-quasi-geodesic disjoint from $\N_{B'}(A)$, then $\diam(\pi_A(\gamma))\leq B$.

\end{lemma}

\begin{lemma}[Large Projection Fellow-Traveling]
\label{lem:BGIrouting}

Let $A\subset X$ be $K$-quasi-convex. Then there exists $\Omega$ dependent only on $\delta$ and $K$ so that for any points $x,y\in X$ and any $x'\in\pi_A(x)$, $y'\in\pi_A(y)$, if $d_A(x',y')\geq\Omega$, then any concatenation $[x,x']_X\cup[x',y']_X\cup[y',y]_X$ is a $(1,\Omega)$-quasi-geodesic.

\end{lemma}

In particular, $[x,x']_X\cup[x',y']_X\cup[y',y]_X$ and any $[x,y]_X$ must fellow-travel. 

\subsection{Gromov Products and the Boundary}

\begin{definition}

Given $x,y,z\in X$, the \emph{Gromov product} $(x\cdot y)_z$ is defined as
$$(x\cdot y)_z=\frac{d_X(x,z)+d_X(y,z)-d_X(x,y)}{2}.$$

\end{definition}

Note that this definition makes sense in any metric space, although it is typically most useful in Gromov hyperbolic spaces.

The following observations are elementary.

\begin{lemma}
\label{lem:gromovprops}
Let $x, y, z, w$ be points in $X$. Then the following properties hold:
\begin{enumerate}
\item We have 
$$
d(z, [x, y]) - 2\delta \le (x \cdot y)_{z} \le d(z, [x, y]).
$$
\item We have
$$(x\cdot y)_z +(z\cdot y)_x=d_X(x,y)$$
\item (Triangle inequality for Gromov products.) We have
$$(x\cdot y)_z\geq\min((x\cdot w)_z,(w\cdot y)_z)-10\delta$$
\end{enumerate}
\end{lemma}

\begin{definition}

Let $(x_0,x_1,...,x_n)$ be a sequence of points in $X$. We say these points are $C$\emph{-aligned} for some constant $C>0$ if for all $1\leq i\leq n-1$,
$$(x_{i-1}\cdot x_{i+1})_{x_i}<C$$

\end{definition}

\begin{lemma}[{\cite[Lemma 3.7]{Gou22}}]
Let $(x_{0}, \ldots, x_{n})$ be a sequence of $C$-aligned points in $X$ such that \[
d_{X}(x_{i}, x_{i+1}) \ge 2C + 2\delta + 1 \quad (i=1, \ldots, n-2).
\]
Then any subsequence $(x_{i_1}, \ldots, x_{i_m})$ of $(x_{0}, \ldots, x_{n})$ is $(C+\delta)$-aligned. In particular, every geodesic $[x_{0}, x_{n}]_{X}$ passes through the $(C + 3\delta)$-neighborhood of $x_{i}$ for each $i=1, \ldots, n-1$.

\end{lemma}

We now turn to the Gromov boundary. Note that this definition does not in fact depend on the choice of basepoint.

\begin{definition}

Fixing a basepoint $x_0\in X$, the \emph{Gromov boundary} $\partial X$ of $X$ is the set of all sequences of points $(x_i)$ such that $\liminf_{i,j}(x_i\cdot x_j)_{x_0}=\infty$, up to the equivalence relation that $(x_i)\sim(y_i)$ if $\lim_{i\to\infty}(x_i\cdot y_i)_{x_0}=\infty$.

\end{definition}

The Gromov boundary carries a natural topology defined via \emph{neighborhood bases} of points. Specifically, for all $\xi\in\partial X$ and $r\geq 0$, define
$U(\xi,r)$ to be the set of all $\chi\in\partial X$ such that there exist sequences $(x_i),(y_i)$ with $(x_i)$ in the equivalence class defining $\xi$, $(y_i)$ in the equivalence class defining $\chi$, and $\liminf_{i,j\to\infty}(x_i\cdot y_j)_{x_0}\ge r$. The collection of all $U(\xi,r)$ forms a neighborhood system at $\xi$.

\begin{remark}

Note that the sets $U(\xi,r)$ are not open in general under the neighborhood basis topology.

\end{remark}

\begin{lemma}

If $X$ is proper, then for any $\xi\in \partial X$, there exists a ray $\gamma:\ntrl_{\geq 0}\to X$ such that $(\gamma(i))$ lies in the equivalence class of $\xi$. For general $X$, there exists a $(1,10\delta)$-quasigeodesic ray $\gamma$ with this property.

\end{lemma}

The boundary $\partial X$ is in fact metrizable via a \emph{visual metric}.

\begin{definition}

For $\xi,\chi\in\partial X$ and $x_0\in X$, we define
\begin{equation}\label{eq:boundarygromov}(\xi\cdot \chi)_{x_0}=\inf_{\substack{(x_i)\in[\xi] \\ (y_i)\in [\chi]}}\liminf_{i,j\to\infty}(x_i\cdot y_j)_{x_0}\end{equation}

\end{definition}

\begin{lemma}
\label{lem:consistentgromov}

For $\xi,\chi\in\partial X$ and any sequences $(x_i)\in[\xi]$ and $(y_i)\in[\chi]$, we have
$$\liminf_{i,j\to\infty}(x_i\cdot y_j)_{x_0}=_{100\delta} (\xi\cdot \chi)_{x_0}.$$

\end{lemma}

\begin{theorem}

There exists a metric $d_v$ on $\partial X$ such that there are constants $a>0$ and $C\geq 1$ depending only on $\delta$ so that for all $\xi,\chi\in\partial X$,
$$\frac{1}{C}a^{-(\xi\cdot\chi)_{x_0}}\leq d_v(\xi,\chi)\leq Ca^{-(\xi\cdot\chi)_{x_0}}.$$

\end{theorem}

The visual metric topology agrees with the topology on $\partial X$ presented earlier.

\subsection{Group Actions}

We briefly go over the needed assumptions on the action $G\curvearrowright X$. Recall that $X$ is a simplicial graph and is $\delta$-hyperbolic, and that $G$ acts on $X$ by isometries. We take a fixed basepoint $x_0\in X$, although the definitions given are in fact basepoint-independent.

\begin{definition}
\label{def:loxodromic}

An element $a\in G$ is \emph{loxodromic} if $(u^ix_0)_{i\in\intgr}$ is a $(\lambda,c)$-quasi-geodesic for some $\lambda\geq 1$ and $c\geq 0$. Two loxodromics $a,b\in G$ are \emph{independent} if there exists some $D>0$ so that for all $i,j\in\intgr$,
$$(u^ix_0\cdot v^jx_0)_{x_0}<D$$

\end{definition}

We now define \emph{acylindrical} and \emph{WPD} (weak proper discontinuity) actions; the definition of acylindricity is due to Osin \cite{DO16}, while the definition of a WPD action is due to Bestvina-Fujiwara \cite{BF02}. The WPD condition is strictly weaker than acylindricity. We will only assume that the actions we study are WPD, but we include the definition of acylindricity as a useful point of comparison.

\begin{definition}
\label{def:WPD}

Assume that $G$ acts by isometries on a metric space $X$. The action $G\curvearrowright X$ is acylindrical if for all $K\geq 0$, there are constants $R(K)$ and $N(K)$ such that for all $x,y\in X$ with $d(x,y)\geq R(K)$, the coarse stabilizer
$$\text{Stab}_K(x,y):=\{g\in G:d_X(x,gx)\leq K \text{ and }d(y,gy)\leq K\}.$$
has cardinality at most $N(K)$.

A specific element $\phi\in G$ is WPD if for any $x\in X$ and any $K\geq 0$, there exists $N_0>0$ such that for $N\geq N_0$, $\text{Stab}_K(x,\phi^Nx)$ is finite.

The action $G\curvearrowright X$ is WPD if $G$ has at least one WPD element.

\end{definition}

\subsection{Probability Measures, Random Walks, and Axes}

We will consider a finite set of probability distributions $\mu_1,...,\mu_k$ on $G$. 

\begin{definition}
\label{def:permissible}

Given a probability distribution $\mu$ on $G$, $(\mu,G)$ is \emph{permissible} if $\mu$ has finite support that generates $G$, the action of $G$ is non-elementary, and the maximal finite normal subgroup of $G$ is trivial. We will usually refer to just a permissible $\mu$ on $G$, with the requirements on $G$ being implied.

\end{definition}

\begin{definition}
\label{def:randomwalk}

Given a permissible probability measure $\mu$ on $G$, let us consider the $\mu$-i.i.d. $(g_{i})_{i > 0} \in G^{\Z_{>0}}$. The \emph{random walk} $(w_{n})_{n \geq 0}$ generated by $\mu$ is the sequence of random variables defined by \[
w_{n} := g_{1} g_{2} \cdots g_{n}.
\]
\end{definition}

It is straightforward to check that, for every loxodromic element $g \in G$, the translation length \[
L(g, x) := \lim_{n \rightarrow +\infty} \frac{d(x, g^{n} x)}{n}
\]
is well-defined, is independent of the choice $x \in X$, and is strictly positive. Moreover, when $L(g) \ge 100\delta$, the set \[
\big\{ x \in X : d(x, gx) \le L(g) + 4\delta\big\}
\]
is nonempty.

\begin{definition}
\label{def:axis}
Let $g \in G$ be a loxodromic element with translation length greater than $100\delta$. Picking a point $x$ such that $d(x, gx) \le L(g) + 4\delta$, we define the \emph{axis} $\Ax(g)$ as the concatenation
$$\bigcup_{i=-\infty}^\infty [g^ix,g^{i+1}x]_X$$
\end{definition}

\begin{definition}
\label{def:walkpermissible}

Fix $D>0$ and $M>0$, and let $G\curvearrowright X$ be as above. Also fix a parameter $n$. Given $w_1,...,w_k\in G$ satisfying $d_G(e,w_i)\leq n$, we say that $(w_1,...,w_k)\in \mathcal{E}_{D,M}(X)$, or $(w_1,...,w_k)\in\mathcal{E}_{D,M}$ if the space $X$ is clear from context, if there exist respective conjugates $v_1,...,v_k$ of $w_1,...,w_k$ satisfying the following properties. In words, we say $(w_1,...,w_k)$ are \emph{admissible}.

\begin{enumerate}
\item For each $i$, $(v_{i}^{k} x_{0})_{k \in \Z}$ is a $(1, \delta)$-coarse geodesic at speed at least $D$. In particular, the $v_{i}$ are loxodromics on $X$, and each $\Ax(v_i)$ is a $8\delta$-convex $(1,4\delta)$-quasi-geodesic.

\item For each $v_ii$ and for each $u \in G$ with $d_{G}(e,u) \le M$, there exists an element $g \in G$ such that $(x_{0}, gx_{0}, gux_{0}, v_{i} x_{0})$ is $10\delta$-aligned, and such that $d_G(e,g)$ and $d_G(e,u^{-1}g^{-1}v_i)$ are bounded above by $2n$. 

\item \emph{(Non-overlapping axes.)} For any $i \neq j$, there is no $g \in G$ such that a $D/M$-long subsgement of $[x_{0}, v_{i} x_{0}]$ and a $D/M$-long subsegment of $[x_{0}, v_{j}x_{0}]$ are $50\delta$-fellow travelling.

For any $i$, there is no $g \in G \setminus \{id\}$ such that a $D/M$-long subsegment of $[x_{0}, v_{i} x_{0}]$ and a $D/M$-long subsegment of $g[x_{0}, v_{i} x_{0}]$ are $50\delta$-fellow traveling. 

\item For any $v_i$, the stabilizer of $\Ax(v_i)$ is $\langle v_i\rangle\times \mathcal{E}(G)$, where $\mathcal{E}(G)$ is the maximal finite normal subgroup of $G$.

\end{enumerate}

\end{definition}

\begin{theorem}
\label{thm:randomwalkproperties}
Let $G, X, \mu_{1}, \ldots, \mu_{k}$ be as in the setup, and let $(w_{n}^{(1)})_{n>0}, \ldots, (w_{n}^{(k)})_{n>0}$ be independent random walks generated by $\mu_{1}, \ldots, \mu_{k}$, respectively. 
Then there exists $\Delta > 0$ such that for the chosen $n$ and for each $M>0$, \[
\lim_{n \rightarrow \infty} \prob \big( (w_{n}^{(1)}, \ldots, w_{n}^{(k)}) \in \mathcal{E}_{\Delta n, M} \big) = 1.
\]
\end{theorem}

\textbf{Proof.} 
Among the three items that should be established, the first and the second items are due to Maher and Tiozzo \cite[Lemma 5.9, Proposition 5.4]{MT18}. We note that the claim for $\Ax(v_i)$ follows from the coarse geodesic claim by \cite{RC16}, Corollary 2.7 and Lemma 3.2. We note that these items do not require the WPD property of the action, and only requires that the action is non-elementary.

Item 3 follows from combining Proposition 9.12-9.14 of \cite{MT21}. See also Proposition 5.13 of \cite{ABMNR26}. Lastly, Item 4 is a direct consequence of Item 3, and it is also recorded in Proposition 5.1 of \cite{MS19}.

\subsection{Action Path-Connectedness}

Lastly, we will need the following assumption. While this assumption follows from the Hierarchy Paths Theorem for hierarchically hyperbolic groups, we state it independently of the full HHG machinery.

\begin{definition}
\label{def:actiopathconnected}

Let $G$ be a group acting on a geodesic metric space $X$ by isometries, and fix a basepoint $x_0$. We say that the action is \emph{action path-connected} with constants $E$ and $W$ if for every pair of points $x,y\in X$, there exists a sequence of group elements $a_0,a_1,...,a_k$ such that $a_0x_0=x$, $a_kx_0=y$, such that $d_G(a_{i-1},a_i)\leq W$ for all $1\leq i\leq k$, and so that the path $\{a_0x_0,a_1x_0,...,a_kx_0\}$ is a $(E,E)$-quasi-geodesic under some reparameterization.

\end{definition}

To obtain examples of actions with this property, we apply the following theorem.

\begin{theorem}[Thm. 6.11, \cite{BHS19}]
\label{thm:hierarchypathsexist}

Let $\mathcal{X}$ be a hierarchically hyperbolic space. There exists a $\lambda\geq 1$ so that for any $x,y\in\mathcal{X}$, there exists a $(\lambda,\lambda)$-quasi-geodesic $\gamma$ joining $x$ and $y$ in $\mathcal{X}$ so that for any coordinate space $X$ of the hierarchically hyperbolic space structure for $\mathcal{X}$, the projection $\pi_X\circ\gamma$ is a unparameterized $(\lambda,\lambda)$-quasi-geodesic in $X$.

\end{theorem}

\begin{lemma}
\label{lem:hierarchypaths}

Let $G$ be a hierarchically hyperbolic group, and let $X$ be the top-level space of a hierarchically hyperbolic space structure realizing $G$ as a hierarchically hyperbolic group. Assume that the action of $G$ on $X$ is transitive. Then the action $G\curvearrowright X$ is action path-connected for some $E$.

\end{lemma}

\textbf{Proof.} For $G$ to be a hierarchically hyperbolic group, the hierarchically hyperbolic space for which $X$ is the top-level space is the Cayley graph $\mathcal{X}$ of $G$ with respect to some finite generating set. Furthermore, we may assume without loss of generality that there is a projection map $\pi_X:\mathcal{X}\to X$ sending each $a\in \mathcal{X}$ to $ax_0$.

Applying Theorem \ref{thm:hierarchypathsexist}, there is a constant $E$ so that for any $x,y\in X$, we may find $a_0,a_k\in \mathcal{X}$ such that $a_0x_0=x$ and $a_kx_0=y$, and a sequence $a_0,a_1,...,a_k$ forming an $(E,E)$-quasi-geodesic in $\mathcal{X}$ so that its projection $x=a_0x_0,a_1x_0,...,a_kx_0=y$ is a unparameterized $(E,E)$-quasi-geodesic. As $a_0,a_1,...,a_k$ form an $E$-quasi-geodesic in $\mathcal{X}$, we conclude by the definition of distances in $\mathcal{X}$ that $d_G(a_{i-1},a_i)\leq W$ for some $W$ dependent only on $E$. $\square$

\section{Hypotheses}

In this brief section, we establish our standing hypotheses.

\begin{hypothesis}
\label{hyp:main}

Let $G$ be a finitely generated group acting by isometries on a simplicial graph $X$ with basepoint $x_{0}$, and let $w_{1}, \ldots, w_{k} \in G$. We assume the following for some fixed constants $E$ and $M > 10^{20}$:
\begin{enumerate}
\item The space $X$ is Gromov hyperbolic with hyperbolicity constant $\delta>1000$.
\item The action $G\curvearrowright X$ is vertex-transitive, WPD, and non-elementary.
\item The action $G\curvearrowright X$ is action path-connected with constant $E$.
\item The maximal finite normal subgroup of $G$ is trivial.
\item We have 
$$(w_1,...,w_k)\in\mathcal{E}_{10^{20}M,M}$$
as in Definition \ref{def:walkpermissible}.
\end{enumerate}

Given these hypotheses, we set $L=\max_iL(w_i)$ and $\mathcal{M}:=50\min_i L(w_i)/M$. We also let $v_1,...,v_k$ be the respective conjugates of $w_1,...,w_k$ obtained from Definition \ref{def:walkpermissible}, and set the notation \[
\mathcal{A} := \big\{ h Ax(v) : h \in G, v \in \{v_{1}, \ldots, v_{k}\} \big\}.
\]
\end{hypothesis}

Note that by Theorem \ref{thm:randomwalkproperties}, there exists a constant $M$ for which the last condition of Hypothesis \ref{hyp:main} holds with probability approaching $1$ as $n$ goes to $\infty$.

\begin{remark}

We note that several of these hypotheses are more or less free. Specifically, given the other properties and that the action $G\curvearrowright X$ is cobounded, we may assume without loss of generality that $G\curvearrowright X$ is vertex-transitive, that $X$ is a simplicial graph with the simplicial metric, and that the maximal finite normal subgroup of $G$ is trivial, perhaps after adjusting constants. Crucially, these modifications do not affect the conclusions of our theorems. The first two reductions are standard, while the third is explained in Appendix A of \cite{ABMNR26}.

\end{remark}

\section{Tight conjugates and minimal representatives}

Let $w_{1}, \ldots, w_{k} \in G$ be elements satisfying Hypothesis \ref{hyp:main}. We will denote \[
L_{i} = L(w_{i}) := \lim_{n} \frac{1}{n} d(x_{0}, w_{i}^{n} x_{0}).
\]

Recall the definition
$$\mathcal{M}:=\min_i L(w_i)/M$$

In the following, we will consider the normal closure of $w_{i}$'s; in this discussion, we may replace each $w_{i}$ with any of its conjugates. Hence, without loss of generality, we will assume that the $(w_{i}^{n} x_{0})_{n \in \Z}$ are $10\delta$-aligned. What we want is the following.

\begin{lemma}\label{lem:CM2.1}
Let $h$ be a nontrivial element of $N:= \llangle w_{1}, \ldots, w_{k} \rrangle$. Then there exist $g \in G$, $i \in \{1, \ldots, k\}$ and a conjugate $v$ of $w_{i}$ such that $(x_{0}, gx_{0}, g v x_{0}, hx_{0})$ is $\mathcal{M}$-aligned. In particular, for every $x_1,x_2\in X$ with $hx_1=x_2$ for $h\in N\setminus\{e\}$, there exists an axis $\gamma \in \mathcal{A}$ such that $d_{\gamma}(x_1,x_2)\geq L(v)-3\M$.
\end{lemma}

This lemma follows from Lemma 2.1(1) of Clay-Mangahas, but we prove it directly for completeness. Before the proof, let us introduce some terminology.


For $w \in  \{w_{1}^{\pm 1}, \ldots, w_{k}^{\pm 1}\}$, a conjugate $n=g wg^{-1}$ of $w$ is called \emph{tight} if $d_{h \Ax (v) }(x_{0}, nx_{0}) < 0.6 L(v)$ for all $v \in \{w_{1}, \ldots, w_{k}\}$ and $h \in G$, except when $h\Ax(v) = g \Ax(w)$.

\begin{lemma}\label{lem:tightProduct}Every conjugate $n$ of $w \in \{w_{1}, \ldots, w_{k}\}$ is  a finite product of tight conjugates.
\end{lemma}

\begin{proof}
We will prove the lemma by induction on $\|\cdot\|_{X}$, where $\|g\|_{X}:=d_X(x_0,gx_0)$. Let $n$ be a conjugate of an element of $\{w_{1}, \ldots, w_{k}\}$, and let $n = gwg^{-1}$ be an effective product decomposition as defined above. Let $h \in G, v \in \{w_{1}^{\pm 1}, \ldots, w_{k}^{\pm 1}\}$ be an arbitrary pair such that  $d_{h \Ax(v)} (x_{0}, nx_{0}) > 0.6L(v)$. We can further assume that $\pi_{h\Ax(v)}(x_{0})$ appears earlier than $\pi_{h\Ax(v)}(nx_{0})$ along $h\Ax(v)$; we can invert the sign of $v$ if not. If $h\Ax(v) = g\Ax(w)$ for all such pairs $(h, v)$, then $n$ is tight. So let us assume $h \Ax(v) \neq g \Ax(w)$. Let \[
D_{1} = d_X\big(x_{0}, h\Ax(v)\big), \,\,D_{2} = d_X\big(h\Ax(v), nx_{0}\big).
\]
The projection of $[x_{0}, nx_{0}]$ onto $h\Ax(v)$ has diameter at least $0.6L(v)$, so we can apply Lemma \ref{lem:BGIrouting} and Lemma \ref{lem:morse}. We obtain a subsegment $\gamma$ of $[x_{0}, nx_{0}]$ that lies in the $10\delta$-neighborhood of $h\Ax(v)$ such that $d(x_{0}, \gamma) =_{10\delta} D_{1}$ and $d(nx_{0}, \gamma) =_{10\delta} D_{2}$. Note that either $D_{1} \le D_{2}$ or $D_{1} \ge D_{2}$ holds. Let us focus on the former case, as the latter can be handled in a symmetric way.

Again by applying Lemma \ref{lem:BGIrouting} and Lemma \ref{lem:morse}, we obtain a subsegment $\gamma'$ of $[x_{0}, nx_{0}]$ lying in the $10\delta$-neighborhood of $g \Ax(w)$. As $v, w \in \mathcal{E}_{10^{20}M, M}$ and $h\Ax(v) \neq g \Ax(w)$, we know that $\gamma$ and $\gamma'$ cannot overlap for length $>\mathcal{M}$. This fact plus $D_{1} \le D_{2}$ implies that \[
D_{1} <_{30\delta} \frac{1}{2} \big(\|n\|_{X} - L(w)\big)-L(v) + \mathcal{M}.
\]

In this case, the vast majority of $\gamma$ precedes $\gamma'$ along $[x_0,nx_0]$, and so we have \[
\diam( \gamma \cap N_{50\delta}([x_{0}, \pi_{g \Ax (w)}(x_{0})])) \ge_{100\delta} l(\gamma) - \mathcal{M} \ge 0.59 L(v).
\]
In particular, $d_{h \Ax(v)}(x_{0}, \pi_{g\Ax(w)}(x_{0})) \ge 0.58L(v)$. As $n$ stabilizes $g\Ax(w)$, this also implies that \[
d_{nh \Ax(v)}(nx_{0}, \pi_{g \Ax(w)}(nx_{0})) = d_{h \Ax(v)}(x_{0}, \pi_{g\Ax(w)}(x_{0})) \ge 0.58L(v).
\]
Now, let us consider the decomposition \[
n = g w g^{-1} =( h v h^{-1}) \cdot (h v^{-1} h^{-1}g w g^{-1} hv h^{-1}) \cdot (h v^{-1} h^{-1}).
\]
This decomposition is depicted in Figure \ref{fig:tight}. First, we have \[
\|h v h^{-1}\| \le 2 D_{1} + L(v) < \|n\|_{X}.
\]
Second, we have \[
\begin{aligned}
\|h v^{-1}h^{-1}nh v h^{-1}\|_{X} &\le_{100\delta} d_X(h v^{-1} h^{-1} x_{0}, h \Ax(v)) + d_{h\Ax(v)}(h v^{-1} h^{-1} x_{0}, \pi_{g \Ax(w)}(x_{0})) \\
&+ d_X(\pi_{g \Ax(w)}(x_{0}), \pi_{g \Ax(w)}(nx_{0})) + d_{nh\Ax(v)}(\pi_{g \Ax(w)}(nx_{0}), n hv^{-1} h^{-1} x_{0})\\&+ d_X(nh\Ax(v), n hv^{-1} h^{-1} x_{0}) \\
&\le d_X(x_{0}, h \Ax(v)) + |d_{h\Ax(v)}(x_{0}, \pi_{g \Ax(w)}(x_{0})) - L(v)| + d_X(\pi_{g \Ax(w)}(x_{0}), \pi_{g \Ax(w)}(nx_{0})) \\
&+ | d_{nh\Ax(v)}(\pi_{g \Ax(w)}(nx_{0}), n  x_{0}) - L(v)| + d_X(nh\Ax(v), n  x_{0})\\
&\le_{100\delta} d(x_{0},nx_{0}) - 0.16L(v).
\end{aligned}
\]
So the three conjugates are of smaller $\|\cdot\|_{X}$-norm than $n$. Induction finishes the proof.
\end{proof}

\begin{figure}[htbp]
    \centering 
    \includegraphics[width=0.8\textwidth]{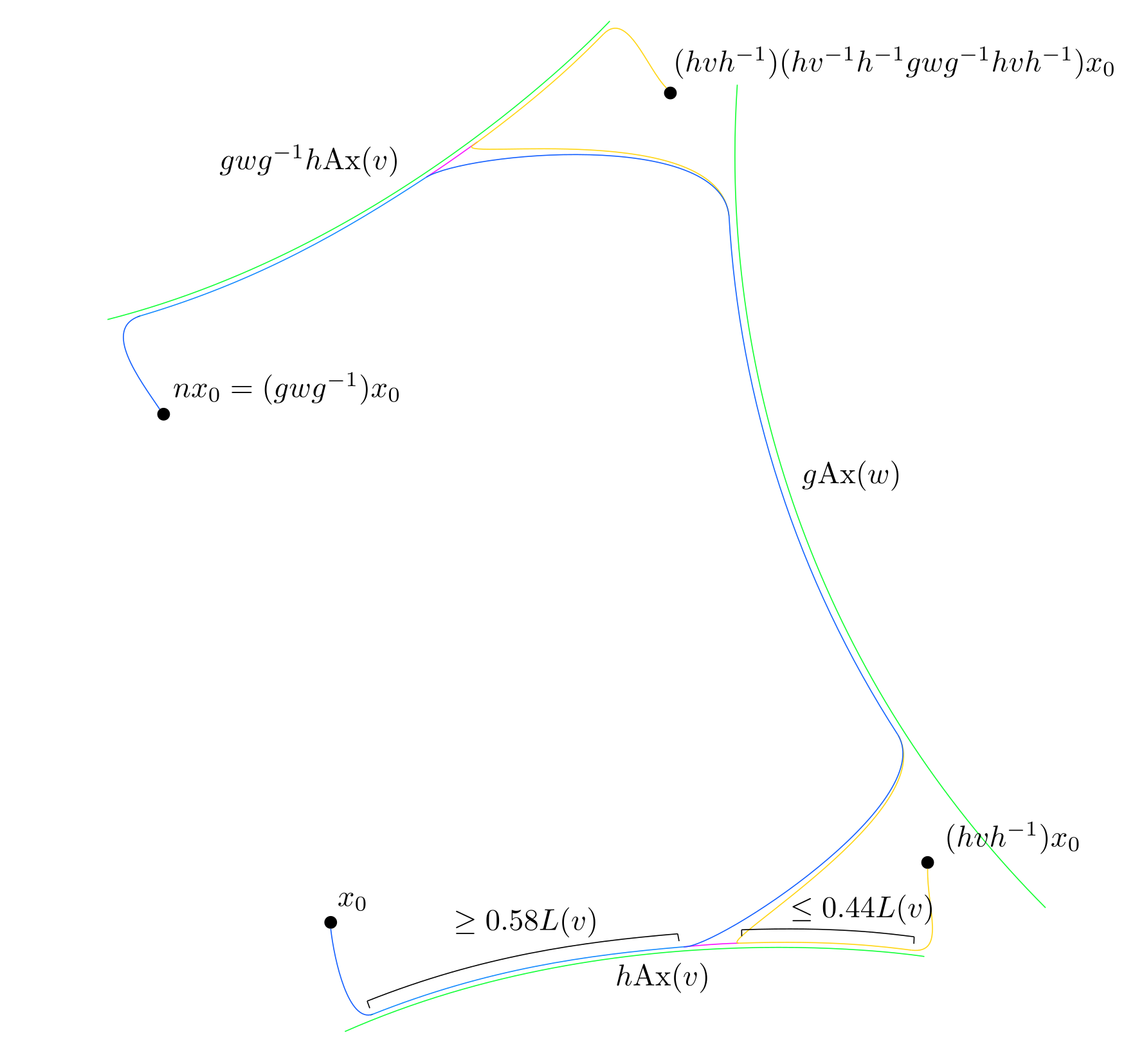} 
    \caption{The inductive tight conjugate decomposition used in the proof of Lemma \ref{lem:tightProduct}. Explicitly, $n=gwg^{-1}$ is decomposed as $(hvh^{-1})\cdot(hv^{-1}h^{-1}gwg^{-1}hvh^{-1})\cdot(hv^{-1}h^{-1})$. Due to the projection distance assumption of $d_{h\Ax(v)}(x_0,\pi_{g\Ax(w)}(x_0))\geq 0.58L(v)$, the three geodesics $[x_0,(hvh^{-1})x_0]$, $[(hvh^{-1})x_0,(hvh^{-1})(hv^{-1}h^{-1}gwg^{-1}hvh^{-1})x_0]$, and $[(hvh^{-1})(hv^{-1}h^{-1}gwg^{-1}hvh^{-1})x_0,nx_0]$ each have shorter length than $[x_0,nx_0]$.}
    \label{fig:tight}
\end{figure}

\begin{lemma}\label{lem:axisDecomp} Let $n \in N$ and let $m$ be the minimum number of tight conjugates needed to express $n$ as a product of tight conjugates.
Then there exist $m$ tight conjugates \[
h_{1} u_{1} h_{1}^{-1}, \ldots, h_{m} u_{m} h_{m}^{-1} \quad (u_{i} \in \{w_{1}^{\pm}, \ldots, w_{k}^{\pm}\}, h_{i} \in G)
\]
whose product equals $n$, and a nonempty sequence $1=l(1) < \ldots < l(T)$ such that for 
\[\begin{aligned}
n_{j} &:= (h_{1} u_{1} h_{1}^{-1}) \cdots (h_{j} u_{j} h_{j}^{-1}) \quad(j=1, \ldots, m),\\
\gamma_{i} &:= n_{l(i)-1} h_{l(i)} \Ax(u_{l(i)}) \quad(i=1, \ldots, T),\\
\end{aligned}
\]
the following holds.
\begin{enumerate}
\item For points \[
p_{i} = \pi_{\gamma_{i}} (n_{l(i)-1} x_{0}),\quad q_{i} = \pi_{\gamma_{i}} (n_{l(i)} x_{0}) \quad (i=1, \ldots, T),
\]we have that $q_{i}$ is $0.5\mathcal{M}$-close to $[n_{l(i+1) - 1} x_{0}, p_{i+1}]$ for each $i=1, \ldots, T-1$.
\item We have $d_{X}(p_{i}, q_{i}) \ge L(u_{l(i)})$, $(q_{i-1} \cdot q_{i})_{p_{i}} < 10\delta$ for $i=1, \ldots, T$ (with $q_{0} := x_{0}$) and $(p_{i} \cdot p_{i+1})_{q_{i}} \le 0.7L(u_{l(i)})$ for $i = 1, \ldots, T-1$. Moreover, $(p_{T} \cdot nx_{0})_{q_{T}} < \mathcal{M}$.
\end{enumerate}
\end{lemma}

\begin{proof}
Let us now prove the lemma by induction. For $m=1$ we can take the sequence $\{1 = l(1)\}$. Now, suppose that a product $n$ of $m$ tight conjugates satisfies the condition, with conjugates \[
h_{1} u_{1} h_{1}^{-1}, \ldots, h_{m} u_{m} h_{m}^{-1}
\]
and an auxiliary sequence $1=l(1) < \ldots < l(T)$ as in the lemma statement. Now let us consider $n'=h u h^{-1} \cdot n$ for some tight conjugate $huh^{-1}$. Let $p_{0} := \pi_{ h Ax(u)} (hu^{-1} h^{-1} x_{0})$ and $q_{0} := \pi_{hAx(u)} (x_{0})$.

\textbf{Case I.} Some part of $[q_{0}, p_{0}]$ and $[p_{i}, q_{i}]$ are $30\delta$-fellow traveling for length $>0.1\mathcal{M}$ for some $i$. Then by the no-overlap condition of Definition \ref{def:walkpermissible}, we have  $h\Ax(u^{-1}) = n_{l(i)-1} h_{l(i)}\Ax(u_{l(i)})$, with the same orientation. Moreover, we have \[
hu^{-1}h^{-1} = n_{l(i)-1} h_{l(i)} \cdot u_{l(i)} \cdot (n_{l(i)-1} h_{l(i)})^{-1}.
\]
This forces that  \[
n' = n_{l(i)-1} \cdot (h_{l(i)+1} u_{l(i)+1} h_{l(i)+1}^{-1} ) \cdots (h_{m} u_{m} h_{m}^{-1}),
\]
and $n'$ becomes a product of at most $m-1$ tight conjugates. Such an element satisfies the conclusion of the lemma by induction.

\textbf{Case II.} The part of $[q_{0}, p_{0}]$ that is $20\delta$-fellow traveling with a subsegment of $[x_{0}, nx_{0}]$ has length $< 0.8\mathcal{M}$. This forces that $\pi_{[q_{0}, p_{0}]}(nx_{0})$ and $\pi_{[q_{0}, p_{0}]}(x_{0}) \subseteq \mathcal{N}_{10\delta}(q_{0})$ must have distance at most $0.9\mathcal{M}$, and we get $(p_{0} \cdot nx_{0})_{q_{0}} < \mathcal{M}$. Then $n'$ satisfies the conclusion of the lemma with $\{l'(1), l'(2),\ldots\} = \{1\}$. 

\textbf{Case III.} A part of $[q_{0}, p_{0}]$ is $20\delta$-fellow traveling with a subsegmemt of $[q_{T}, x_{0}]$. Then one can see that $[x_{0}, q_{0}]$ contains a subsegment that is $30\delta$-fellow traveling with $[p_{T}, q_{T}]$, of length $> L(u_{t}) - 2.9 \mathcal{M}$.
Namely, we have an axis $\gamma_{T}$ that is closer to $x_{0}$ than $h Ax(u)$ is, whereas $d_{\gamma_T}(x_0,huh^{-1}x_0)>0.6L(u_T)$. This contradicts the fact that $h u h^{-1}$ was a tight conjugate.

\textbf{Case IV.} We are not in Case I, but some part of $[q_{0}, p_{0}]$ is $20\delta$-fellow traveling with a subsegmemt of $[ \pi_{[p_{i}, q_{i}]}(p_{i+1}), p_{i+1}]$ for some $i$, for length $>0.6 \mathcal{M}$. 

In this case, we claim that $(p_{0} \cdot p_{i+1})_{q_0}<0.7L(u)$. This will conclude that $n'$ has the desired decomposition with indices $0 < l(i) < \ldots < l(T)$.

To see this, remember that $(x_{0} \cdot q_{i})_{p_{i}} < 20\delta$ and $(p_{i} \cdot p_{i+1})_{q_{i}} \leq 0.7 L(u_{l(i)})$. In particular, there are points $P, Q \in [x_{0}, p_{i+1}]$ with $P$ coming before $Q$ such that $P \in \mathcal{N}_{30\delta}(p_{i})$, $Q \in \mathcal{N}_{20\delta}( \pi_{[p_{i}, q_{i}]}(p_{i+1}))$ and such that $d(P, Q) > 0.2L(u_{i})$.

Now, let $[q', p']\subseteq [q_{0}, p_{0}]$ be the maximal subsegment that is $20\delta$-fellow traveling with a subsegment of  $[ \pi_{[p_{i}, q_{i}]}(p_{i+1}), p_{i+1}] \subseteq \mathcal{N}_{50\delta}([n_{l(i+1)-1} x_{0}, n_{l(i+1)} x_{0}])$. Then we have $d(q', p') < 0.6L(u)$ by the tightness of $h_{l(i)} u_{l(i)} h_{l(i)}^{-1}$. Also, note that $[q_0, q']$ (however short it is) will be $50\delta$-fellow traveling with part of $[p_{i}, \pi_{[p_{i}, q_{i}]}(p_{i+1})]$. Such fellow traveling should be shorter than $0.1\mathcal{M}$ as Case I is excluded. In short, we have $d(q_{0}, p') < 0.65L(u)$. It follows that $(p_{0} \cdot p_{i+1})_{q_{0}} < 0.7L(u)$ as desired. Also, note that $d(q_{0}, [n_{l(i+1)-1} x_{0}, p_{i+1}]) < d(q_{0}, q') + 50\delta < 0.5\mathcal{M}$. 
\end{proof}

\textbf{Proof of Lemma \ref{lem:CM2.1}.} Given $n \in N$ with the decomposition described in Lemma \ref{lem:axisDecomp}, one has that $(x_{0}, p_{T}, q_{T}, nx_{0})$ is an $\mathcal{M}$-aligned sequence with $p_{T}, q_{T} \in n_{l(T)-1} h_{l(T)} Ax(u_{T})$ and $d(p_{T}, q_{T}) = L(u_{T})$. $\square$

\

\begin{figure}[htbp]
    \centering 
    \includegraphics[width=1\textwidth]{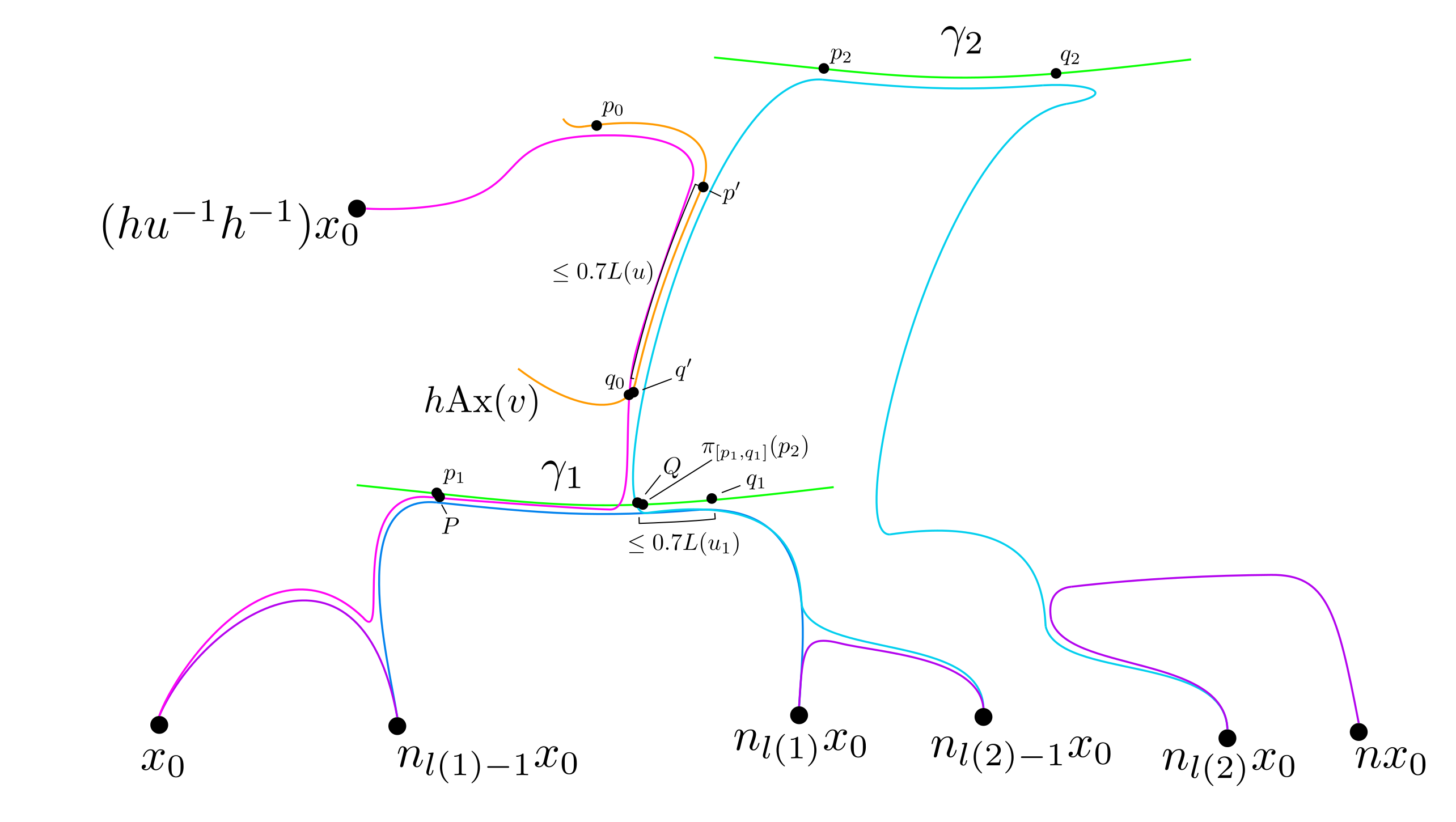} 
    \caption{The decomposition defined in Lemma \ref{lem:axisDecomp}. In the case depicted, $T=2$, and we have extended the construction by defining $n'=(huh^{-1})n$. As $[q_0,p_0]$ fellow-travels a segment of $[\pi_{[p_1,q_1]}(p_2),p_2]$, we are in Case IV. We therefore adjust the auxiliary sequence to $0=l(0)<l(2)$. In particular, we have that $\pi_{\gamma_2}((hu^{-1}h^{-1})x_0,nx_0)\geq L(u_2)-3\M$.}
    \label{fig:angle}
\end{figure}

As projections between axes have $\M$-bounded diameters by the non-overlapping condition and $\delta$-hyperbolicity, Lemma \ref{lem:CM2.1} also holds for axes, albeit with slightly worse constants.

\begin{corollary}
\label{cor:CM2.1strutedition}

Let $h_1\Ax(v_1)$ and $h_2\Ax(v_2)$ be distinct axes such that $\pi_{X/N}(h_1\Ax(v_1))=\pi_{X/N}(h_2\Ax(v_2))$. Then there exists an axis $h\Ax(v)$ such that $d_{h\Ax(v)}(h_1\Ax(v_1),h_2\Ax(v_2))\geq L(v)-10\M$.

\end{corollary}

Lemma \ref{lem:CM2.1} is useful to show that $\bar{X} = X/N$ is hyperbolic, with hyperbolicity constant at most $\max_i(L_i)+1000\delta$. 

\begin{lemma}
\label{lem:quotienthyp}

The space $X/N$ is Gromov hyperbolic, with hyperbolicity constant at most $\max_i(L_i)+1000\delta$.

\end{lemma}

\begin{proof}
By Papasoglu's theorem \cite[Theorem 1.4]{papa95}, it suffices to prove that all geodesic bigons are $\delta$-coarsely $\max_{i}(L_{i})$-thin. In fact, we will prove that: \begin{quote}
    Let $\bar{x}, \bar{y}, \bar{z} \in X/N$ be such that $d_{X/N}(\bar{y}, \bar{z}) < 100\delta$. Then $[\bar{x}, \bar{y}]$ and $[\bar{x}, \bar{z}]$ are within Hausdorff distance $\le \max_{i} L_{i} + 1000\delta$.
\end{quote}
To show this, let us lift $[\bar{x}, \bar{y}]$, $[\bar{y}, \bar{z}]$ and $[\bar{z}, \bar{x}]$ to $X$. More precisely, denoting the projection from $X$ to $X/N$ by $\pi_N$, let $[\bar{x}, \bar{y}] = \pi_N(x=x_{0}, x_{1}, \ldots, x_{n})$ for some $X$-path $(x_{0}, \ldots, x_{n})$. Similarly, pick $(x_{n}, \ldots, x_{n'})$ and $(x_{n'}, \ldots, x_{n''})$ that respectively project to $[\bar{y}, \bar{z}]$ and $[\bar{z}, \bar{x}]$. Then $x_{0}$ and $x_{n''}$ are identified under projection, so there exists $n \in N$ such that $nx_{0} = x_{n''}$. Here, because  $[\bar{x}, \bar{y}]$ is a geodesic, $(x_{0}, \ldots, x_{n})$ is also a geodesic from $x_{0}$ to $x_{n}$. Similarly, $(x_{n}, \ldots, x_{n}'), (x_{n'}, \ldots, x_{n''})$ are also geodesics.

Now, by Lemma \ref{lem:CM2.1}, there exists $g \in G$ and a conjugate $v$ of $\{w_{1}^{\pm1}, \ldots, w_{k}^{\pm 1}\}$ such that $(x_{0}, gx_{0}, gvx_{0}, nx_{0})$ is $L(v)/50$-aligned. Let $\gamma \subseteq [x_{0}, nx_{0}]$ be the subsegment of $[x_0,nx_0]$ contained in the $50\delta$-neighborhood of $[gx_{0}, gvx_{0}]$. If $N_{50\delta}(\gamma)$ intersects the geodesic $(x_{0}, \ldots, x_{n})$ for length $>0.6L(v)$, then $d_{X}(gvg^{-1} x_{0}, x_{n})$ becomes shorter than $d_{X}(x_{0}, x_{n})$. This gives a contradiction, as then $ d_{X/N}(\bar{x}, \bar{y})=d_{X}(x_{0}, x_{n})  < d_{X}(gvg^{-1} x_{0}, x_{n}) = d_{X/N}(\bar{x}, \bar{y})$. (Note that $\pi(x_{0}) = \pi(gvg^{-1} x_{0})$.) Likewise, $N_{50\delta}(\gamma)$ cannot intersect the geodesic $(x_{n'}, \ldots, x_{n''})$ for length more than $0.6L(v)$.

It follows that the beginning point of $\gamma$ is $50\delta$-close to a point $x_{i}$ on $(x_{0}, \ldots, x_{n})$ and the ending point of $\gamma$ is $50\delta$-close to a point $x_{j}$ on $(x_{n'}, \ldots, x_{n''})$. Clearly, $\pi(x_{i})$ and $\pi(x_{j}) = \pi(g v^{-1} g^{-1} x_{i})$ are $100\delta$-close. Now, for the quadrangle with vertices $x_{i}, x_{n}, x_{n'}, x_{j}$, we know that $d(x_{n}, x_{n'}) \le 100\delta$ and $d(x_{i}, x_{j}) < L(v) + 100\delta$. For such quadrangles, $[x_{i}, x_{n}]$ and $[x_{n'}, x_{j}]$ are $L(v) + 1000\delta$-close in the Hausdorff distance (even at the level of $X$). Now we can apply our inductive hypothesis to $y' = \pi(x_{i})$ and $z' = \pi(x_{j})$.
\end{proof}

\section{Strut Sequences}

For this section, it will be important to sometimes replace projections with the modified projections defined by Bestvina--Bromberg--Fujiwara--Sisto in \cite{BBFS19}, and measure projection distances with respect to these modified projections. These modified projection distances capture a `stability' property of closest point projections with respect to translation along axes, and thus make certain inductive arguments cleaner.

Recall the following definition:

\begin{definition}
\label{def:axes}

Given $w_1,...,w_k$ as in Hypothesis \ref{hyp:main}, let $v_1,...,v_k$ be as in Definition \ref{def:walkpermissible}. We define the set $\A$ of \emph{axes} by 
$$\A=\big\{h\Ax(v):h\in G,\ v\in\{v_1,...,v_k\}\big\}.$$
Note that $h\Ax(v)=\Ax(hvh^{-1})$.
\end{definition}

Note that distinct axes in $\mathcal{A}$ have mutual projections with diameter $\le 0.02 \mathcal{M}$, by the 5th condition in Hypothesis \ref{hyp:main}. This also gives the Behrstock inequality (given that $L(v_{i})$'s are all long enough so that $\mathcal{M} > 1000\delta$): \[
d_{\gamma}(\gamma', \gamma'') \ge 0.05 \mathcal{M} \Rightarrow d_{\gamma'}(\gamma, \gamma'') \le 0.05 \mathcal{M} \quad (\forall \gamma, \gamma', \gamma'' \in \mathcal{A}).
\]
These facts lead to the following consequence. 

\begin{lemma}[{\cite[Theorem 4.1]{BBFS19}}]
\label{lem:complex}

There exist modified projection functions $\Pi$ and modified distance functions $D$ such that, for any axes $\alpha$ and $\beta$ and $x,y,z\in X\cup\A$:
\begin{enumerate}
\item We have $\Pi_\alpha(x)\subset \N_{\mathcal{M}}(\pi_\alpha(x))$ and $\diam(\Pi_\alpha(x))\leq\mathcal{M}$.
\item We have
$$D_\alpha(x,y)=_{\theta}d_\alpha(x,z)$$
\item We have
$$D_\alpha(x,y)=\diam(\Pi_\alpha(x)\cup\Pi_\alpha(y))$$
\item For any $g\in G$, $\Pi_{g\alpha}(gx)=g\Pi_\alpha(x)$, so as a consequence $D_{\alpha}(gx,gy)=D_\alpha(x,y)$.
\item We have $D_\alpha(x,z)\leq D_\alpha(x,y)+D_\alpha(y,z)$.
\item (Behrstock Inequality) We have
$$\min(D_\alpha(x,\beta),D_\beta(x,\alpha))\leq\mathcal{M}.$$
\item The set $\{\eta\in\A:D_\eta(x,y)\geq\mathcal{M}\}$ is finite.
\item (Strong Behrstock Inequality / `Forcing') If $d_\beta(x,\alpha)>\theta$, then $\Pi_\alpha(x)=\Pi_\alpha(\beta)$. As a consequence,
$$d_\beta(x,\alpha)>\mathcal{M}\text{ implies } D_\alpha(x,w)=D_\alpha(\beta,w)$$

\end{enumerate}

\end{lemma}

This last item is generally called the Strong Behrstock Inequality, although we will use the term `forcing' to describe its conclusion.

\begin{remark}
From here onwards, we increase $\delta$ so that for all $\alpha \in \A$ and $x\in X\cup \A$, $\N_{10\delta}(\pi_\alpha(x))\subset \Pi_\alpha(x)$.

\end{remark}

We would like to be able to lift geodesic triangles from $X/N$ to $X$; unfortunately, this is usually not possible. Instead, we will lift triangles that are \emph{near-geodesic} in the following sense.

\begin{definition}
\label{def:minimal}

Let $x,y\in X\cup\A$, with projections $\bar{x}=\pi_{X/N}(x)$ and $\bar{y}=\pi_{X/N}(y)$. We say that $\bar{x}$ and $\bar{y}$ are \emph{minimal} if $d_X(x,y)=d_{X/N}(\bar{x},\bar{y})$. Furthermore, we say that $x$ and $y$ are $\tau$-near-minimal for $0<\tau<0.99$ if for all axes $h\Ax(v)$, $D_{h\Ax(v)}(x,y)\leq \tau L(v)$. If no $\tau$ is stated, we assume $\tau=0.8$.

\end{definition}

\begin{lemma}
\label{lem:minnearmin}

If $x$ and $y$ are minimal, then for all axes $h\Ax(v)$, $d_{h\Ax(v)}(x,y)\leq 0.51L(v)$, so in particular $x$ and $y$ are near-minimal. Furthermore, if $x$ and $y$ are $\tau$-near-minimal, then for all axes $h\Ax(v)$, $d_{h\Ax(v)}(x,y)\leq (\tau+0.01)L(v)$.

\end{lemma}

\textbf{Proof.} For any axis $h\Ax(v)$, there exists some integer $m$ so that $d_{h\Ax(v)}(x,gy)\leq 0.51L(v)$ for $g=hv^mh^{-1}$. In fact, this upper bound may be sharpened to $0.505L(v)$. If $x$ and $y$ are minimal but $d_{h\Ax(v)}(x,y)>0.51L(v)$ for some axis $h\Ax(v)$, then it is straightforward to check from the Bounded Geodesic Image Theorem that $d_X(x,gy)<d_X(x,y)$, contradicting minimality. The second claim is immediate from Lemma \ref{lem:complex}. $\square$

\begin{lemma}
\label{lem:verifyminimal}

Fix $\bar{x},\bar{y}\in X/N\cup\A/N$, along with respective lifts $x$ and $y$ in $X\cup\A$. Assume that for all axes $h\Ax(v)$, $d_{h\Ax(v)}(x,y)\leq 0.45L(v)$. Then $x$ and $y$ are minimal.

\end{lemma}

\textbf{Proof.} Let $x'$ be a minimal lift of $\bar{x}$ with respect to $y$. If $x=x'$, we're done. Otherwise, by Lemma \ref{lem:CM2.1}, there exists an axis $h\Ax(v)$ such that $d_{h\Ax(v)}(x,x')\geq 0.99L(v)$. We have $d_{h\Ax(v)}(x,y)\leq 0.45L(v)$ by assumption, and $d_{h\Ax(v)}(x',y)\leq 0.51L(v)$ by minimality, so by the triangle inequality $d_{h\Ax(v)}(x,x')\leq 0.96L(v)$. This is a contradiction. $\square$

\begin{figure}[htbp]
    \centering 
    \includegraphics[width=0.9\textwidth]{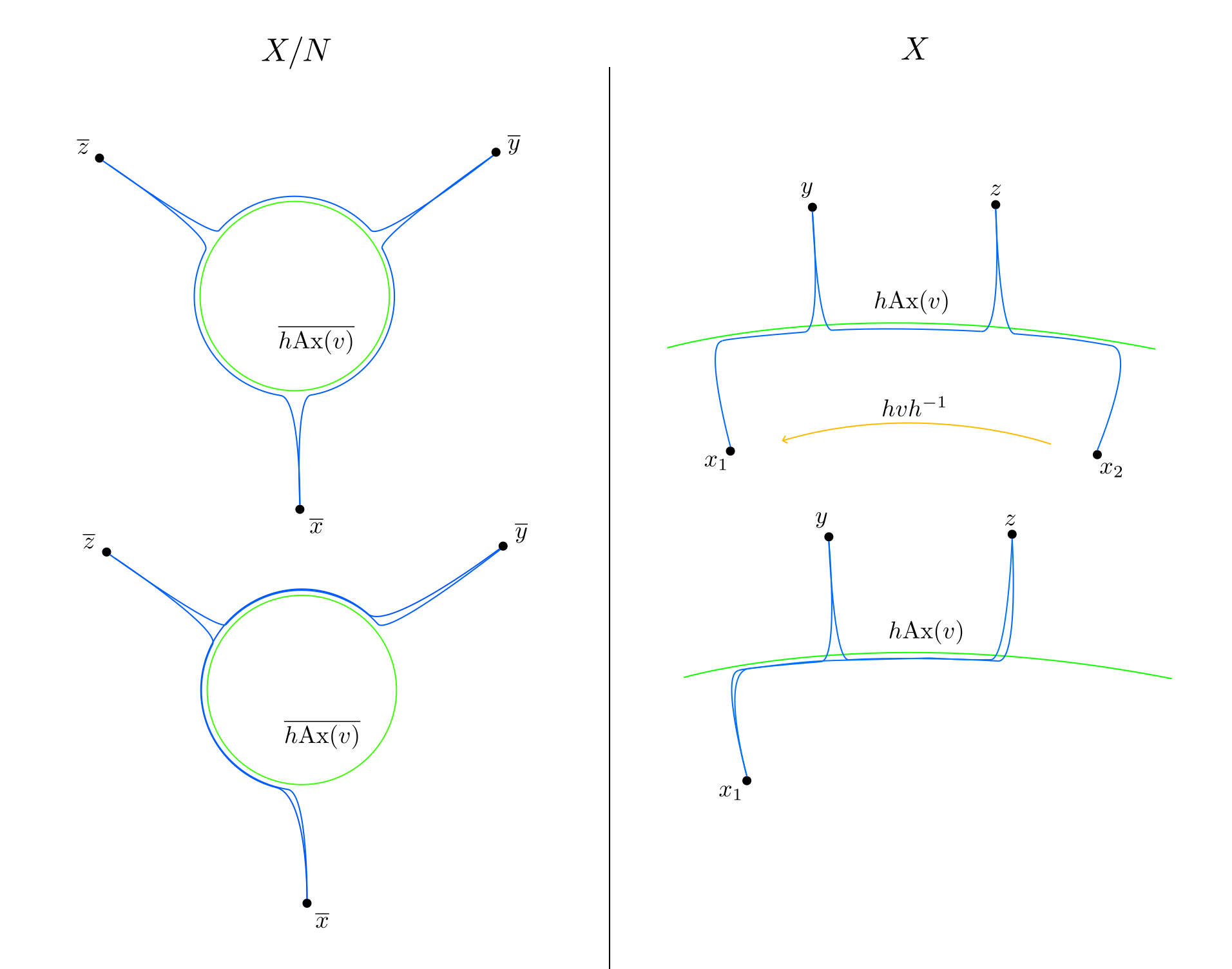} 
    \caption{A representative illustration of Lemma \ref{lem:nearmintriple}. The left pane shows the quotient space $X/N$, while the right pane shows $X$. While lifting the geodesic triangle between $\bar{x}$, $\bar{y}$, and $\bar{z}$ from $X/N$ to $X$ is not possible, we may `close it up' by succesive `rotations' along axes. In the quotient, these rotations correspond to `rerouting' paths around axes.}
    \label{fig:rotating}
\end{figure}

\

We will now define \emph{strut sequences}, which will be a primary technical tool for the rest of the paper. This machinery captures the necessary small-cancellation properties of axes that we need.

\begin{definition}[Strut Sequences]
\label{def:strutsequences}

Let $x,y\in X\cup\A$, and assume $x$ and $y$ are near-minimal. For $0<\theta<1$, the \emph{$\theta$-strut sequence} for $x$ and $y$ is
$$\S_\theta(x,y)=\{h\Ax(v)\in\A:D_{h\Ax(v)}(x,y)\geq \theta L(v)\}\cup\{x,y\}.$$
We define a finite total order $<_\S$ on $\S_\theta(x,y)$ by defining $x$ to be the minimal element and then setting $\alpha<_\S\beta$ if $d_\alpha(x,\beta)>10\M$. If no $\theta$ is stated, then it is implied that $\theta=0.1$.

\end{definition}

For an illustration of strut sequences in $X$ and in $X/N$, see Figure $\ref{fig:strutseq}$. Note that the visual representation in this figure assumes several facts that have not yet been proven.

\begin{figure}[htbp]
    \centering 
    \includegraphics[width=0.9\textwidth]{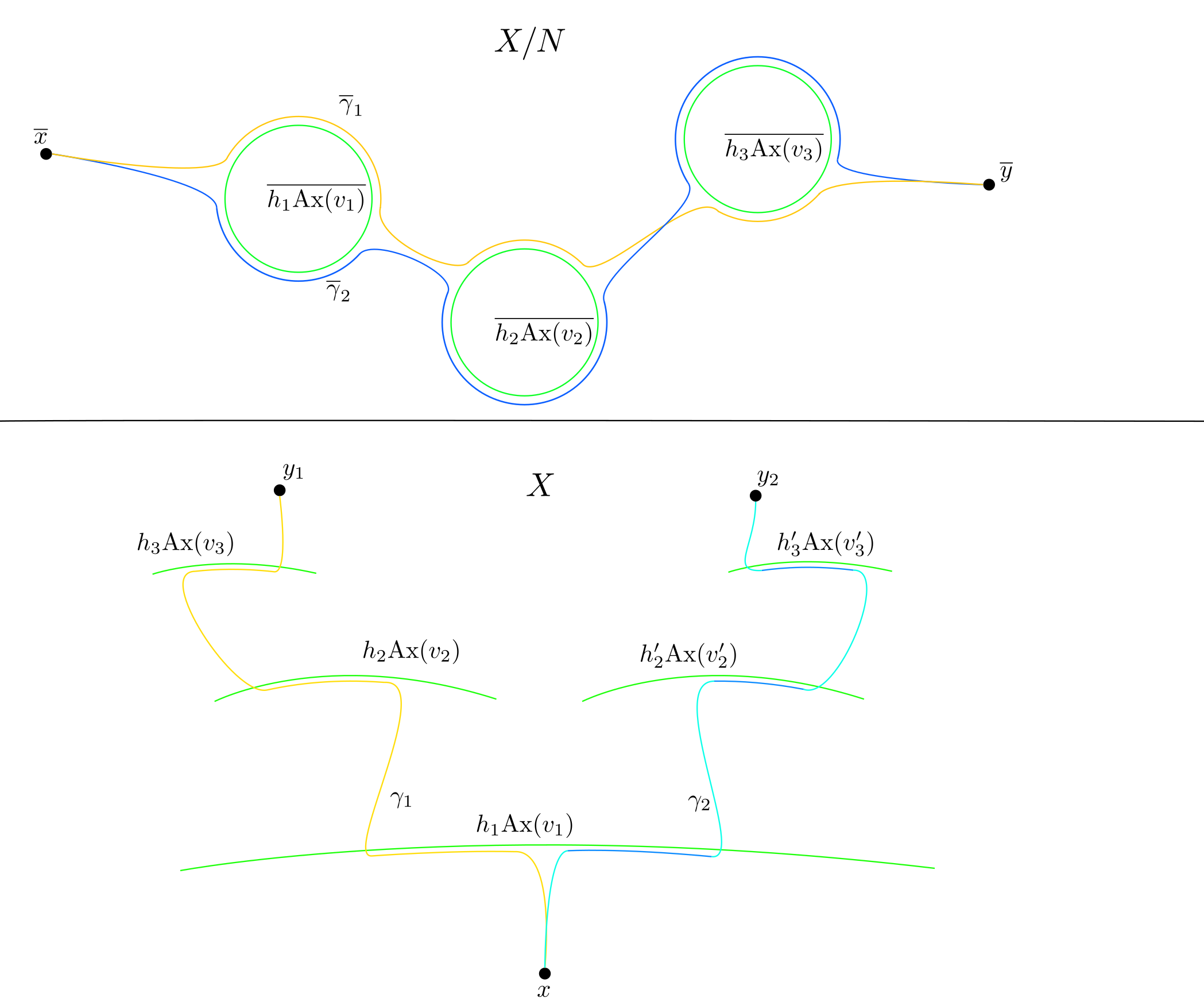} 
    \caption{A depiction of strut sequences in $X$ and under projection to $X/N$, and how these sequences control routing of geodesics between near-minimal pairs. The lower pane shows a point $x\in X$ and two points $y_1,y_2$ in a common orbit, such that $y_1$ and $x$ are minimal and $y_2$ and $x$ are near-minimal. In $X$, the strut sequences $\S(x,y_i)$ represent sequences of axes on which the $[x,y_i]$ make `large turns'. In $X/N$, we see that the paths project to two paths between $\bar{x}$ and $\bar{y}$, which $50\delta$-fellow-travel except when routing around the quotient axes obtained by projecting either strut sequence. A heuristic to keep in mind is that while $X/N$ is \textbf{not} $\delta$-hyperbolic, the failure of $\delta$-hyperbolicity only comes from looping around quotient axes. In particular, the two geodesics $\overline{\gamma}_1$ and $\overline{\gamma}_2$ must approximately $\delta$-fellow-travel between the struts, allowing us to recover some features of $\delta$-hyperbolicity in $X/N$.}
    \label{fig:strutseq}
\end{figure}

We now set up some basic facts about strut sequences.

\begin{lemma}
\label{lem:strutfacts}
We have the following properties for any $x,y\in X\cup\A$ and any $0.01\leq \theta\leq 0.4$. Assume that $\S_\theta(x,y)$ is enumerated as
$$x<_\S h_1\Ax(v_1)<_\S...<_\S h_t\Ax(v_t)<_\S y$$
\begin{enumerate}
\item For any $x,y\in X\cup\A$, the total order $<_\S$ defined in Definition \ref{def:strutsequences} is well-defined, and $\S_\theta(x,y)$ is finite.
\item For any $h_i\Ax(v_i)$ and any $j_1<i<j_2$,
$$\Pi_{h_i\Ax(v_i)}(x)=\Pi_{h_i\Ax(v_i)}(h_{j_1}\Ax(v_{j_1})),\ \pi_{h_i\Ax(v_i)}(h_{j_2}\Ax(v_{j_2}))= \Pi_{h_i\Ax(v_i)}(y)$$

\item We have that $<_\S$-consecutive elements are minimal. If $x$ and $y$ are near-minimal, then any two elements in $\S_\theta(x,y)$ are near-minimal.
\item If $w,z\in\S_\theta(x,y)$ with $w<_\S z$, then $\S_\theta(w,z)$ is equal to the suborder of $\S_\theta(x,y)$ bounded by $w$ and $z$.
\end{enumerate}
\end{lemma}

\textbf{Proof.} Finiteness follows immediately from Lemma \ref{lem:complex}. Now, for any $h\Ax(v)\in \S_\theta(x,y)$, there must be some maximal segment $\gamma_{h\Ax(v)}$ of $\gamma$ such that $\gamma_{h\Ax(v)}$ $10\delta$-fellow-travels $h\Ax(v)$. By the Bounded Geodesic Image Theorem, these segments each have length at least $0.09L(v)$, and by the no-overlap hypothesis for axes (item 3 in Definition \ref{def:walkpermissible}), may only overlap for distance $10\M$. Their order on $\gamma$ gives a finite total order, and items 6 and 8 of Lemma \ref{lem:complex} implies this order is equal to the order $<_\S$.

The second item now follows immediately from the forcing property in Lemma \ref{lem:complex}. The proofs of the third and fourth claim are omitted.  $\square$

\

The following useful lemma follows near-immediately from the definitions and the forwarding property for projections.

\begin{lemma}
\label{lem:paste}

Fix $0.01\leq\theta\leq 0.4$, $0.55\leq\tau\leq 0.95$, $x,y\in X\cup \A$ and $h\Ax(v)\in \A$ such that $\theta\leq D_{h\Ax(v)}(x,y) / L(v) \leq \tau$. Then $\S_\theta(x,y)$ is obtained by concatenating $S_\theta(x,h\Ax(v))$, $\{h\Ax(v)\}$, and $S_\theta(h\Ax(v),y)$ in that order. Furthermore, if $h\Ax(v)$ is near-minimal with respect to both $x$ and $y$, then $x$ and $y$ are near-minimal.

\end{lemma}

The primary reason to study near-minimal representatives is the following lemma. This result is conceptually analogous to Prop. 3.12 in \cite{ABMNR26}. 

\begin{lemma}
\label{lem:nearmintriple}

Let $\bar{x},\bar{y},\bar{z}\in X/N$. Also let $0.78\leq\tau<0.95$. Then there exist respective lifts $x,y,z\in X$ of $\bar{x}$, $\bar{y}$, and $\bar{z}$ such that $y$ and $z$ are minimal, $x$ and $y$ are $\tau$-near-minimal, and $x$ and $z$ are $\tau$-near-minimal.

\end{lemma}

\textbf{Proof.} Choose lifts $x_1$, $y$, $z$, and $x_2$ so that $x_1$ and $y$, $y$ and $z$, and $z$ and $x_2$ are all minimal. We induct on $d_X(x_1,x_2)$, maintaining at each step that $y$ and $x_1$ are $\tau$-near-minimal and that $z$ and $x_2$ are $\tau$-near-minimal. If $d_X(x_1,x_2)$ is ever equal to $0$, we finish immediately.

Otherwise, by Lemma \ref{lem:CM2.1}, there exists an axis $h\Ax(v)$ such that $d_{h\Ax(v)}(x_1,x_2)\geq 0.99L(v)$. We consider the projections of $x_1$, $y$, $z$, and $x_2$ to $h\Ax(v)$. By assumption, we have $d_{h\Ax(v)}(x_1,y)\leq (\tau +0.01)L(v)$, $d_{h\Ax(v)}(z,x_2)\leq (\tau +0.01) L(v)$, and $d_{h\Ax(v)}(y,z)\leq 0.51L(v)$. We may also assume that $v$ is oriented so that the projection order on $h\Ax(v)$ is $x_1\to x_2$.

From the distance bounds, either $d_{h\Ax(v)}(x_1,y)\geq 0.24L(v)$ and we have the projection order $x_1\to y\to x_2$, or $d_{h\Ax(v)}(z,x_2)\geq 0.24L(v)$ and we have the projection order $x_1\to z\to x_2$. Assume without loss of generality that we are in the first case. Then if $g=hvh^{-1}$, by the distance bounds and the translation length of $g$ we have 
$$(0.99-\tau)L(v)\leq d_{h\Ax(v)}(gx_1,y)\leq 0.77L(v).$$
By the inductive hypothesis plus Lemma \ref{lem:strutfacts}, we have that $gx_1$ and $gh\Ax(v)=h\Ax(v)$ are $\tau$-near-minimal, and that $y$ and $h\Ax(v)$ are $\tau$-near-minimal. We may therefore apply Lemma \ref{lem:paste} and conclude $gx_1$ and $y$ are near-minimal. Furthermore, by the Bounded Geodesic Image Theorem, $d_X(gx_1,x_2)<d_X(x_1,x_2)$, so we may apply induction and conclude. $\square$

\

This `rotation' action is depicted in Figure \ref{fig:rotating}, although with $g$ applied to $x_2$. 

We may now prove a crucial result about how struts sequences behave under quotients. See Figure \ref{fig:strutseq} for an illustration of this property.

\begin{lemma}
\label{lem:strutlifts}

Fix $0.55<\tau<0.95$ and $0.01<\theta<0.4$, and assume $\theta+\tau<0.98$. Fix $\bar{x},\bar{y}\in X/N$, and consider lifts $x_1,x_2\in \pi_{X/N}^{-1}(\bar{x})$ and $y_1,y_2\in \pi_{X/N}^{-1}(\bar{y})$. Assume that $x_1$ and $y_1$ are $\tau$-near-minimal, and that $x_2$ and $y_2$ are $\tau$-near-minimal. Enumerate strut sequences
$$\S_\theta(x_1,y_1)=x_1<_\S h_1\Ax(v_1)<_\S...<_\S h_t\Ax(v_t)<_\S y_1$$
$$\S_\theta(x_2,y_2)=x_2<_\S h_1'\Ax(v_1')<_\S...<_\S h_{t'}'\Ax(v_{t'}')<_\S y_2$$
Then $t=t'$ and $\pi_{X/N}(h_i\Ax(v_i))=\pi_{X/N}(h'_i\Ax(v_i'))$ for all $h_i\Ax(v_i)$. 

Furthermore, fix $i<j$, and consider $h_i\Ax(v_i),h_{j}\Ax(v_{j}),h_i'\Ax(v_i'),h_{j}'\Ax(v_{j}')$ with the assumption that for all $i<k<j$,
$$D_{h_k\Ax(v_k)}(h_i\Ax(v_i),h_j\Ax(v_j))\leq 0.48L(v_k).$$
Then there is a single $h\in N$ such that $$h\pi_{h_i\Ax(v_i)}(h_{i+1}\Ax(v_{i+1}))=\pi_{h_i'\Ax(v_i')}(h_{i+1}'\Ax(v_{i+1}'))$$
and
$$h\pi_{h_{i+1}\Ax(v_{i+1})}(h_{i}\Ax(v_{i}))=\pi_{h_{i+1}'\Ax(v_{i+1}')}(h_{i}'\Ax(v_{i}'))$$

\end{lemma}

\textbf{Proof.} We consider the first claim, assuming without loss of generality that $y_1=y_2=:y$. The strategy is induction on $\max(d_X(x_1,y_1),d_X(x_2,y_2))$; if both are less than, say, $0.01\min_vL(v)$, then by Lemma \ref{lem:CM2.1} and the triangle inequality we must have $x_1=x_2$, and are done. If $x_1\neq x_2$, by Lemma \ref{lem:CM2.1} there exists an axis $h\Ax(v)$ such that $d_{h\Ax(v)}(x_1,x_2)\geq 0.99L(v)$. By assumption, $d_{h\Ax(v)}(x_i,y)\leq \tau$ for $i\in\{1,2\}$, so by the triangle inequality $d_{h\Ax(v)}(x_i,y)\geq 0.98-\tau$. Therefore $h\Ax(v)\in \S_\theta(x_i,y)$. Furthermore, certainly $d_X(x_i,h\Ax(v))<d_X(x_i,y)$, so we have the claim for $\S_\theta(x_i,h\Ax(v))$. Applying Lemma \ref{lem:paste} with $h\Ax(v)$, $\S_\theta(x_i,h\Ax(v))$, and $\S_\theta(h\Ax(v),y)$ finishes the argument. 

For the second claim, we may assume without loss of generality that $h_i\Ax(v_i)=h_i'\Ax(v_i')$. Certainly there exists some $n=h_iv_i^mh_i^{-1}$ such that $d_{h_i\Ax(v_i)}(nh_j\Ax(v_j), h'_j\Ax(v_j'))\leq 0.52L(v_i)$. The assumed bounds and the triangle inequality now imply that for all $h\Ax(v)$, $d_{h\Ax(v)}(nh_j\Ax(v_j),h'_j\Ax(v_j'))\leq 0.98L(v)$, and so by Lemma \ref{lem:CM2.1} we have $nh_j\Ax(v_j)=h_j'\Ax(v_j')$. The claim follows. $\square$

\section{Routing Geodesics in the Quotient}

We now establish consequences of the results we have obtained so far. The following claim is illustrated in Figure \ref{fig:strutseq}.

\begin{lemma}
\label{lem:nearmindistbound}

Fix $0.55<\tau<0.95$ and $0.01<\theta<0.4$, and assume $\theta+\tau<0.98$. Fix $\bar{x},\bar{y}\in X/N$, and consider lifts $x_1,x_2\in \pi_{X/N}^{-1}(\bar{x})$ and $y_1,y_2\in \pi_{X/N}^{-1}(\bar{y})$. Assume that $x_1$ and $y_1$ are $\tau$-near-minimal, and that $x_2$ and $y_2$ are $\tau$-near-minimal, and choose geodesics $[x_1,y_1]$ and $[x_2,y_2]$. Then the Hausdorff distance of $\pi_{X/N}([x_1,y_1])$ and $\pi_{X/N}([x_2,y_2])$ is at most $\max_vL(v)+1000\M$. Furthermore, 
$$d(x_1,y_1)\leq 1.01\frac{\tau}{\theta}d(x_2,y_2)$$

\end{lemma}

\textbf{Proof.} Enumerate struts as follows:
$$\S_\theta(x_1,y_1)=x_1<_\S h_1\Ax(v_1)<_\S...<_\S h_t\Ax(v_t)<_\S y_1$$
$$\S_\theta(x_2,y_2)=x_2<_\S h_1'\Ax(v_1')<_\S...<_\S h_{t'}'\Ax(v_{t}')<_\S y$$
By the Bounded Geodesic Image Theorem, the geodesic $[x_1,y_1]$ may be decomposed into segments $\gamma_i$ that each $10\delta$-fellow-travel $h_i\Ax(v_i)$ and interstitial segments $\eta_i$ that each originate in $\N_{10\delta}(\pi_{h_i\Ax(v_i)}(h_{i+1}\Ax(v_{i+1})))$ and terminate in $\N_{10\delta}(\pi_{h_{i+1}\Ax(v_{i+1})}(h_i\Ax(v_i)))$. We construct analogous segments $\gamma_i'$ and $\eta_i'$ for $[x_2,y_2]$. Note that this notation breaks slightly at the endpoints; these cases follow from essentially the same argument.

We now apply Lemma \ref{lem:strutlifts}. Under projection to $X/N$, any $\gamma_i$ and $\gamma_i'$ both lie in $\N_{10\delta}(\overline{h_i\Ax(v_i)})$; this set has diameter at most $L(v_i)+100\delta$. Furthermore, both $\gamma_i$ and $\gamma_i'$ have length at least $\theta-100\M$ and at most $\tau+100\M$.

From the projection complex bounds and the second claim of Lemma \ref{lem:strutlifts}, we obtain that any $\eta_i$ and $\eta_i'$ $10\M$-fellow-travel, and have lengths that differ by no more than $10\M$.

We may now assemble these results to obtain the claims. If in fact there are no struts between $x_1$ and $y_1$, then $[x_1,y_1]$ and $[x_2,y_2]$ in fact $10\delta$-fellow-travel after projection to $X/N$. $\square$

\

For our main construction, we will need to build our own strut sequences, and obtain fairly precise control over those sequences. We do so as follows.

\begin{lemma}
\label{lem:enemyprongs}

Fix $\bar{x},\bar{y}\in X/N$. Given any fixed lift $x\in\pi_N^{-1}(\bar{x})$, there exist two (possibly non-distinct) points $p_1=p_1(x,\bar{y})$, $p_2=p_2(x,\bar{y})$ in $X$ with the following properties:
\begin{enumerate}
\item The distances $d_X(x,p_1)$ and $d_X(x,p_2)$ are bounded above by $\frac{\M}{10}+1$ and bounded below by $\min(\frac{\M}{10}-1,d_X(x,y))$.
\item For any $y\in \pi_N^{-1}(\bar{y})$ such that $x$ and $y$ are near-minimal, either $(x\cdot y)_{p_1}\leq 100\delta$ or $(x\cdot y)_{p_2}\leq 100\delta$
\end{enumerate}

\end{lemma}

\textbf{Proof.} Choose $y\in \pi_{X/N}^{-1}(\bar{y})$ so that $x$ and $y$ are minimal. Enumerate the strut sequence for $x$ and $y$ as
$$\S(x,y)=x<_\S h_1\Ax(v_1)<_\S...<_\S h_t\Ax(v_t)<_\S y$$
If there are no struts between $x$ and $y$ (i.e. $t=0$), then we choose $p_1=p_2$ to be a point on $[x,y]$ satisfying the distance bounds. Otherwise, we take a geodesic $[x,q]$ to a point $q$ on $h_1\Ax(v_1)$ realizing $d_X(x,h_1\Ax(v_1))$, and then `branch' this geodesic along $h_1\Ax(v_1)$ in both possible directions for distance $0.01L(v_1)$ to obtain points $q_1$ and $q_2$. We then choose $p_1$ and $p_2$ to be points along $[x,q_1]$ and $[x,q_2]$ satisfying the appropriate length conditions.

For any $y'\in \pi_{X/N}^{-1}(\bar{y})$ such that $x$ and $y'$ are near-minimal, then the last claim of Lemma \ref{lem:strutfacts} implies that both $[x,y]$ and $[x,y']$ have $h_1\Ax(v_1)$ as a initial strut. So $[x,y']$ must pass within $10\delta$ of $\pi_{h_1\Ax(v_1)}(x)$ and then follow $h_1\Ax(v_1)$ for at least distance $0.09L(v_1)$ in one of the two possible directions, and the claim follows. $\square$

\

\begin{definition}

Given $\bar{x}$, $\bar{y}$, and $x\in\pi_N^{-1}(\bar{x})$ as in Lemma \ref{lem:enemyprongs}, we call the $p_1=p_1(x,\bar{y})$ and $p_2=p_2(x,\bar{y})$ constructed in this lemma \emph{enemy prongs} for $x$ relative to $\bar{y}$.

\end{definition}

\begin{corollary}
\label{cor:enemyprongfriends}

Fix $\bar{x},\bar{y}\in X/N$, choose a fixed lift $x\in \pi_N^{-1}(\bar{x})$, and construct enemy prongs $p_1$ and $p_2$ as in Lemma \ref{lem:enemyprongs}. Assume that $d_{X/N}(\bar{x},\bar{y})\geq 1000L$, and let $\bar{z}\in X/N$ be a point satisfying $d_{X/N}(\bar{y},\bar{z})\leq 200L$. Let $y$ and $z$ be respective lifts of $\bar{y}$ and $\bar{z}$ such that $x$ and $y$ are near-minimal and $y$ and $z$ are near-minimal. Then either $(x\cdot z)_{p_1}\leq 110\delta$ or $(x\cdot z)_{p_2}\leq 110\delta$.

\end{corollary}

\textbf{Proof.} By the assumed distance bounds and Lemma \ref{lem:nearmindistbound}, $d_X(x,y)\geq 1000L$ while $d_X(y,z)\leq 200L$. Furthermore, by Lemma \ref{lem:enemyprongs}, either $(x\cdot y)_{p_1}\leq 200\delta$ or $(x\cdot y)_{p_2}\leq 1000\delta$, and $d_X(x,p_1)$, $d_X(x,p_2)$ are bounded above by, say, $L$. The claim now follows from a standard $\delta$-thinness argument applied to a triangle with vertices $x$, $y$, and $z$. $\square$

\

We apply the enemy prongs construction as follows. See Figure \ref{fig:repel} for an illustration of the following lemma.

\begin{lemma}
\label{lem:quotientrouting}
Let $B$ be a large enough constant determined only in terms of the action $G \curvearrowright X$. Let $n$ be a sufficiently large number, let $w_{n}^{(1)}, \ldots, w_{n}^{(k)}$ be the $n$-step elements of independent permissible random walks on $G$ and let $N = \llangle w_{n}^{(1)}, \ldots, w_{n}^{(k)} \rrangle$.


Fix $\bar{x},\bar{y},\bar{z}\in X/N$ and fix respective lifts $x\in\pi_N^{-1}(\bar{x})$ and $z\in\pi_N^{-1}(\bar{z})$, and let $p_1,p_2$ be constructed from $x$ and $\bar{y}$ as in Lemma \ref{lem:enemyprongs}. Make the following assumptions:
\begin{enumerate}
\item We have $(p_1\cdot z)_x\leq B$, $(p_2\cdot z)_x\leq B$.
\item We have that $x$ and $z$ are near-minimal.
\item Further assume that any axis $h\Ax(v)$ that $50\delta$-fellow-travels $[x,z]_X$ for at least $0.0001L(v)$ does not $10B$-fellow-travel $[x,p_1]_X$ or $[x,p_2]_X$ for distance more than $100B$.
\end{enumerate} 
Then $[\bar{z},\bar{y}]$ passes within $1000B$ of $\bar{x}$, and in fact for any near-minimal respective lifts $z'$ of $\bar{z}$ and $y'$ of $\bar{y}$ and any $[z',y']$, this geodesic passes within $1000B$ of some lift $x'$ of $\bar{x}$. 

Furthermore, if $d_{X/N}(\bar{x},\bar{y})>1000B$ and $\bar{w}$ is some point within distance $200B$ of $\bar{y}$, then $[\bar{z},\bar{w}]$ passes within $2000B$ of $\bar{x}$.
\end{lemma}

\textbf{Proof.} We note that Figure \ref{fig:repel} illustrates many components of the construction used in this proof.

Let $y$ be a minimal lift of $\bar{y}$ with respect to $x$. We then label the strut sequences \[\begin{aligned}
\mathcal{S}(z, x) &:= \big\{ x <_{\mathcal{S}} h_{1} \Ax (v_{1}) <_{\mathcal{S}} \ldots <_{\mathcal{S}} h_{s} \Ax(v_{s}) <_{\mathcal{S}} x\big\}, \\
\mathcal{S}(x, y) &:= \big\{ x <_{\mathcal{S}} j_{1} \Ax (w_{1}) <_{\mathcal{S}} \ldots <_{\mathcal{S}} j_{t} \Ax(w_{t}) <_{\mathcal{S}} z \big\}.
\end{aligned}
\]We may assume without loss of generality that $d_{X/N}(\bar{x},\bar{y})> 500B$, and so $p_1$ and $p_2$ are both at least distance $400B$ from $x$. Also assume that $n$ is sufficiently large so that $\min_iL(v_i)\geq 10^6B$. 
By Lemma \ref{lem:enemyprongs}, as $x$ and $y$ are near-minimal, we have that either $(x\cdot y)_{p_1}\leq 100\delta$ or $(x\cdot y)_{p_2}\leq 100\delta$; assume we are in the first case. Recall that by assumption we have $(p_1\cdot z)_x\leq B$. Taking arbitrary geodesics $[z,x]$, $[x,y]$, and $[z,y]$, we have that $p_1$ lies within $100\delta$ of $[x,y]$, and that $x$ lies within $B$ of $[z,p_1]$. Recall that $d_X(x,p_1)\geq 400B$. Applying the second and third properties of Lemma \ref{lem:gromovprops}, we obtain that $(z\cdot y)_x\leq 2B$ and $(z\cdot y)_{p_1}\leq 150\delta$. 

Applying hyperbolicity, we obtain that $\gamma:=[z,y]$ may be subdivided into an arc $\gamma_z$ that $10\delta$-fellow-travels $[z,x]$ outside of $\N_{3B}(\{x\})$, and a disjoint arc $\gamma_y$ that $10\delta$-fellow-travels $[x,y]$ outside of $\N_{3B}(\{x\})$. Furthermore, $\gamma_y$ passes within $200\delta$ of $p_1$. 

We note that any axis $h\Ax(v)$ that $50\delta$-fellow-travels $\gamma_z$ for distance at least $0.0002L(v)$ must $50\delta$-fellow-travel $[x,z]_X$ for distance at least $0.0001L(v)$. But then such a geodesic can only $50\delta$-fellow-travel $\gamma_y$ for distance at most $105B$. As a consequence, any strut $h\Ax(v)\in \S(x,z)$ satisfies $D_{h\Ax(v)}(x,y)\leq 0.01L(v)$, and vice-versa. As a consequence, $z$ and $y$ must be $0.81$-near-minimal.

Next, from properties of $\gamma_z$ and $\gamma_y$ and a axis-fellow-traveling argument, we see that for any $h_i\Ax(v_i)$, $D_{h_i\Ax(v_i)}(z,y)\geq 0.099L(v_i)$, while for any $j_i\Ax(w_i)$, $D_{j_i\Ax(w_i)}(z,y)\geq 0.099L(w)$. So every strut in $\S(z,x)$ and $\S(x,y)$ lies in $\S_{0.099}(z,y)$. Furthermore, as $\gamma_z$ precedes $\gamma_x$, we must have $h_s\Ax(v_s)<_\S j_1\Ax(w_1)$, and as a consequence the strut sequence $\S_{0.099}(z,y)$ contains (as a subsequence) all struts in $\S(z,x)$ followed by all struts of $\S(x,y)$. By Lemma \ref{lem:strutlifts}, we have the second claim.
If any strut $h\Ax(v)$ lies between $h_s\Ax(v_2)$ and $j_1\Ax(w_1)$, then as it did not appear in $\S(z,x)$ or $\S(x,y)$, and cannot $10\delta$-fellow-travel both $\gamma_z$ and $\gamma_y$ for distance at least $0.01L(v)$, we must have that $D_{h\Ax(v)}(h_s\Ax(v_s),j_1\Ax(w_1))\leq 0.12L(v)$. 

Now, recall the procedure for constructing $p_1$ and $p_2$: we construct a geodesic from $x$ to $j_1\Ax(w_1)$, branch this geodesic in both directions along $j_1\Ax(w_1)$ to produce two geodesics, and then take $p_1$ and $p_2$ as points along this geodesic with (in this case) distance at least $400B$ from $x$.

Assume now that $h_s\Ax(v_s)$ and $j_1\Ax(w_1)$ $20\delta$-fellow-travel for distance $400B$, and let $\eta$ be the segment of $j_1\Ax(w_1)$ where this fellow-traveling occurs. As $h_s\Ax(v_s)$ may only $10\delta$-fellow-travel the first $105B$ units of $\gamma_y$, there must be a point $r$ on $\eta$ within $110B$ of $x$. Furthermore, $r$ must lie within $110B$ of $\pi_{j_1\Ax(w_1)}(x)$. We then have that a segment of either $[x,p_1]$ or $[x,p_2]$ must $10\delta$-fellow-travel $\eta$ of length at least $150B$, contradicting our assumptions on $h\Ax(v_i)$. As a consequence, we obtain the bounds
$$\diam(\pi_{h_s\Ax(v_s)}(j_1\Ax(w_1))),\diam(\pi_{j_1\Ax(w_1)}(h_s\Ax(v_w)))\leq 450B$$
Furthermore, from the fact that $h_s\Ax(v_s)$ $10\delta$-fellow-travels $\gamma_z$ and $j_1\Ax(w_1)$ $10\delta$-fellow-travels $\gamma_y$, we must have that any geodesic from $\pi_{h_s\Ax(v_s)}(j_1\Ax(w_1)))$ to $\pi_{j_1\Ax(w_1)}(h_s\Ax(v_w))$ passes within $900B$ of $x$. We apply Lemma \ref{lem:strutlifts} to conclude.

To obtain the first conclusion for $\bar{w}$, we note that we may find respective lifts $z'$, $y'$, and $w'$ of $\bar{z}$, $\bar{y}$, and $\bar{w}$ so that $w'$ and $z'$ are minimal, $z'$ and $w'$ are near-minimal, and $y'$ and $w'$ are near-minimal. From Lemma \ref{lem:nearmindistbound}, $d_X(y',w')\leq 5000B$, while from the analysis above we have that any $[z',y']$ passes within $1000B$ of some lift $x'$ of $\bar{x}$. From $\delta$-thinness of triangles in $X$, $[w',z']$ must pass within $1200B$ of $x'$. The proof of the second condition is similar. $\square$

\begin{figure}[htbp]
    \centering 
    \includegraphics[width=0.8\textwidth]{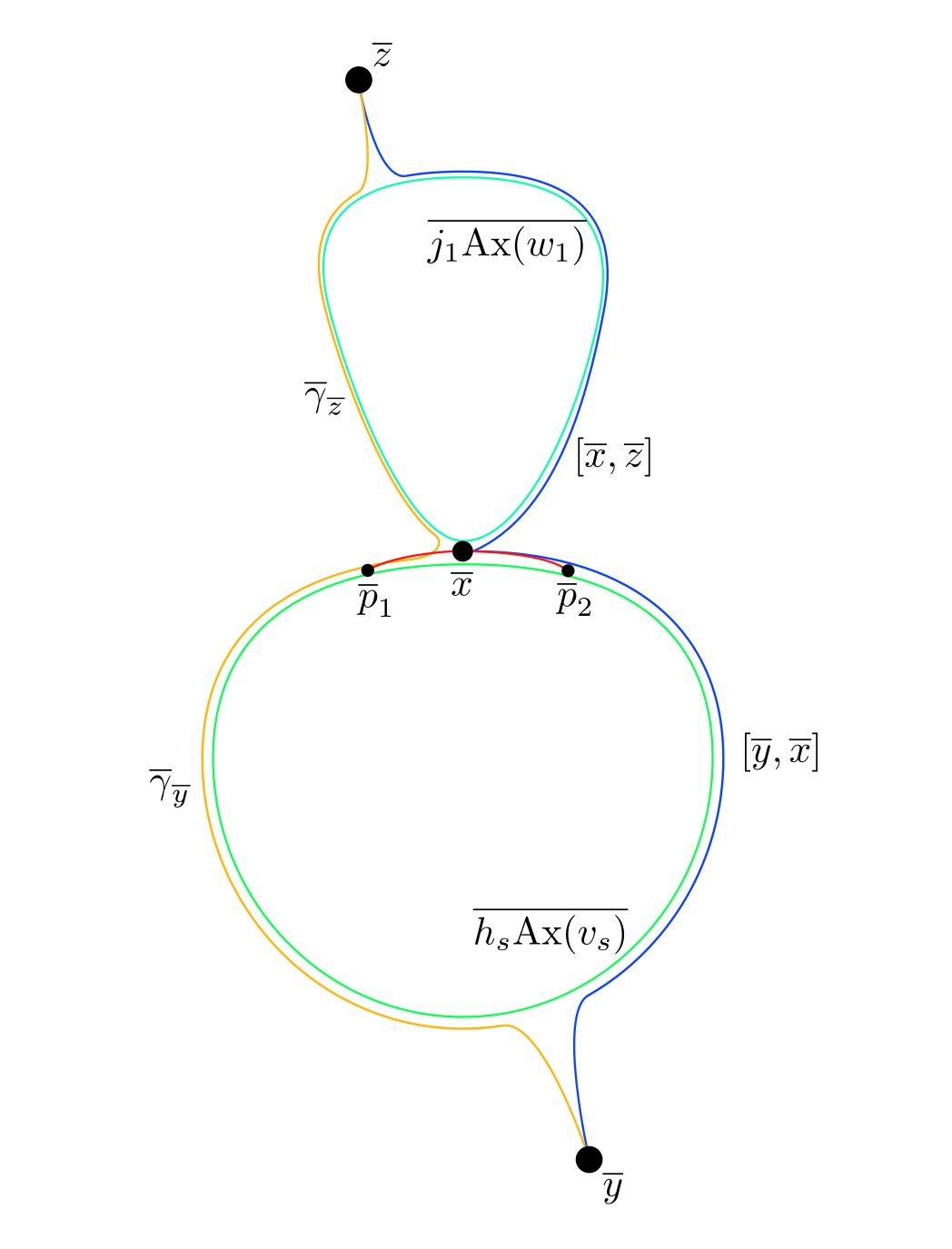} 
    \caption{The construction used in Lemma \ref{lem:quotientrouting}, in the case where $\S(y,x)$ and $\S(x,z)$ each contain only one strut. The key point is that our assumptions force $\overline{j_1\Ax(w_1)}$ to `repulse' from $[\bar{x},\bar{p}_1]$ and $[\bar{x},\bar{p}_2]$. As a result $\overline{j_1\Ax(v_1)}$ and $\overline{h_s\Ax(v_s)}$ may only overlap on the scale of $B$ (which does not grow with $n$), not on the scale of $\M$. This overlap bound plus properties of strut sequences forces the desired routing.}
    \label{fig:repel}
\end{figure}

In order to make use of this lemma, we must demonstrate that we can find axes that are sufficiently `repulsing' from a given set of enemy prongs. To do so, we use the notion of a \emph{Schottky set}. These sets were first defined in \cite{BMSS20}, although the formulation we use is slightly modified from the formulation used by Gouëzel in \cite{SG22}. Intuitively, given two independent loxodromic isometries in $G$, the construction gives a method for building a large number of distinct loxodromic isometries, all pointing in `different directions'  from a basepoint $x_0$.

\begin{definition}

Let $\eta,E_G,E_D, E_D',E_M\geq 0$. A finite set $\EE\subset G$ of isometries is a $(\eta,E_G,E_D,E_D',E_M)$-\emph{Schottky set} if the following conditions are satisfied:
\begin{enumerate}
\item For all $x,y\in X$, we have $|\{s\in \EE:(x\cdot sy)_{x_0}\leq E_G\}|\geq (1-\eta)|\EE|$ and $|\{s\in S: (x,s^{-1}y)_{x_0}\leq E_G\}|\geq (1-\eta)|\EE|$.
\item For all $s\in \EE$, we have $E_D\leq d_X(x_0,sx_0)\leq E_D'$, and $d_G(e,s)\leq E_M$.

\end{enumerate}

\end{definition}

The following lemma is a straightforward consequence of the proof of Prop. 3.12 in \cite{SG22}. It will be used to build axes in the main construction so that Lemma \ref{lem:quotientrouting} can be applied.

\begin{lemma}[Prop. 3.12, \cite{SG22}]
\label{lem:schottkyset}

For all $\eta>0$ and $E,F>0$, there exists $E_G$, $E_D$, $E_D'$ and $E_M$ and a set $\EE$ such that $\EE$ is a $(\eta,E_G,E_D,E_D',E_M)$-Schottky set, and such that $E_D\geq 10E_G+F$ and $E_G\geq 1000(\delta+E)$. Furthermore, $E_G$, $E_D$,, $E_D'$ and $E_M$ depend only on $\eta$ and the action $G\curvearrowright X$, and in particular are independent of $n$.

\end{lemma}

\section{Main construction}

In this section, we aim to prove the following result. Once this result is proven, linear connectivity of the boundary will follow from standard arguments.

As usual, let $n$ be a sufficiently large number, let $w_{n}^{(1)}, \ldots, w_{n}^{(k)}$ be the $n$-step elements of independent permissible random walks on $G$ and let $N = \llangle w_{n}^{(1)}, \ldots, w_{n}^{(k)} \rrangle$. We assume we are in the setup of Hypothesis \ref{hyp:main}, with $M$ as in this hypothesis. Note that we may take $M$ to be as large as needed by taking $n$ sufficiently large.

We will fix a uniform basepoint $\bar{x}_0\in X/N$ throughout these arguments. Before we begin, we note that by Lemma \ref{lem:quotienthyp}, all boundary points in $\partial (X/N)$ may be represented by $(1,20L)$-quasi-geodesic rays, and furthermore, when calculating Gromov products for boundary points as in Equation \ref{eq:boundarygromov}, it suffices to take the infimum over all representative sequences derived from $(1,20L)$-quasi-geodesic rays. 

Our goal is to prove the following lemma, which will be broken up over several subsections and sublemmas.

\begin{lemma}
\label{lem:mainconstruction}

Fix $\xi_1,\xi_2\in\partial (X/N)$ with $(\xi_1\cdot \xi_2)_{\bar{x}_0}=D$, along with a constant $F$. Then for all $n$ sufficiently large in terms of $F$, the following property holds: there exists a finite sequence of boundary points $\xi_1=\chi_1,\chi_2,...,\chi_Q=\xi_2$ such that for $1\leq i\leq Q-1$, $(\chi_i\cdot\chi_{i+1})_{\bar{x}_0}\geq D+\frac{1}{2}F$, and for all $2\leq i\leq Q-1$ and $j\in\{1,Q\}$, $(\chi_i\cdot\chi_j)_{\bar{x}_0}\geq D-100L$.

\end{lemma}

\subsection{Building the Construction}

\begin{figure}[htbp]
    \centering 
    \includegraphics[width=1\textwidth]{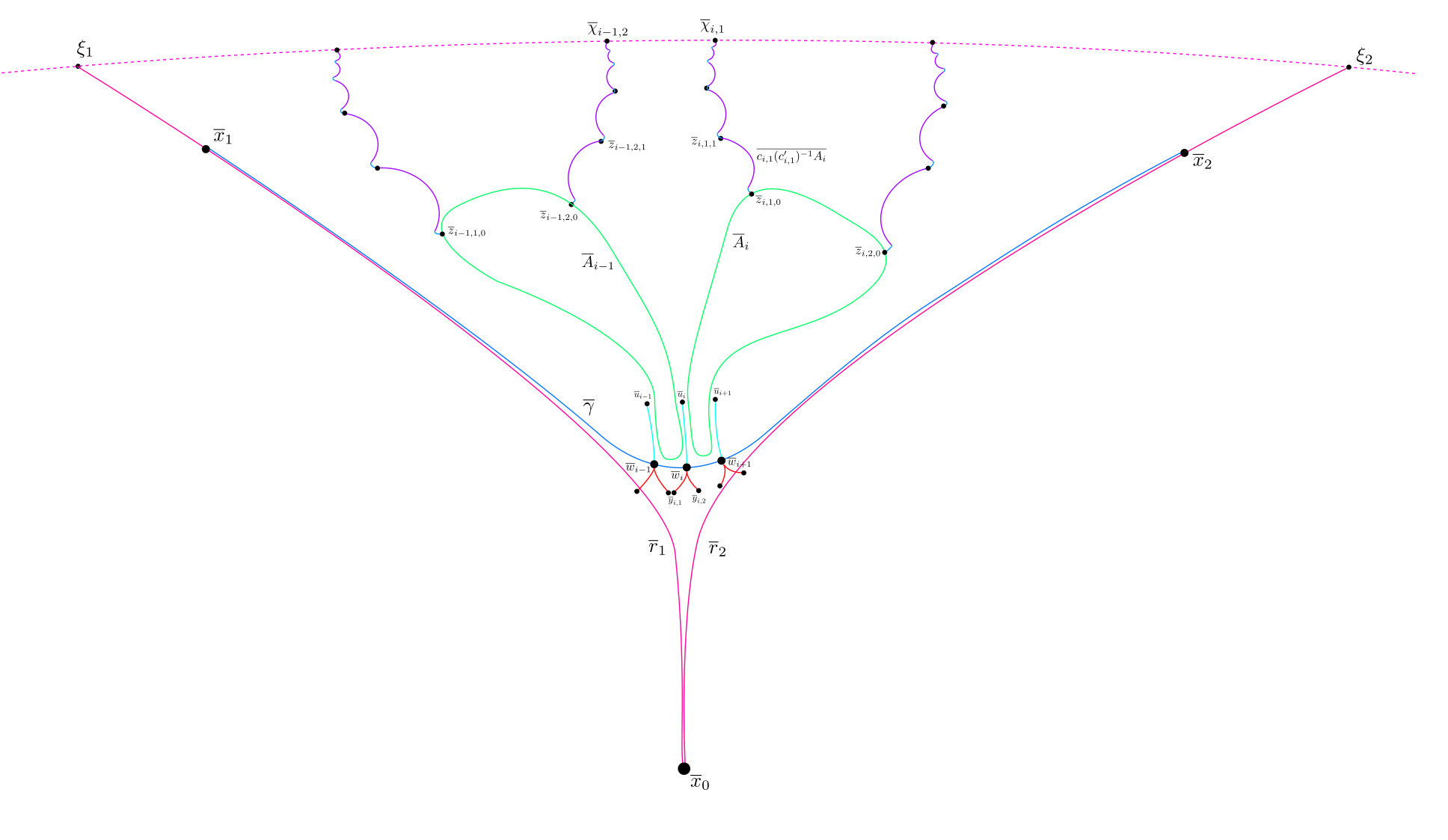} 
    \caption{A schematic of the construction used to prove Lemma \ref{lem:mainconstruction}, showing the relevant objects under projection to $X/N$. Note that the axes $\bar{A}_i$ are constructed using the points $\bar{u}_i$ as guides, and in particular must be repulsed from the points $\bar{y}_{i,j}$.}
    \label{fig:mainconstruction}
\end{figure}

We will first set up the construction. Note that most relevant features of the construction will appear in Figure \ref{fig:mainconstruction}; keeping this figure in hand throughout the construction and its verification may be useful. We will also need several constants with a somewhat complex relationship; see Figure \ref{fig:dependence} as a guide. 

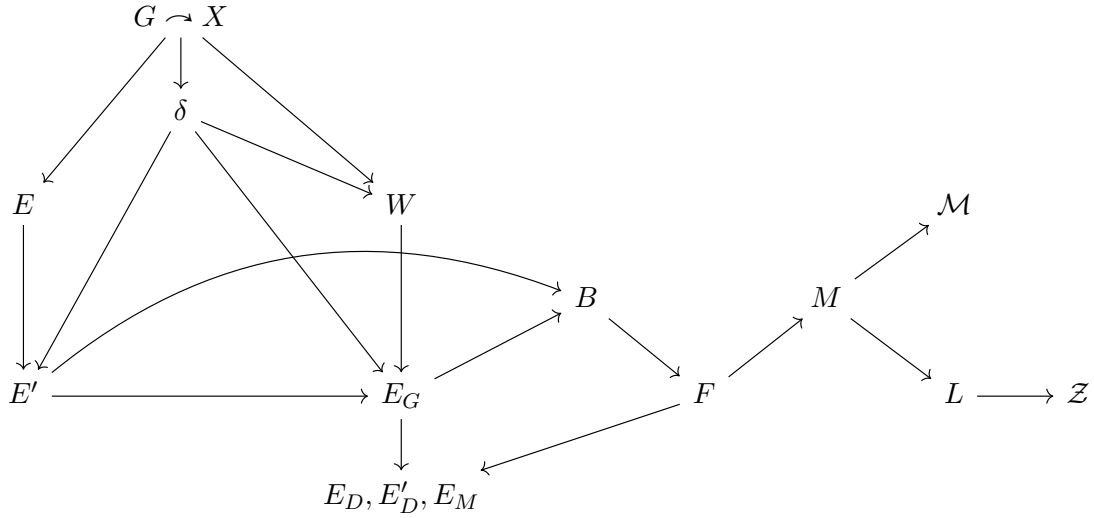
\begin{figure}[htbp]
    \centering 
    \begin{tikzcd}
                                      & G\curvearrowright X \arrow[ldd] \arrow[rdd] \arrow[d] &                          &              &                          &                         &             &             \\
                                      & \delta \arrow[lddd] \arrow[rddd] \arrow[rd]           &                          &              &                          &                         &             &             \\
E \arrow[dd]                          &                                                       & W \arrow[dd]             &              &                          &                         & \mathcal{M} &             \\
                                      &                                                       &                          & B \arrow[rd] &                          & M \arrow[ru] \arrow[rd] &             &             \\
E' \arrow[rr] \arrow[rrru, bend left] &                                                       & E_G \arrow[d] \arrow[ru] &              & F \arrow[lld] \arrow[ru] &                         & L \arrow[r] & \mathcal{Z} \\
                                      &                                                       & {E_D, E_D', E_M}         &              &                          &                         &             &            
\end{tikzcd}
    \caption{A dependence chart of constants used in the main construction. An arrow from $a$ to $b$ indicates that $b$ must depend on the value of $a$. An arrow from $G\curvearrowright X$ indicates constants that depend only on the action of $G$ on $X$, including constants intrinsic to $X$. This graph must contain no cycles. We note that while $F$ is an input to Lemma \ref{lem:mainconstruction}, we may need to replace it with a larger constant, and it is sufficient to verify the lemma for this new constant.}
    \label{fig:dependence}
\end{figure}

Let $\bar{r}_1,\bar{r}_2:\intgr_{\geq 0}\to X/N$ be $(1,20L)$-quasi-geodesic rays originating at $\bar{x}_0$ and respectively representing $\xi_1$ and $\xi_2$. Recalling that 
$$L:=\max_{h\Ax(v)\in \A}L(v),$$
we let $\mathcal{Z}$ be a large constant, on the order of $1000L$ or larger, that will be chosen later. This constant is uniform with respect to $\xi_1$, $\xi_2$, and $D$. We may choose $\bar{r}_1$ and $\bar{r}_2$, and respective points $\bar{x}_1$ and $\bar{x}_2$ on these rays satisfying the following conditions:
\begin{enumerate}
\item We have $d_{X/N}(\bar{x}_0,\bar{x}_i)\geq \mathcal{Z}$.
\item We have $(\bar{x}_1,\bar{x}_2)_{\bar{x}_0}=_{10} D$.
\end{enumerate}

Let $x_0,x_1,x_2$ be respective lifts of $\bar{x}_0,\bar{x}_1,\bar{x}_2$ to $X$, chosen via Lemma \ref{lem:nearmintriple} so that $x_1$ and $x_2$ are minimal, $x_0$ and $x_1$ are near-minimal, and $x_0$ and $x_2$ are near-minimal. Let $\gamma$ be a geodesic from $x_1$ to $x_2$; setting $\bar{\gamma}=\pi_{X/N}(\gamma)$, $\bar{\gamma}$ is a geodesic from $\bar{x}_1$ to $\bar{x}_2$. As $X$ is action path-connected (see Definition \ref{def:actiopathconnected}), there exist a sequence of group elements $a_1,...,a_k$ such that $a_1x_0=x_1$, $a_kx_0=x_2$, such that $d_G(a_{i},a_{i+1})\leq W$ for all $1\leq i\leq k-1$, and so that $x_1=a_1x_0,a_2x_0,...,a_{k-1}x_0,a_kx_0=x_2$ is a unparameterized $(E,E)$-quasi-geodesic. Here, $E$ and $W$ are independent of $n$ and $L$. We will fix the notation $w_i=a_ix_0$, and define $\bar{w}_i=\pi_{X/N}(w_i)$. Of course, $a_i\bar{x}_0=\bar{w}_i$. 

Now, there exists some constant $E'$ depending only on $E$ and $\delta$ so that every $w_i$ lies within distance $E'$ of $\gamma$, and hence every $\bar{w}_i$ lies within distance $E'$ of $\bar{\gamma}$. Furthermore, for all $\bar{w}_i$ and $j\in\{1,2\}$, we have $(\bar{w}_i\cdot \bar{x}_j)_{\bar{x}_0}\geq D-10E'$. Note that this particular Gromov-product bound follows from the triangle inequality, and so holds in any geodesic metric space.

We now construct enemy prongs. Let $B$ be the constant appearing in Lemma \ref{lem:quotientrouting}; we will need to assume this constant is sufficiently large in terms of several other constants, but not that it will \emph{not} need to depend on $n$. For each $w_i$ with $2\leq i\leq k-1$, define $y_{i,1},y_{i,2}$ via Lemma \ref{lem:enemyprongs} so that, in the notation of that lemma, $y_1=p_1(w_i,\bar{x}_0)$ and $y_2=p_2(w_i,\bar{x}_0)$.

\begin{lemma}
\label{sublem:insertaxes}

There exist constants $E_G$, $E_D$, $E_D'$ and $E_M$ satisfying
$$E_D\geq 10(E_G+F+B),\ E_G\geq 10000(\delta+F+E')$$
where $E'$ is as in the preceding discussion, such that the following statement is true.

There exist group elements $s_i\in G$, points $u_i=w_is_ix_0$ in $X$, and axes $A_i=h_i\Ax(v_i)$ (stretching notation so that here the $v_i$ index is different from the index used in our hypotheses) such that for all $2\leq i\leq k-1$ and $j\in\{1,2\}$, the following properties hold, with appropriate modifications for $i=2$ and $i=k-1$:
\begin{enumerate}
\item We have
$$(y_{i,j}\cdot u_i)_{w_i}\leq E_G \text{, } (y_{i-1,j}\cdot u_i)_{w_i}\leq E_G, \text{ and } (y_{i+1,j}\cdot u_i)_{w_i}\leq E_G.$$
\item We have
$$(u_{i-1}\cdot u_{i})_{w_i}\leq E_G\text{ and } (u_{i+1}\cdot u_i)_{w_i}\leq E_G.$$
\item We have
$$E_D\leq d_X(w_i,u_i)\leq E_D'.$$
\item We have
$$d_G(e,s_i)\leq E_M.$$
\item The points $w_i$ and $u_i$ are minimal. In particular,
$$d_{X/N}(\pi_N(w_i),\pi_N(u_i))\geq E_D.$$
\item The axis $A_i$ passes within distance $10\delta$ of both $u_i$ and $u_{i+1}$.

\end{enumerate}

\end{lemma}

\textbf{Proof.} We apply Lemma \ref{lem:schottkyset}. Consider $w_2$. Set $\eta=\frac{1}{1000}$, and let $E_G,E_D,E_D',E_M$ be as in the aforementioned lemma, so that there exists a $(\eta,E_G,E_D,E_D',E_M)$-Schottky set $\EE$ with respect to basepoint $w_2$, and so that $E_D\geq 10E_G+F$ and $E_G\geq 1000(\delta+F+E')$, where $E'$ is the constant from the action path-connectedness construction. Note that the dependence of these constants is somewhat tricky; see \ref{fig:dependence}. However, note that these constants are independent of $n$ and $L$. For $j\in\{1,2\}$, by definition of $\EE$,
$$|\{s\in\EE:(y_{2,j}\cdot a_2sx_0)_{w_2}\leq E_G\}|\geq \frac{999}{1000}|\EE|$$
By a union bound, there exists at least one (in fact, many) $s_2\in \EE$  such that if $u_2=a_2s_2x_0$, then $(y_{2,j}\cdot u_2)_{w_2}\leq E_G$ and $(y_{3,j}\cdot u_2)_{w_2}\leq E_G$ for $j\in\{1,2\}$. Furthermore, we have $d_X(w_2,u_2)\geq E_D$, and $d_G(e,s_2)\leq E_M$.

For sufficiently large $n$, $E_D'\leq 0.01\min_v(L(v))$, and we obtain as a result that for all $h\Ax(v)\in \A$, $d_{h\Ax(v)}(w_2,u_2)\leq 0.01L(v)$. So $w_2$ and $u_2$ have trivial strut sequence, and by Lemma \ref{lem:strutfacts} are minimal.

For each $w_i$ up to $w_{k-1}$, we construct $u_i$ and $s_i$ analogously, except we also require that 
$$(u_i\cdot u_{i-1})_{w_i}\leq E_G\text{ and } (u_i\cdot y_{i-1,j})_{w_i}\leq E_G.$$
These additional conditions pose no obstacle to the union bound argument. Notice the two Gromov products in the second item of the lemma statementare coarsely equal with coarseness only depending on the action $G\curvearrowright X$, so, up to allowing a slightly larger $E_G$, it suffices to prove only one of the two.

We now construct sequence of partially overlapping axes. Consider some $u_i$ and $u_{i+1}$. Recalling $u_i=a_is_ix_0$, we have 
$$d_G(a_is_i,a_{i+1}s_{i+1})\leq d_G(e,s_i)+d_G(a_i,a_{i+1})+d_G(e,s_{i+1})\leq 2E_M+W$$
Notice that the bound $2E_M+W$ does not depend on $n$ or $L$. We may therefore apply Theorem \ref{thm:randomwalkproperties} and conclude that, for sufficiently large $n$, there exists an axis $A_i=h_i\Ax(v_{ i})$ such that $A_i$ passes within $10\delta$ of both $u_2$ and $u_3$ (the word-length bound is not needed here). $\square$

\

Set $L_i=L(v_i)$. For each $A_i$, we choose two points $z_{i,1}$, $z_{i,2}$ on $A_i$ such that $d_{A_i}(w_i,z_{i,1})=_{100\delta}0.4L(A_i)$, $d_{A_i}(w_i,z_{i,2})=_{100\delta} 0.4L(A_i)$, and $d_{A_i}(z_{i,1},z_{i,2})=_{100\delta}0.2L(A_i)$. Necessarily  $d_X(u_i,z_{i,1})\leq 0.45L(A_i)$, and $d_X(u_{i+1},z_{i,2})\leq 0.45L(A_i)$. 

For each $z_{i,j}$, say $z_{i,j}=h_{i,j}x_0$. 

\begin{lemma}
\label{sublem:axisextender}

For each point $z_{i,j}$ and axis $A_i=h_i\Ax(v_i)$ as above, there exists a group element $c_{i,j}\in G$ with the following properties:

\begin{enumerate}
\item There exists a group element $c'_{i,j}$ so that $[x_0,c_{i,j}x_0]\subset\N_{20\delta}((c_{i,j}')^{-1}h_i\Ax(v_i))$.
\item We have $d_X(x_0,c_{i,j}x_0)=_{100\delta}\frac{1}{3}L_i$.
\item No $(c_{i,j}')^{-1}h_iv_ih_i^{-1}c'_{i,j}$-translate of $(c_{i,j}')^{-1}z_{i,j}$ lies within distance $\frac{1}{4}L_i$ of $[x_0,c_{i,j}x_0]$.
\end{enumerate}

\end{lemma}

\textbf{Proof.} Note that the $(c_{i,j}')^{-1}h_iv_ih_i^{-1}c'_{i,j}$-translates of $(c_{i,j}')^{-1}z_{i,j}$ are spaced at distance at least $0.99L_i$ along $(c_{i,j}')^{-1}\Ax(v_i)$. Therefore, we may construct the desired $c_{i,j}$ by choosing a segment of $(c_{i,j}')^{-1}\Ax(v_i)$ with length approximately $\frac{1}{3}$ of a fully translation such that the segment is appropriately far away from any of these translates. $\square$

\

We now define a sequence $z_{i,j,q}=h_{i,j}c_{i,j}^qx_0$ for $q\in\intgr_{\geq 0}$, and consider projections $\bar{z}_{i,j,q}=\pi_{X/N}(z_{i,j,q})$ of thse points. We will define $\chi_{i,j}$ to be the boundary point derived from the sequence $(\bar{z}_{i,j,q})_q$. (Clearly we must check that this sequence defines a boundary point; we will do so shortly.) Our claim will be that the sequence of boundary points
$$\xi_1=\chi_1,\chi_{2,1},\chi_{2,2},\chi_{3,1},...,\chi_{k-1,1},\chi_{k-1,2},\chi_k=\xi_2$$
satisfies the conclusion of Lemma \ref{lem:mainconstruction} after the appropriate reindexing.

\subsection{Verifying the Construction}

We have several claims to verify, which we split into a sequence of lemmas. As before, most relevant objects are depicted in Figure \ref{fig:mainconstruction}.

\begin{lemma}
\label{sublem:sequencelimits}

For all $2\leq i\leq k-1$, $1\leq j\leq 2$ the sequence $(\bar{z}_{i,j,q})_q$ defines a boundary point in $\partial (X/N)$, and in fact there exists a $(1,1000\M)$-quasi-geodesic ray originating at $\bar{z}_{i,j,0}$ and passing through each $\bar{z}_{i,j,q}$.

\end{lemma}

\textbf{Proof.} For ease of notation, fix $i$ and $j$ and set $z_q=z_{i,j,q}$ and $\bar{z}_q=\bar{z}_{i,j,q}$. Also, we may assume without loss of generality that $h_{i,j}=e$. We must show that
$$\lim_{q,r\to\infty}(\bar{z}_q\cdot \bar{z}_{r})_{\bar{x}_0}=\infty$$
Consider the concatenation
$$[x_0,c_{i,j}x_0]\cup[c_{i,j}x_0,c_{i,j}^2x_0]\cup...\cup[c_{i,j}^{q-1}x_0,c_{i,j}^qx_0]$$
Each geodesic segment has length at least $0.31L_i$. On the other hand, observe that by construction
$$[c_{i,j}^nx_0,c_{i,j}^{n+1}x_0]\subset \N_{20\delta}(c_{i,j}^n(c_{i,j}')^{-1}h_i\Ax(v_i)).$$
Therefore, if $[c_{i,j}^nx_0,c_{i,j}^{n+1}x_0]$ and $[c_{i,j}^{n+1}x_0,c_{i,j}^{n+2}x_0]$ $50\delta$-fellow-travel for more than distance $100\M$, then by the no-overlap condition, $c_{i,j}$ must fix the axis $(c_{i,j}')^{-1}h_i\Ax(v_i)$. But then, as the length of $[x_0,c_{i,j}x_0]$ is approximately $\frac{1}{3}L_i$, we would find that the no-self-overlap condition is violated. In particular, $c_{i,j}$ does not fix $(c_{i,j}')^{-1}h_i\Ax(v_i)$.

As $100\M$ is far smaller than $\frac{1}{3}L_i$, we conclude that the $z_q$ are $200\M$-aligned, and that any geodesic $[z_0,z_q]$ must decompose into segments that $50\delta$-fellow-travel translates of $\Ax(v_i)$ for length at least $0.3L_i$ and at most $0.35L_i$, and `hops' in between these segments having length at most $200\M$. It is straightforward to verify from these conditions that the resulting path is a $(1,1000\M)$-quasi-geodesic.

We claim that all $z_q$ are minimal with respect to $z_0$. If not, there is some distinct $z_q'$ that is minimal with respect to $z_0$, with $z_q'=hz_q$ for $h\in N\setminus\{e\}$ and $d_X(z_q',z_0)<d_X(z_q,z_0)$. By Lemma \ref{lem:CM2.1}, there is some axis $h\Ax(v)$ such that $d_{h\Ax(v)}(z_q',z_q)\geq 0.99L(v)$. As $z_0$ and $z_q'$ are minimal, $d_{h\Ax(v)}(z_q',z_0)\leq 0.51L(v)$, so we must have $d_{h\Ax(v)}(z_q,z_0)\geq 0.4L(v)$ by the triangle inequality. Therefore, there is a segment $\gamma$ of $[z_0,z_q]$ that $50\delta$-fellow-travels $h\Ax(v)$ for length at least $0.36L(v)$, and this segment must overlap a segment on which $[z_0,z_q]$ $50\delta$-fellow-travels a translate of $\Ax(v_i)$, with the overlap having length at least $1000\M$. We have that $h\Ax(v)$ cannot be equal to this translate, as the axis-fellow-traveling segments of $[z_0,z_q]$ have length uniformly bounded above by $0.35L_i$. Therefore, we obtain a match between $h\Ax(v)$ and some translate of $\Ax(v_i)$, which is a contradiction by the no-overlap condition.

By translation, we obtain that all elements of the sequence $(z_q)_q$ are pairwise minimal. Therefore $(\bar{z}_q\cdot \bar{z}_r)_{\bar{z}_0}=(z_q\cdot z_r)_{z_0}$ for all $q,r$. Furthermore, by the quasi-geodesic claim made above, it is more or less immediate that $\lim_{q,r\to\infty}(z_q\cdot z_r)_{z_0}=\infty$. $\square$

\begin{lemma}
\label{sublem:forcedpassthrough}

For all $z_{i,j,q}$ and all $h\Ax(v)\in \A$, we have
$$d_{h\Ax(v)}(z_{i,j,q},w_i)\leq 0.41L(v).$$
In particular, $z_{i,j,q}$ and $w_i$ are minimal.

Furthermore, there exists a constant $B$ independent of $n$ and $F$ with the following property. For $M$ sufficiently large in terms of $B$, and for all $\bar{z}_{i,j,q}$ and any geodesic $[\bar{x}_0,\bar{z}_{i,j,q}]$, this geodesic passes within $1000B$ of $\bar{w}_i$ and within $150\M$ of $\bar{z}_{i,j,0}$.

\end{lemma}

\textbf{Proof.} It essentially suffices to verify the hypotheses of Lemma \ref{lem:quotientrouting}. Let $w_i$ and $z_{i,j,q}$ be as in the construction, and recall that $y_{i,1}$ and $y_{i,2}$ are constructed as enemy prongs relative to $\bar{x}_0$. Assume that $j=1$; the case for $j=2$ is analogous.

First, we check that $w_i$ and $z_{i,1,q}$ are minimal. Consider the path $[w_i,z_{i,1,0}]\cup [z_{i,1,0},z_{i,1,q}]$. The axis $A_i$ passes within $10\delta$ of $u_i$ and $u_{i+1}$, and by construction $(u_i,u_{i+1})_{w_i}\leq E_G$. Furthermore, both $u_i$ and $u_{i+1}$ have distance at least $9E_G$ from $w_i$, and so we may conclude that $A_i$ must pass within $2E_G$ of $w_i$. Furthermore, $z_{i,1,0}$ lies on $A_i$, and $d_{A_i}(w_i,z_{i,1,0})=_{100\delta} 0.4L(A_i)$. We conclude by the Bounded Geodesic Image Theorem that $[w_i,z_{i,1,0}]$ has length at least $0.39L(A_i)$ and at most $0.41L(A_i)$, and $50\delta$-fellow-travels $A_i$ except for a segment of length at most $4E_G$.

We verified in the proof of Lemma \ref{sublem:sequencelimits} that there is a $(1,1000\M)$-quasi-geodesic path from $z_{i,1,0}$ to $z_{i,1,q}$, which $50\delta$-fellow-travels a sequence of distinct translates of $A_i$ for lengths between $0.3L_i$ and $0.35L_i$, with segments in between of length at most $200\M$. We also know that $[z_{i,1,0},z_{i,1,1}]$ $50\delta$-fellow-travels some translate $A_i'$ of $\Ax(v_i)$ for length at least $0.32L_i$ and at most $0.34L_i$, and that $[w_i,z_{i,1,0}]$ $5E_G$-fellow-travels $A_i$ for length at least $0.39L_i$ and at most $0.41L_i$. 

We claim that $A_i'\neq A_i$, using Lemma \ref{sublem:axisextender}. If not, then as 
$$[h_{i,1}x_0,h_{i,1}c_{i,1}x_0]=[z_{i,1,0},z_{i,1,1}]\subset \N_{20\delta}(h_{i,1}(c'_{i,1})^{-1}h_i\Ax(v_i))$$
we must have that $A_i'=h_{i,1}(c'_{i,1})^{-1}h_i\Ax(v_i)$, and as $h_{i,1}(c'_{i,1})^{-1}$ therefore fixes $A_i'=A_i=h_i\Ax(v_i)$, we must have that $h_{i,j}(c'_{i,1})^{-1}\in\langle h_iv_ih_i^{-1}\rangle$. As a consequence of Lemma \ref{sublem:axisextender}, no $h_{i,1}(c'_{i,1})^{-1}h_iv_ih_i^{-1}c'_{i,1}(h_{i,1})^{-1}$-translate of $h_{i,1}(c'_{i,1})^{-1}z_{i,1,0}$ lies within distance $\frac{1}{4}L_i$ of $[z_{i,1,0},z_{i,1,1}]$. Using $h_{i,j}(c'_{i,1})^{-1}\in\langle h_iv_ih_i^{-1}\rangle$ and the fact that $z_{i,1,0}$ lies on $A_i$, we have that no $h_iv_ih_i^{-1}$-translate of $z_{i,1,0}$ lies within distance $\frac{1}{4}L_i$ of $[z_{i,1,0},z_{i,1,1}]$, which is clearly contradictory. 

As such, $[w_i,z_{i,1,0}]$ and $[z_{i,1,0},z_{i,1,1}]$ cannot $50\delta$-fellow-travel for length more than $100\M$ without inducing a failure of the no-overlap condition. Arguing as in Lemma \ref{sublem:sequencelimits}, we may patch axis-following segments together and conclude that for any $h\Ax(v)\in \A$,
$$d_{h\Ax(v)}(w_i,z_{i,1,q})\leq0.41L(v).$$
Therefore $w_i$ is minimal relative to all $z_{i,1,q}$. As a further consequence, any geodesic $[w_i,z_{i,1,q}]$ must $50\delta$-fellow-travel $[w_i,z_{i,1,0}]$ until it has passed within $1000\M$ of $z_{i,1,0}$. Since $w_i$ lies within $4E_G$ of $h_i\Ax(v_i)$ and $z_{i,1,0}$ lies on $h_i\Ax(v_i)$, we must have that $h_i\Ax(v_i)$ is the first strut appearing in $\S(w_i,z_{i,1,q})$. By analogous analysis, the axis that $50\delta$-fellow-travels $[z_{i,1,0},z_{i,1,1}]$ is the second strut appearing in $\S(w_i,z_{i,1,q})$, and if this axis is labeled $h_i'\Ax(v_i')$, $z_{i,1,0}\in \N_{120\M}(\pi_{h_i\Ax(v_i)}(h_i'\Ax(v_i'))$. We aim to apply this analysis together with Lemma \ref{lem:quotientrouting} to conclude that for any near-minimal respective lifts $x_0'$ of $x_0$ and $z'_{i,1,q}$ of $\bar{z}_{i,1,q}$, any $[x_0',z'_{i,1,q}]$ passes within $150\M$ of some lift of $\bar{z}_{i,1,0}$

Next, it follows from the geometric analysis performed above, the fact that $d_X(w_i,u_i)\geq 9E_G$, and a standard thin-quadrilaterals argument that any geodesic $[w_i,z_{i,1,q}]$ passes within $50\delta$ of $u_i$. Recall that for $j\in\{1,2\}$, $(y_{i,j}\cdot u_i)_{w_i}\leq E_G$. Therefore, for $B$ sufficiently large in terms of $E_G$, we have that $(y_{i,1}\cdot z_{i,j,q})_{w_i}$ and $(y_{i,2}\cdot z_{i,j,q})_{w_i}$ are both bounded above by $B$. Lastly, assume we have some axis $h\Ax(v)$ such that $h\Ax(v)$ $50\delta$-fellow-travels $[w_i,z_{i,j,q}]$ for length at least $0.0001L(v)$, and that this axis $10B$-fellow-travels, say, $[w_i,y_{i,1}]$ for distance at least $200B$. By taking $M$ and $L$ sufficiently large, we may assume that $\M\geq 10^6*B$, and so $h\Ax(v)$ must pass within $10\M$ of $w_i$. From the structure of the $(1,1000\M)$-quasi-geodesic constructed above between $w_i$ and $z_{i,j,q}$, we have that $h\Ax(v)$ must $50\delta$-fellow-travel one of the axes tracked by this quasi-geodesic for length at least $0.00001L(v)$, and thus be equal to one of these axes (as otherwise the no-overlap condition would be violated). But then its proximity to $w_i$ forces $h\Ax(v)=A_i$. Lastly, $A_i$ passes within $10\delta$ of both $u_i$ and $u_{i+1}$, while $(y_{i,1}\cdot u_i)_{w_i}$ and $(y_{i,1}\cdot u_{i+1})_{w_{i+1}}$ are bounded above by $E_G$. Recall that $d_X(w_i,w_{i+1})\leq 10E$ (where $E$ is the action path-connectedness constant), so for $E_G$ taken suitably large relative to $E$, $(y_{i,1}\cdot u_{i+1})_{w_i}\leq 2E_G$ as well. For $B$ sufficiently large in terms of $E_G$, a straightforward $\delta$-hyperbolicity argument now shows that $A_i$ cannot $10B$-fellow-travel $[w_i,y_{i,1}]$ for distance more than $200B$.

Lemma \ref{lem:quotientrouting} now yields both results. $\square$

\

We now turn to the Gromov-product claims.

\begin{lemma}
\label{sublem:outernoslip}

For all $1\leq i\leq Q-1$, we have $(\xi_1\cdot \chi_{i})_{\bar{x}_0}\geq D-100L$ and $(\xi_2,\chi_i)_{\bar{x}_0}\geq D-100L$.

\end{lemma}

\textbf{Proof.} We prove the case of $\xi_1$; the case of $\xi_2$ follows by identical argument. Let $\bar{x}_1'$ be a point on $\bar{r}_1$; we may assume by taking $\bar{x}_1'$ sufficiently far from $\bar{x}_0$ that $d_{X/N}(\bar{x}_0,\bar{x}_1')\geq 2d_{X/N}(\bar{x}_0,\bar{x}_1)$. Applying the triangle inequality for Gromov products together with Lemma \ref{lem:quotienthyp}, for any $z_{i,j,q}$,
$$(\bar{x}_1'\cdot \bar{z}_{i,j,q})_{\bar{x}_0}\geq\min((\bar{x}_1',\bar{w}_i)_{\bar{x}_0},(\bar{w}_i,\bar{z}_{i,j,q})_{\bar{x}_0})-10L$$
By Lemma \ref{sublem:forcedpassthrough} and the previous observation on $d_{X/N}(\bar{w}_i,\bar{x}_0)$,
$$(\bar{w}_i,\bar{z}_{i,j,q})_{\bar{x}_0}\geq d_{X/N}(\bar{w}_i,\bar{x}_0)-2000B\geq D-10E'-3000B$$
On the other hand,
$$(\bar{x}_1',\bar{w}_i)_{\bar{x}_0}\geq\min((\bar{x}_1'\cdot\bar{x}_1)_{\bar{x}_0},(\bar{x}_1\cdot \bar{w}_i)_{\bar{x}_0})-10L$$
By setting $\mathcal{Z}$ (which controls $d_{X/N}(\bar{x}_0,\bar{x}_1)$) sufficiently large in terms of $L$ and $D$, we can ensure $(\bar{x}_1'\cdot\bar{x}_1)_{\bar{x}_0}\geq 10D$, as $\bar{x}_1'$ and $\bar{x}_1$ lie on the same quasi-geodesic ray $\bar{r}_1$. On the other hand, $(\bar{x}_1\cdot \bar{w}_i)_{\bar{x}_0}$ is at least $D-1000B$ by Lemma \ref{sublem:forcedpassthrough} and the triangle inequality. After unwinding the minima, we obtain that for the particular representatives of $\xi_1$ and $\chi_i$ chosen, the limiting Gromov product is at least $D-25L$. We may now apply Lemma \ref{lem:consistentgromov} to conclude. $\square$

\begin{lemma}
\label{sublem:innergain}

For all $\chi_i$, $1\leq i\leq Q-1$, $(\chi_i\cdot\chi_{i+1})_{\bar{x}_0}\geq D+\frac{1}{2}F$.

\end{lemma}

\textbf{Proof.} Without loss of generality, we assume that $F\geq 10^9B$, as $B$ does not depend on $F$.  We then require that $\mathcal{Z}\geq D+F+10^6*L$ Recalling the reindexing required to construct $\chi_1,...,\chi_Q$, there are a few cases to handle.

The first case is $\xi_1$ and $\chi_{2,1}$. For this case, Lemma \ref{sublem:forcedpassthrough} implies that any geodesic from $\bar{z}_{2,1,q}$ to $\bar{x}_0$ passes within $1000B$ of $\bar{w}_2$, and by construction $(\bar{x}_1,\bar{w}_2)_{\bar{x}_0}\geq \mathcal{Z}-10L$. Arguing as in Lemma \ref{sublem:outernoslip} then yields that $(\xi_1,\chi_{2,1})_{\bar{x}_0}\geq \mathcal{Z}-100L$.

The case of $\xi_2$ and $\chi_{k-1,2}$ is analogous.

Next, we consider $\chi_{i,1}$ and $\chi_{i,2}$. Let $\bar{\rho}_1$ be a $(1,20L)$-quasi-geodesic ray representing $\chi_{i,1}$, and let $\bar{\rho}_2$ be a $(1,20L)$-quasi-geodesic ray representing $\chi_{i,2}$, such that both rays originate from $\bar{x}_0$. Recall the existing $(1,1000\M)$-quasi-geodesic rays representing $\chi_{i,1}$ and $\chi_{i,2}$ constructed from the $z_{i,j,q}$ points. 

Applying the Morse Lemma, every point $\bar{s}_1$ sufficiently far along on $\bar{\rho}_1$ satisfies $d_{X/N}(\bar{x}_0,\bar{s}_1)\geq 10^6*L$, $d_{X/N}(\bar{w}_i,\bar{s}_1)\geq 10^6*L$, $d_{X/N}(\bar{z}_{i,1,0},\bar{s}_1)\geq 10^6 *L$, and $d_{X/N}(\bar{s}_1,\bar{z}_{i,1,q})\leq 200L$ for some $\bar{z}_{i,1,q}$. Applying Lemma \ref{lem:nearmintriple}, there exists $s_1$ lifting $\bar{s}_1$ and $x_0'$ lifting $\bar{x}_0$ such that $s_1$ and $x_0$ are minimal, $z_{i,1,q}$ and $x_0$ are near-minimal, and $z_{i,1,q}$ and $s_1$ are near-minimal. Lemma \ref{lem:nearmindistbound} then implies $d_X(s_1,z_{i,1,q})\leq 2000L$. 

By Lemma \ref{sublem:forcedpassthrough}, $[z_{i,1,q},x_0']$ passes within $1000B$ of some lift $w_i'$ of $\bar{w}_i$, and within $150\M$ of some lift $z_{i,1,0}'$ of $\bar{z}_{i,1,0}$. As $z_{i,1,q}$ and $s_1$ both have distance at least $10^5L$ from $w_i'$ and $z_{i,1,0}'$, we may conclude by $\delta$-thinness of the triangle with sides $[s_1,x_0']$, $[z_{i,1,q}',x_0']$, and $[s_1,z_{i,1,q}']$, along with minimality of $s_1$ and $x_0'$, that $[\bar{x}_0,\bar{s}_1]$ passes within $1100B$ of $\bar{w}_i$ at some point $\bar{s}^w_1$ and within $160\M$ of $\bar{z}_{i,1,0}$ at some point $\bar{s}^z_1$. Let $s^w_1$ be the unique lift of $\bar{s}^w_1$ that lies within $1100B$ of $w_i$, and let $s^z_1$ be the unique lift of $\bar{z}_{i,1,0}$ that lies within $160\M$ of $z_{i,1,0}$. Essentially by the same argument used to prove that $w_i$ and $z_{i,1,0}$ are minimal in Lemma \ref{sublem:forcedpassthrough}, $s^w_1$ and $s^z_1$ are minimal. By a standard thin-quadrilaterals argument, any geodesic $[s^w_1,s^z_1]$ passes through some point $t_1$ lying within distance $50\delta$ of $A_i$ and within distance $40000\M$ of $z_{i,1,0}$. Let $\bar{t}_1$ be the projection of such a point on $[\bar{x}_0,\bar{s}_1]$. 

By identical argument, $[\bar{x}_0,\bar{s}_2]$ passes through some point $\bar{t}_2$ lying within distance $50\delta$ of $\bar{A}_i$ and within distance $40000\M$ of $\bar{z}_{i,1,0}$.

From here, we observe the following inequality, which follows solely from the triangle inequality and does not require hyperbolicity:
$$(\bar{s}_1\cdot \bar{s}_2)_{\bar{x}_0}\geq (\bar{t}_1\cdot \bar{t}_2)_{\bar{x}_0}.$$
Using the lifts $t_1$ of $\bar{t}_1$ and $t_2$ of $\bar{t}_2$ constructed above, we have that $d_{A_i}(t_j,w_i)$ both lie in $[0.39L(A_i),0.41L(A_i)]$ for $j\in\{1,2\}$. As a consequence, $d_{X/N}(\bar{t}_1,\bar{t}_2)\leq 0.25L(A_i)$. Now, it is straightforward to verify as in Lemma \ref{sublem:forcedpassthrough} that $t_1$ and $t_2$ satisfy the hypotheses of Lemma \ref{lem:quotientrouting} (by the same argument used for $z_{i,1,0})$, and so $[\bar{t}_1,\bar{x}_0]$ and $[\bar{t}_2,\bar{x}_0]$ both pass within $1000B$ of $\bar{w}_i$, say at points $\bar{t}_1'$ and $\bar{t}_2'$. By taking the (unique) lifts of these points that lie within distance $1000B$ close to $w_i$, we may check minimality as above and so verify that $d_{X/N}(\bar{t}_1,\bar{t}_1')$ and $d_{X/N}(\bar{t}_2,\bar{t}_2')$ are both at least $0.38L(A_i)$. Therefore, for $j\in\{1,2\}$,
$$d_{X/N}(\bar{t}_j,\bar{x}_0)\geq 0.37L(A_i)+D$$
and so
$$(\bar{s}_1,\bar{s}_2)_{\bar{x}_0}\geq(\bar{t}_1,\bar{t}_2)_{\bar{x}_0}=\frac{d_{X/N}(\bar{t}_1,\bar{x}_0)+d_{X/N}(\bar{t}_2,\bar{x}_0)-d_{X/N}(\bar{t}_1,\bar{t}_2)}{2}.$$
$$\geq \frac{0.74L(A_i)+2D-0.25L(A_i)}{2}\geq D+F$$
as long as $M$ (and hence $\min_iL_i$) is chosen to be sufficiently large with respect to $F$. As $\bar{s}_1$ and $\bar{s}_2$ were chosen arbitrarily, we obtain the conclusion that $(\chi_{i,1}\cdot\chi_{i,2})_{\bar{x}_0}\geq D+F$ (and in fact obtain a significantly larger increment).

The last remaining case is $\chi_{i,2}$ and $\chi_{i+1,1}$. Copying the argument used in the previous case, we obtain $\bar{t}_1$ and $\bar{t}_2$ respectively lying close to $\bar{z}_{i,1,0}$ and $\bar{z}_{i,2,0}$. Furthermore, $[\bar{t}_1,\bar{x}_0]$ must pass within $1000B$ of $\bar{w}_{i+1}$, say at some point $\bar{t}_1'$, and $[\bar{t}_2,\bar{x}_0]$ must pass within $1000B$ of $\bar{w}_{i+1}$, say at some point $\bar{t}_2'$. Let $t_1'$ and $t_2'$ be respective lifts of these points lying within $1000B$ of $\bar{w}_{i+1}$. Now, $t_1$, $t_1'$, $w_{i+1}$, and $u_{i+1}$ all lie within $200\M$ of a subsegment of $A_i$ with length at most $0.41L_i$, and it is straightforward to check from this fact that these four points are pairwise minimal. Recall that $u_{i+1}$ lies within $10\delta$ of $A_i$, and its projection to $A_i$ lies between the projection of $w_i$ and the projection of $z_{i,2,0}$, with an analogous statement for $u_{i+1}$ and $A_{i+1}$. As $d_{X/N}(\bar{w}_{i+1},\bar{u}_{i+1})\geq 10E_G+F$  and $F$ be chosen to be sufficiently large with respect to $B$, we may assert that any $[\bar{t}_1,\bar{x}_0]$ passes within $100\delta$ of $\bar{u}_{i+1}$, say at some point $\bar{y}_1$. Likewise, $[\bar{t}_2,\bar{x}_0]$ passes within $100\delta$ of $\bar{u}_{i+1}$, say at $\bar{y}_2$ - here it is important that both $A_i$ and $A_{i+1}$ were constructed with $u_{i+1}$ as a `guide'. Lastly, we may check as in Lemma \ref{sublem:forcedpassthrough} that any geodesic $[\bar{u}_{i+1},\bar{x}_0]$ or $[\bar{y}_1,\bar{x}_0]$ passes within $1000B$ of $\bar{w}_i$. As a result,
$$d_{X/N}(\bar{u}_{i+1},\bar{x}_0)\geq d_{X/N}(\bar{u}_{i+1},\bar{w}_{i+1})+d_{X/N}(\bar{w}_{i+1},\bar{x}_0)-10000B$$
We already know that $d_{X/N}(\bar{w}_{i+1},\bar{x}_0)\geq D-B$ for $B$ sufficiently large in terms of $E$ and $E'$, as well as $d_{X/N}(\bar{w}_{i+1},\bar{u}_{i+1})\geq 10E_G+F$ by the minimality claim of Lemma \ref{sublem:insertaxes}. We conclude that
$$(\bar{s}_1\cdot\bar{s}_2)_{\bar{x}_0}\geq (\bar{t}_1\cdot\bar{t}_2)_{\bar{x}_0}\geq (\bar{u}_{i+1}\cdot \bar{u}_{i+1})_{\bar{x}_0}-1000\delta$$
$$=d_{X/N}(\bar{u}_{i+1},\bar{x}_0)-1000\delta\geq F+D-20000B$$
Choosing $F$ sufficiently large in terms of $B$ produces the desired bound $(\chi_{i,2}\cdot \chi_{i+1,1})_{\bar{x}_0}\geq D+\frac{1}{2}F$. $\square$

\

We have completed the proof of Lemma \ref{lem:mainconstruction}.

\subsection{Linear Connectivity}

We obtain linear connectivity of the boundary as a result of Lemma \ref{lem:mainconstruction}.

\begin{theorem}
\label{thm:maintheorem}

Assume that $G\curvearrowright X$ and random elements $w_{1}, \ldots, w_{k}$ chosen as the $n$-step element of independent permissible random walks satisfy Hypothesis \ref{hyp:main}, where $n$ is sufficiently large in terms of $G\curvearrowright X$. Then there exists a constant $\mathcal{C}'$ such that for all $\xi_0,\xi_1\in\partial (X/N)$, there exists a path $\gamma(t):[0,1]\to\partial(X/N)$ with $\gamma(0)=\xi_0$, $\gamma(1)=\xi_1$, and for all $0\leq t\leq 1$ and $j\in\{0,1\}$,
$$(\xi_j\cdot\gamma(t))_{\bar{x}_0}\geq (\xi_0\cdot\xi_1)_{\bar{x}_0}-\mathcal{C}'.$$

As an immediate consequence, $\partial (X/N)$ is linearly connected.

\end{theorem}

\textbf{Proof.} Given Lemma \ref{lem:mainconstruction}, the remainder of the proof is essentially standard, and we will follow the proof of Prop. 2.2 in \cite{AW24}. We present the details for clarity.

Recall that the visual metric $d_v$ on $\partial (X/N)$ satisfies the property that there exists $a>1$ and $C>1$ so that for all $\xi,\chi\in\partial (X/N)$,
$$\frac{1}{C}a^{-(\xi\cdot\chi)_{\bar{x}_0}}\leq d_v(\xi,\chi)\leq Ca^{-(\xi\cdot\chi)_{\bar{x}_0}}.$$

Given points $\xi_0$ and $\xi_1$ on $\partial (X/N)$, we will iteratively define functions $\gamma_r:I_r\to\partial (X/N)$, for $r\in\intgr_{\geq 0}$, where $I_r\subset[0,1]$ and $I_r\subset I_{r+1}$ for all $r$. We further require that for all $r<r'$, $\gamma_{r'}|_{I_r}=\gamma_r$.We will set the convention that $\xi_t:=\gamma_r(t)$ for any $t\in[0,1]$ and any $r$ such that $t\in I_r$. The base case is $I_0=\{0,1\}$ and $\gamma_0(0)=\xi_0$, $\gamma_0(1)=\xi_1$.

Define $(\xi_0\cdot\xi_1)_{\bar{x}_0}=A$.

We make the following inductive hypotheses:
\begin{enumerate}
\item We have $I_r\subset I_{r+1}$, and $\gamma_{r+1}|_{I_r}=\gamma_r$.
\item If $t_1$ and $t_2$ are adjacent points of $I_r$ with $t_1<t_2$, then
$$(\xi_{t_1}\cdot \xi_{t_2})_{\bar{x}_0}\geq A+\frac{1}{2}Fr$$
\item If $t_1$ and $t_3$ are adjacent points of $I_{r}$ with $t_1<t_3$, and $t_2\in I_{r+1}$ satisfies $t_1\leq t_2\leq t_3$, then
$$(\xi_{t_1}\cdot \xi_{t_2})_{\bar{x}_0}\geq A+\frac{1}{2}Fr-100L$$
\end{enumerate}
Assuming these hypotheses are satisfied, the inductive procedure consists of applying Lemma \ref{lem:mainconstruction} to each pair of adjacent points of $I_r$. Specifically, say $t_1$ and $t_2$ are adjacent in $I_r$ with $t_1<t_2$, so $(\xi_{t_1}\cdot\xi_{t_2})_{\bar{x}_0}\geq A+\frac{1}{2}Fr$. Lemma \ref{lem:mainconstruction} gives a list $\xi_{t_1}=\chi_0,\chi_1,...,\chi_Q=\xi_{t_2}$ such that for $1\leq i\leq Q-1$, $(\chi_i\cdot\chi_{i+1})_{\bar{x}_0}\geq A+\frac{1}{2}F(r+1)$, while for all $1\leq i\leq Q-1$, $(\chi_0\cdot \chi_i)_{\bar{x}_0}\geq A+\frac{1}{2}Fr-100L$. We may therefore set
$$I_{r+1}\cap[t_1,t_2]=\left\{t_1+i\frac{t_2-t_1}{n}:i\in\intgr\cap[0,Q]\right\}$$
and define $\gamma_{r+1}\left(t_1+i\frac{t_2-t_1}{n}\right)=\chi_i$. Performing this construction for all pairs of consecutive points in $I_r$ yields $I_{r+1}$ satisfying the inductive hypotheses.

We set $I=\bigcup_{r=0}^\infty I_r$, and may by the consistency assumption define $\gamma:I\to\partial X$ by extending the $\gamma_r$. Next, we observe the formula
$$\diam(\gamma(I))\leq \diam(\gamma(I_0))+2\sum_{r=0}^\infty \max_{s\in I_{r+1}}\min_{t\in I_r}d_v(\gamma(s),\gamma(t))$$
From our inductive hypotheses, for every $s\in I_{r+1}$, there is a $t\in I_{r}$ with $(\xi_s\cdot\xi_t)_{\bar{x}_0}\geq A+Fr-100L$. Translating to the visual metric,
$$d_v(\gamma(s),\gamma(t))\leq Ca^{-A-Fr+100L}$$
so we obtain
$$\diam(\gamma(I))\leq Ca^{-A}+2C\sum_{r=0}^\infty a^{-A-Fr+100L}.$$
This expression is certainly bounded by $a^{-A}\geq \frac{1}{C}d_v(\xi_0,\xi_1)$ times a constant $\mathcal{L}$ depending only on $C$ and $L$. We conclude that there exists some constant $\mathcal{C'}$ such that for that every $t\in I$,
$$(\xi_1\cdot\gamma(t))_{\bar{x}_0}\geq (\xi_1\cdot\xi_2)_{\bar{x}_0}-\mathcal{C}'$$

Essentially the same estimate yields that if $t_1$ and $t_2$ are consecutive points of $I_r$, then the diameter of $\gamma(I\cap[t_1,t_2])$ is at most $a^{-A}\mathcal{L}$ times $a^{-Fr}$. Given any $\epsilon>0$, we may therefore find $R$ such that if $t_1$ and $t_2$ are consecutive points on $I_R$, then $\diam(\gamma(I\cap[t_1,t_2]))\leq\epsilon /2$. Choosing $\delta$ to be, say, $\frac{1}{10}$ of the length of the smallest distance between consecutive points of $I_R$, then for any $s_1,s_2$ on $I$ with $|s_1-s_2|<\delta$, there is a $z\in I_R$ within $\delta$ of both $x$ and $y$. From the triangle inequality between $\gamma(s_1)$, $\gamma(s_2)$, and $\gamma(z)$, we have $d(\gamma(s_1),\gamma(s_2))\leq\epsilon$.

Therefore, $\gamma$ is uniformly continuous on $I$, and must extend to a continuous function on the closure $\bar{I}$. From the iterative construction, it is straightforward to check that $\bar{I}$ must equal $[0,1]$, and so $\gamma([0,1])$ is connected. Lastly, we have that the diameter of $\gamma(\bar{I})=\bar{\gamma(I)}$ is equal to the diameter of $\gamma(I)$.

We conclude that $\partial (X/N)$ is linearly connected. $\square$

\

\textbf{Proof of Theorem \ref{thm:maintheoremintro}}. With the setup of Theorem \ref{thm:maintheoremintro}, we have that Theorem \ref{thm:randomwalkproperties} implies that Hypothesis \ref{hyp:main} holds with probability approaching $1$ as $n\to\infty$. Theorem \ref{thm:maintheorem} now supplies the desired result. $\square$

\section{Dimension of the Boundary}

\begin{remark}

In this section, all topological spaces will be assumed to be separable metric spaces. This condition holds for all spaces relevant to us, and avoids some edge cases related to the definition of topological dimension.

\end{remark}

\subsection{Goal and Setup}

Recall the following definition.

\begin{definition}
\label{def:topdim}

A separable metric space $S$ has topological dimension $\dim(S)=k$ if $k$ is the minimum nonnegative integer such that every open cover $\mathcal{U}$ of $S$ admits a refinement $\mathcal{U}'$ such that every $s\in S$ lies in at most $k+1$ elements of $\mathcal{U}'$.

\end{definition}

We note that for all $T\subset S$ with the subspace topology, $\dim(T)\leq \dim(S)$. Furthermore, if $S$ is path-connected, then $\dim(S)\geq 1$.

Our goal is to prove the following lemma.

\begin{lemma}
\label{lem:onedim}

Let $G$, $X$, and $N$ be as in Hypothesis \ref{hyp:main}. Then 
$$1\leq \dim(\partial (X/N))\leq\dim (\partial X)+1.$$

\end{lemma}

Before going on, we state two topological results. Both results are classical.

\begin{lemma}
\label{lem:countableclosure}

Let $X$ be a separable metric space, and let $T_1,T_2,...$ be a countable collection of closed subsets. Then
$$\dim\left(\bigcup_{i=1}^\infty T_i\right)\leq\max_i\dim(T_i)$$
where all subsets carry the subspace topology.

\end{lemma}

\begin{lemma}[Urysohn-Menger Sum Formula]
\label{lem:urymengformula}

Let $X$ be a separable metric space, and let $Y_1$ and $Y_2$ be arbitrary subspaces of $X$ such that $X=Y_1\cup Y_2$. Then
$$\dim(X)\leq \dim(Y_1)+\dim(Y_2)+1$$
where $Y_1$ and $Y_2$ carry the subspace topology.
\end{lemma}

We now cover a few elementary results concerning boundary sets of hyperbolic spaces.

\begin{definition}
\label{def:starset}

Let $X$ be a geodesic metric space and let $A$ and $B$ be two subsets of $X$ with $A\subset B$. We define the \emph{star} of $B$ relative to $A$, $B^*(A)$, as
$$B^*(A)=\{x\in X:\exists \ a\in A,b\in B \text{ such that }x\in [a,b]\}.$$

\end{definition}

\begin{definition}
\label{def:limitset}

Let $X$ be a Gromov hyperbolic metric space, and let $A$ be a subset of $M$. We define the \emph{limit set} $\Lambda(A)$ as the subset of $\partial M$ definable via sequences lying entirely in $A$. 
\end{definition}

The following two lemmas are elementary.

\begin{lemma}
\label{lem:starset}

Let $X$ be a Gromov hyperbolic geodesic metric space, and let $A\subset B\subset X$ with $A\neq\emptyset$ and $\diam(A)<\infty$. Then $B^*(A)$ is quasi-convex.

\end{lemma}

\begin{lemma}
\label{lem:limittopology}

Let $X$ be a Gromov hyperbolic simplicial graph, and let $A$ be a connected quasi-convex subsgraph of $X$. Equip $A$ with the induced subgraph metric. Then $\Lambda(A)$ is naturally indentified with $\partial A$. Furthermore, $\Lambda(A)$ is closed in the visual-metric topology on $\partial X$, and the intrinsic visual-metric topology on $\partial A$ agrees with the subspace topology on $\Lambda(A)$.

\end{lemma}

\subsection{Strut Prefixes and Decomposition}

Throughout, we fix an arbitrary basepoint $\bar{x}_0\in X$. Recall that $\A$ is the set of all axes $h\Ax(v)$ in $X$, and $\A/N$ is the set of all images of axes $\bar{h\Ax(v)}$ in $X/N$.

\begin{definition}
\label{def:quotientstruts}

Given $\bar{x},\bar{y}\in X/N$, choose near-minimal respective lifts $x$ and $y$ of $\bar{x}$ and $\bar{y}$, and enumerate $\S(x,y)$ as
$$x<_\S h_1\Ax(v_1)<_\S...<_\S h_k\Ax(v_k)<_\S y$$

We then define the \emph{quotient strut sequence} $\bar{\S}(\bar{x},\bar{y})$ to be the ordered sequence
$$\bar{x}<_{\bar{\S}} \bar{h_1\Ax(v_1)}<_{\bar{\S}}...<_{\bar{\S}} \bar{h_k\Ax(v_k)}<_{\bar{\S}} \bar{y}.$$

\end{definition}

By Lemma \ref{lem:strutlifts}, $\bar{\S}(\bar{x},\bar{y})$ is well-defined, regardless of the lifts chosen. 

\begin{definition}
\label{def:branchset}

We define the \emph{non-turning locus based at $\bar{z}$} $S_{\bar{z}}'$ as the set of all points $\bar{x}\in X/N$ such that there exist minimal respective lifts $x$ and $z$ of $\bar{x}$ and $\bar{z}$ such that for all axes $h\Ax(v)\in \A$, 
$$D_{h\Ax(v)}(x,z)\leq 0.39L(v).$$
We then define the set $S_{\bar{z}}$ by
$$S_{\bar{z}}=(S'_{\bar{z}})^*(\{z\}).$$

\end{definition}

\begin{lemma}
\label{lem:branchconvex}

Let $\bar{x}\in S_{\bar{x}_0}$. Then there exists minimal respective lifts $x$ and $x_0$ of $\bar{x}$ and $\bar{x}_0$ such that for all axes $h\Ax(v)\in \A$,
$$D_{h\Ax(v)}(x,x_0)\leq 0.4L(v).$$

\end{lemma}

\textbf{Proof.} For any $\bar{y}\in S_{\bar{x}_0}$, there exists some $\bar{x}\in X/N$ and a geodesic $[\bar{x}_0,\bar{x}]$ such that $\bar{y}$ lies on this geodesic and $\bar{x}\in S'_{\bar{x}_0}$. Let $x$ be defined so that $[x_0,x]$ lifts $[\bar{x}_0,\bar{x}]$, and let $y$ be the unique lift of $\bar{y}$ on this geodesic. By minimality, for all axes $h\Ax(v)\in \A$, $D_{h\Ax(v)}(x_0,x)\leq 0.51L(v)$. As $\bar{x}\in S'_{\bar{x}_0}$, there exists some lift $x'$ of $\bar{x}$ such that $x_0$ and $x'$ are minimal, and such that for all $h\Ax(v)\in \A$, $D_{h\Ax(v)}(x_0,x')\leq 0.39L(v)$. 

From the triangle inequality, we conclude that for all $h\Ax(v)\in \A$, $D_{h\Ax(v)}(x,x')\leq 0.95L(v)$. From Lemma \ref{lem:CM2.1}, we must have $x=x'$. Now, let $h\Ax(v)\in \A$, so that $D_{h\Ax(v)}(x_0,x)\leq 0.39L(v)$. From the Bounded Geodesic Image Theorem, the diameter of $\pi_{h\Ax(v)}([x_0,x])$ is at most $0.395L(v)$. As $y\in[x_0,x]$, we conclude that $D_{h\Ax(v)}(x_0,y)\leq 0.4L(v)$. $\square$

\

\begin{lemma}

There exists a well-defined bi-Lipschitz map $\iota: S_{\bar{x}_0}\to X$.

\end{lemma}

\textbf{Proof.} Given $\bar{x}\in S_{\bar{x}_0}$, define $\iota(\bar{x})$ to be a lift $x$ of $\bar{x}$ such that $x$ and $x_0$ are minimal. For any two such minimal lifts $x_1,x_2$, we have by Lemma \ref{lem:branchconvex} that for any $h\Ax(v)\in \A$, $D_{h\Ax(v)}(x_0,x_i)\leq 0.4L(v)$ for $i\in\{1,2\}$, so by the triangle inequality $D_{h\Ax(v)}(x_1,x_2)\leq 0.8L(v)$. But then Lemma \ref{lem:CM2.1} forces $x_1=x_2$, and we conclude that $\iota$ is well-defined. 

An analogous argument shows that for $\bar{x},\bar{y}\in S_{\bar{x}_0}$, $x:=\iota(\bar{x})$ and $y:=\iota(\bar{y})$ are near-minimal. As a result, by Lemma \ref{lem:nearmindistbound},
$$d_X(x,y)\leq 10d_{X/N}(\bar{x},\bar{y}).$$
On the other hand, certainly $d_{X/N}(\bar{x},\bar{y})\leq d_X(x,y)$. We obtain the desired result. $\square$

\

We note that for any $\bar{z}=g\bar{x}_0$, we have by equivariance of projections that $gS_{\bar{x}_0}=S_{\bar{z}}$. We therefore obtain the following corollary.

\begin{corollary}
\label{cor:quotientquasitree}

For any $\bar{z}\in X/N$, the set $S_{\bar{z}}$ is connected, quasi-convex, and quasi-isometric to a subset of $X$. As a result, $\partial S_{\bar{z}}:=\Lambda(S_{\bar{z}})$ is closed in $\partial (X/N)$ and satisfies
$$\dim(\partial S_{\bar{z}})\leq\dim(\partial X).$$

\end{corollary}

We now introduce the main machinery. To do so, we first must define \emph{stable struts}.

\begin{definition}
\label{def:stablestruts}

Let $\xi\in\partial (X/N)$, and let $\bar{r}$ be a $(1,20L)$-quasi-geodesic ray originating from $\bar{x}_0$ representing $\xi$. Let $\bar{h\Ax(v)}\in \A/N$. We say that $\bar{h\Ax(v)}$ is a \emph{stable strut} for $\bar{r}$ if there exists a quantity $t_0\geq 0$ satisfying the following two conditions.
\begin{enumerate}
\item For all $t\geq t_0$, $\bar{h\Ax(v)}\in \bar{\mathcal{S}}(\bar{x}_0,\bar{r}(t))$
\item For all $t\geq t_0$, $d_{X/N}(\bar{r}(t),\bar{h\Ax(v)})\geq 1000L$.
\end{enumerate}

\end{definition}

\begin{definition}
\label{def:boundarydecomp}

Define 
$$S:=\bigcup_{\bar{z}\in X/N}S_{\bar{z}},\ \partial S:=\bigcup_{\bar{z}\in X/N} \partial S_{\bar{z}}.$$
Also define $\partial T=\partial S^C$. For each $\bar{h\Ax(v)}\in \A/N$, define $\partial T_{\bar{h\Ax(v)}}$ to be the set of all $\xi\in \partial T$ for which $\bar{h\Ax(v)}$ is a stable strut.

\end{definition}

An illustration of the construction used in the following lemma is provided as Figure \ref{fig:stablestrut}.

\begin{lemma}
\label{lem:finitestablestruts}

Let $\xi\in \partial (X/N)$, and assume that $\xi$ is represented by some $(1,20L)$-quasi-geodesic ray $\bar{r}$ originating from $\bar{x}_0$ such that $\bar{r}$ has finitely many stable struts. Then there exists some $\bar{z}\in X/N$ such that $\xi\in \partial S_{\bar{z}}$.

\end{lemma}

\textbf{Proof.} Note that key objects in this proof are illustrated in Figure \ref{fig:stablestrut}.

By assumption, there exists some finite list $\bar{h_1\Ax(v_1)},...,\bar{h_K\Ax(v_K)}$ of stable struts for $\bar{r}$, and some $t_0$ such that for all $t\geq t_0$, we have $d_{X/N}(\bar{h_i\Ax(v_i)},\bar{r}(t))\geq 1000L$ and $\bar{h_i\Ax(v_i)}\in \bar{\S}(\bar{x}_0,\bar{r}(t))$ for $i\in\{1,...,K\}$. As a consequence of the fact that any $\bar{h_i\Ax(v_i)}$ must lie within $50\delta$ of $[\bar{x}_0,\bar{r}(t_0)]$ and each $\bar{h_i\Ax(v_i)}$ has diameter at most $2L$, we have that 
\begin{equation}\label{eq:strutnearbound}d_{X/N}(\bar{x}_0,\bar{h_i\Ax(v_i)})\leq d_{X/N}(\bar{x}_0,\bar{r}(t_0))+100L.\end{equation}
Now, let $\bar{y}=\bar{r}(t_0+10000L)$, and construct a geodesic $[\bar{x}_0,\bar{y}]$. Let $\bar{z}$ be a point on this geodesic such that $d_{X/N}(\bar{z},\bar{y})= 4000L$; certainly such a point exists. We will claim that $\xi\in \partial S_{\bar{z}}$. 

Let $t\geq t_0+10000L$, and let $\bar{y}_t=\bar{r}(t)$. Construct a geodesic $[\bar{x}_0,\bar{y}_t]$, and let $\bar{z}_t$ be a point on this geodesic such that $d_{X/N}(\bar{z}_t,\bar{y}_t)=2000L$. Our goal is to show that $\bar{z}_t\in S_{\bar{z}}$, and in fact $\bar{z}_t\in S'_{\bar{z}}$. Assume not, so there exist minimal respective lifts $z_t$ and $z$ of $\bar{z}_t$ and $\bar{z}$ and some axis $h\Ax(v)\in \A$ with $D_{h\Ax(v)}(z,z_t)>0.39L(v)$. As $\bar{r}$ is a $(1,20L)$-quasi-geodesic, we may apply a standard thin-quadrilaterals argument and the fact that $\bar{h\Ax(v)}$ has diameter at most $2L$ to obtain
\begin{equation}\label{eq:strutfarbound}d_{X/N}(\bar{x}_0,\bar{h\Ax(v)})\geq d_{X/N}(\bar{x}_0,\bar{r}(t_0))+2000L\end{equation}

We may now apply Lemma \ref{lem:nearmintriple} to find a lift $x_0$ of $\bar{x}_0$ such that $x_0$ is near-minimal with respect to both $z$ and $z_t$. From the triangle inequality, we must have that either $D_{h\Ax(v)}(x_0,z)$ or $D_{h\Ax(v)}(x_0,z_t)$ is at least $0.19L(v)$. Assume for now that $D_{h\Ax(v)}(x_0,z)\geq 0.19L(v)$, so $\bar{h\Ax(v)}\in \bar{\S}(\bar{x}_0,\bar{z})$. In particular, $\bar{h\Ax(v)}$ must lie within $50\delta$ of $[\bar{x}_0,\bar{z}]$, and hence 
$$d_{X/N}(\bar{x}_0,\bar{h\Ax(v)})\leq d_{X/N}(\bar{x}_0,\bar{y})-1500L$$

As a consequence of Lemma \ref{lem:strutlifts} with $\theta=0.15$, if $z'$ is a separate lift of $\bar{z}$ such that $[x_0,z']$ lifts $[\bar{x}_0,\bar{z}]\subset [\bar{x}_0,\bar{y}]$, then there exists some lift $h'\Ax(v')$ with $D_{h'\Ax(v')}(x_0,z')\geq 0.15L(v')$. Note that $x_0$ and $z'$ are necessarily minimal; we may furthermore extend $[x_0,z']$ to $[x_0,y']$ for some lift $y'$ of $\bar{y}$ such that $x_0$ and $y$ are minimal. The Bounded Geodesic Image Theorem plus $[\bar{x}_0,\bar{z}]\subset [\bar{x}_0,\bar{y}]$ then implies $D_{h'\Ax(v')}(x_0,y')\geq 0.14L(v')$. So $\bar{h'\Ax(v')}\in \S_{0.14}(\bar{x}_0,\bar{y})$

Now, consider some $\bar{y}^*=\bar{r}(t^*)$ with $t^*\geq t_0+10000L$. As $\bar{r}$ is a $(1,20L)$-quasi-geodesic based at $\bar{x}_0$, we have that $[\bar{y},\bar{y}^*]$ does not pass within $d_{X/N}(\bar{x}_0,\bar{y})-200L$ of $\bar{x}_0$. In particular, this geodesic does not pass within $1200L$ of $\bar{h\Ax(v)}$. We may now find respective lifts $y^*$, $x_0''$, and $y''$ of $\bar{y}^*$, $\bar{x}_0$, and $\bar{y}$ such that $y^*$ and $y''$ are minimal, $y^*$ and $x_0''$ are near-minimal, and $y''$ and $x_0''$ are near-minimal. From Lemma $\ref{lem:strutlifts}$, there is some lift $h''\Ax(v'')$ of $\bar{h\Ax(v)}$ such that $D_{h''\Ax(v'')}(x_0'',y'')\geq 0.14L(v)$. But as $[y'',y^*]$ can be chosen to not pass within $1200L$ of $h''\Ax(v'')$ by the construction above, we may apply the Bounded Geodesic Image Theorem and conclude after projection that $\bar{h\Ax(v)}\in \bar{\S}(\bar{x}_0,\bar{y}^*)$. As $t^*$ was chosen arbitrarily, we conclude that $\bar{h\Ax(v)}$ is a stable strut for $\bar{r}$. Combining Equations \ref{eq:strutnearbound} and \ref{eq:strutfarbound} now yields a contradiction.

We conclude that $\bar{z}_t\in S_{\bar{z}}$ for all $t\geq t_0+10000L$. As the sequence $\bar{z}_t$ $5000L$-fellow-travels $\bar{r}$, it must serve as a defining sequence for $\xi$. We conclude that $\xi\in \Lambda(S_{\bar{z}})$. $\square$

\begin{figure}[htbp]
    \centering 
    \includegraphics[width=0.6\textwidth]{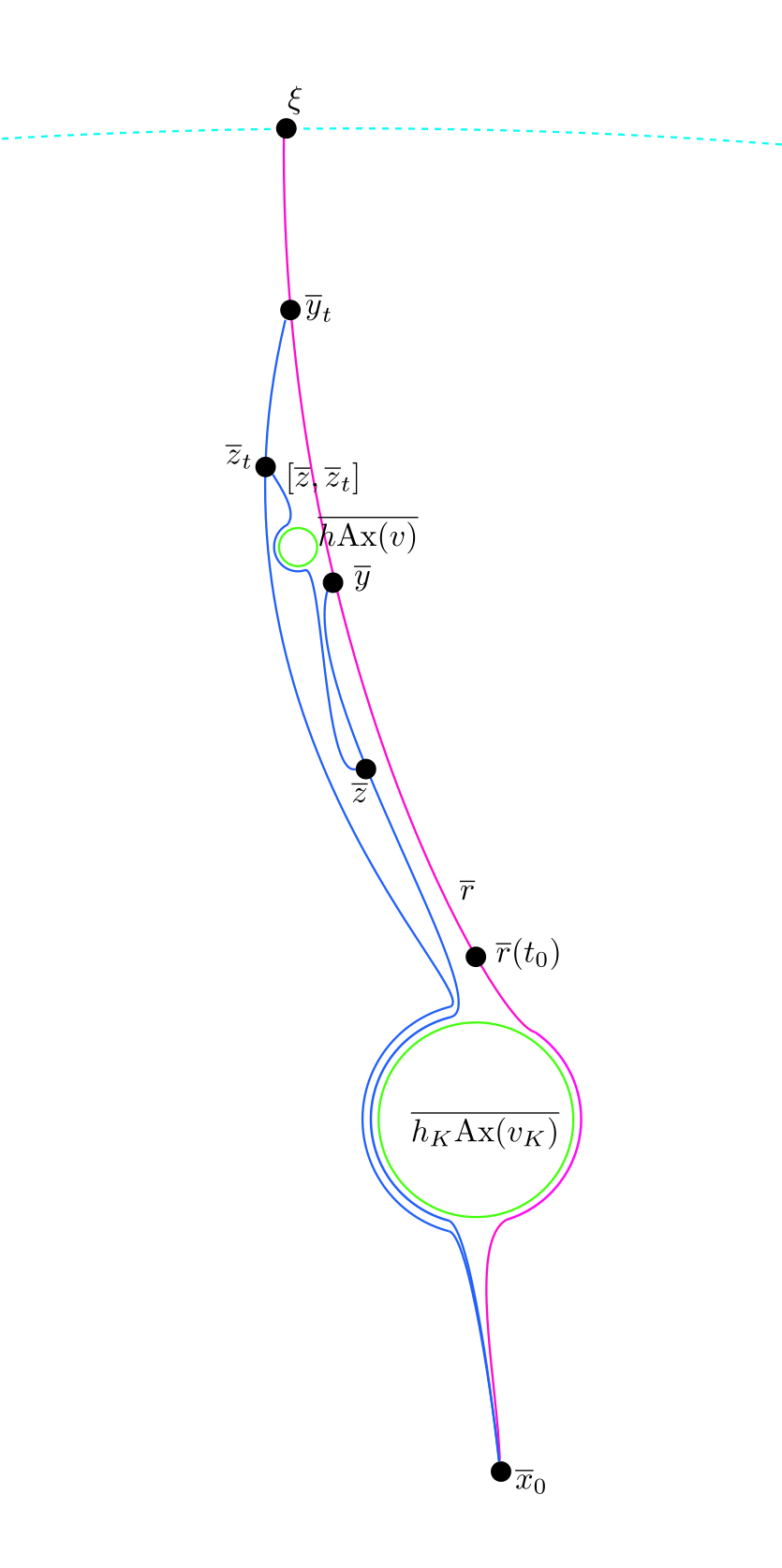} 
    \caption{A schematic of the construction used to prove Lemma \ref{lem:finitestablestruts}, showing the relevant objects under projection to $X/N$. Note that the strut $\overline{h\Ax(v)}$ is proved to in fact not exist over the course of the proof.}
    \label{fig:stablestrut}
\end{figure}

\begin{corollary}
\label{cor:infinitestruts}

If $\xi\in \partial T$, then $\xi$ has infinitely many stable struts.

\end{corollary}

We now investigate the $\partial T_{\bar{h\Ax(v)}}$ sets. An illustration of the construction used in the following lemma is provided as Figure \ref{fig:lockin}.

\begin{lemma}
\label{lem:lockin}

For all $\bar{h\Ax(v)}\in \A/N$, the set $\partial T_{\bar{h\Ax(v)}}$ is clopen in the subspace topology on $\partial T\subset \partial (X/N)$. Furthermore, for any $\xi\in \partial T$, there exists a collection of $\partial T_{\bar{h\Ax(v)}}$ forming a basis at $\xi$ in the subspace topology on $\partial (X/N)$.

\end{lemma}

\textbf{Proof.} Note that key objects in the following construction are depicted in Figure \ref{fig:lockin}.

If $\xi\in\partial T$, then there exists a $(1,20L)$-quasi-geodesic ray $\bar{r}$ representing $\xi$ such that $\bar{r}$ has infinitely many distinct stable struts. Furthermore, let $\bar{h'\Ax(v')}$ be a stable strut for $\bar{r}$, and consider some $t$ such that for all $s\geq t$, $d_{X/N}(\bar{r}(s),\bar{h\Ax(v')})\geq 1000L$. Certainly such a $t$ exists. Then for all minimal respective lifts $x_0$ and $r(t)$, of $\bar{x}_0$ and $\bar{r}(t)$, there exists some lift $h'\Ax(v')$ of $\bar{h'\Ax(v')}$ such that, by Lemma \ref{lem:strutfacts} and the proof of Lemma \ref{lem:finitestablestruts}, $D_{h'\Ax(v')}(x_0,r(t))\geq 0.099L(v')$. Therefore $[x_0,r(t)]$ must $100\delta$-fellow-travel $h'\Ax(v')$ for distance at least $0.099L(v')$, and the same must hold after projection to the quotient. By the non-overlapping condition on axes, if $\bar{r}$ possesses infinitely many stable struts, then it must possess stable struts with arbitrary distance from $\bar{x}_0$.

Now, fix $\bar{h\Ax(v)}\in \A/N$, and let $\xi\in T_{\bar{h\Ax(v)}}$. From the preceding paragraph, we may find some $\bar{h'\Ax(v')}$ such that $\bar{h'\Ax(v')}$ is a stable strut for $\xi$, and such that
$$10d_{X/N}(\bar{x}_0,\bar{h\Ax(v)})+10000L\leq d_{X/N}(\bar{x}_0,\bar{h'\Ax(v')})$$
There exists some $t_0$ such that $\bar{h\Ax(v)},\bar{h'\Ax(v')}\in \bar{\mathcal{S}}(\bar{x}_0,\bar{r}(t_0))$, and such that both $\bar{h\Ax(v)}$ and $\bar{h'\Ax(v')}$ have distance at least $1000L$ from $\bar{r}(t_0)$. Applying the distance bound and Lemmas \ref{lem:strutlifts} and \ref{lem:strutfacts} to $[\bar{x}_0,\bar{r}(t_0)]$, we must have that $\bar{h\Ax(v)}$ comes before $\bar{h'\Ax(v')}$ in $\bar{\mathcal{S}}(\bar{x}_0,\bar{r}(t_0))$, and so $\bar{h\Ax(v)}\in \bar{\mathcal{S}}(\bar{x}_0,\bar{h'\Ax(v')})$.

Choose a constant
$$\mathcal{L}\geq d_{X/N}(\bar{x}_0,\bar{h'\Ax(v')})+10000L$$
and consider the neighborhood
$$\{\chi\in \partial T: (\chi\cdot\xi)_{\bar{x}_0}\geq \mathcal{L}\}.$$
For any such $\chi$, let $\bar{r}'$ be a $(1,20L)$-quasi-geodesic ray originating from $\bar{x}_0$ representing $\chi$. Let $\bar{y}_\chi$ be a point on $\bar{r}'$ such that $d_{X/N}(\bar{x}_0,\bar{y}_\xi)\geq \mathcal{L}+1000L$; note that there is certainly some $t_0$ so that for all $t\geq t_0$, $\bar{r}'(t)$ satisfies this condition. Furthermore, $\bar{y}_\chi$ must satisfy $d_{X/N}(\bar{h'\Ax(v')},\bar{y}_\chi)\geq 1000L$. Likewise, choose a point $\bar{y}_\xi$ on $\bar{r}$ such that
$$d_{X/N}(\bar{x}_0,\bar{y}_\xi)=_{100L} \mathcal{L}+1000L$$
From $(\chi\cdot\xi)_{\bar{x}_0}\geq\mathcal{L}$, the definition of $\mathcal{L}$, and a standard thin-triangles argument, any geodesic $[\bar{y}_\chi,\bar{y}_\xi]$ does not pass within $1000L$ of $\bar{h'\Ax(v')}$.

We may then choose pairwise near-minimal respective lifts $y_\chi$, $y_\xi$, and $x_0$ of $\bar{y}_\chi$, $\bar{y}_\xi$, and $\bar{x}_0$, so that $y_\chi$ and $y_\xi$ are in fact minimal. Necessarily there exists some lift $h'\Ax(v')$ of $\bar{h'\Ax(v')}$ in $\mathcal{S}(x_0,y_\xi)$, and $D_{h'\Ax(v')}(x_0,y_\xi)\geq 0.1L(v')$. There also exists some lift $h\Ax(v)$ of $\bar{h\Ax(v)}$ in $\mathcal{S}(x_0,y_\xi)$, and this axis satisfies
$$D_{h\Ax(v)}(x_0,y_\xi)\geq 0.1L(v),\ D_{h'\Ax(v')}(h\Ax(v),y_\xi)\geq 0.1L(v)$$
Here, the second inequality comes from the fact established above that $\bar{h\Ax(v)}\in \bar{\S}(\bar{x}_0,\bar{h'\Ax(v')})$. As $h'\Ax(v')$ has distance at least $1000L$ from any geodesic from $y_\xi$ to $y_\chi$ by construction, the Bounded Geodesic Image Theorem implies that 
$$D_{h'\Ax(v')}(h\Ax(v),y_\chi)\geq 0.099L(v')$$
By forwarding, we therefore have
$$D_{h\Ax(v)}(x_0,y_\chi)\geq 0.1L(v)$$
and as $x_0$ and $y_\chi$ are near-minimal, $\bar{h\Ax(v)}\in\bar{\mathcal{S}}(\bar{x}_0,\bar{y}_\chi)$.
From this analysis, we may conclude that $\bar{h\Ax(v)}$ is a stable strut for $\chi$, and hence $\chi\in \partial T_{\bar{h\Ax(v)}}$. So
$$\{\chi\in\partial T:(\chi\cdot\xi)_{\bar{x_0}}\geq\mathcal{L}\}\subset \partial T_{\bar{h\Ax(v)}}.$$

Next, assume instead that $\xi\notin T_{\bar{h\Ax(v)}}$. We choose $\bar{h'\Ax(v')}$ and $\mathcal{L}$ as before. Let $\bar{y}_\xi$ be a point on $\bar{r}$ such that $\bar{h'\Ax(v')}\in\bar{\mathcal{S}}(\bar{x}_0, \bar{y})$ and $d_{X/N}(\bar{h'\Ax(v')}, \bar{y}_\xi)\geq 10000L$. Without loss of generality, we may also assume that $d_{X/N}(\bar{x}_0,\bar{y}_\xi)\geq \mathcal{L}+1000L$. If $\bar{h\Ax(v)}\in\bar{\mathcal{S}}(\bar{x}_0,\bar{h'\Ax(v')})$, we would have $\bar{h\Ax(v)}\in \bar{\S}(\bar{x}_0,\bar{y}_\xi)$. Then setting $\chi=\xi$ and rerunning the preceding argument would yield that $\bar{h\Ax(v)}$ is a stable strut for $\xi$, contradicting $\xi\notin \partial T_{\bar{h\Ax(v)}}$. We conclude that for any pairwise minimal lifts $x_0$ and $h'\Ax(v')$ of $\bar{x}_0$ and $\bar{h'\Ax(v')}$ and any lift $h\Ax(v)$ of $\bar{h\Ax(v)}$, we have $D_{h\Ax(v)}(x_0,h'\Ax(v'))<0.1L(v)$.

Consider the neighborhood
$$\{\chi\in\partial T:(\chi\cdot\xi)_{\bar{x}_0}\geq\mathcal{L}\}$$
and assume there exists some $\chi$ lying in this neighborhood so that $\chi\in \partial T_{\bar{h\Ax(v)}}$. We may choose a $(1,20L)$-quasi-geodesic ray $\bar{r}'$ representing $\chi$; for all $y_\chi$ on $\bar{r}'$ sufficiently far from $\bar{x}_0$, we have that $\bar{h\Ax(v)}\in \bar{\mathcal{S}}(\bar{x}_0,\bar{y}_\chi)$ and $d_{X/N}(\bar{x}_0,\bar{y}_\xi)\geq \mathcal{L}+1000L$. As before, by choice of $\mathcal{L}$ and $(\chi\cdot\xi)_{\bar{x}_0}\geq \mathcal{L}$, any geodesic in $X/N$ from $\bar{y}_\xi$ to $\bar{y}_\chi$ must not pass within $1000L$ of $\bar{h'\Ax(v')}$

By essentially the same forwarding argument as in the previous case, plus the upper bound on $D_{h\Ax(v)}(x_0,h'\Ax(v'))$ given above, we conclude that
$$\{\chi\in\partial (X/N):(\chi\cdot\xi)_{\bar{x}_0}\geq\mathcal{L}\}\subset \partial T_{\bar{h\Ax(v)}}^C$$
As sets of the form
$$\{\chi\in\partial (X/N):(\chi\cdot\xi)_{\bar{x}_0}\geq R\}$$
for $R\geq 0$ form a basis neighborhood for $\xi$ in the visual topology on $\partial (X/N)$, we conclude that $\partial T_{\bar{h\Ax(v)}}$ is both closed and open.

Lastly, for any $\bar{h\Ax(v)}\in \A/N$ and for any $\xi,\chi\in \partial T_{\bar{h\Ax(v)}}$, it is straightforward to verify from Lemma \ref{lem:strutfacts} and the stable strut condition that any $(1,20L)$-quasi-geodesic rays representing $\xi$ or $\chi$ must both pass within, say, $100L$ of $\bar{h\Ax(v)}$. As $\bar{h\Ax(v)}$ has diameter at most $2L$, we may conclude that
$$(\xi\cdot\chi)_{\bar{x}_0}\geq d_{X/N}(\bar{x}_0,\bar{h\Ax(v)})-1000L$$
For any $\xi$, we may find stable struts for $\xi$ with arbitrary distance from $\bar{x}_0$. Constructing sets $\partial T_{\bar{h\Ax(v)}}$ from these stable struts therefore gives a neighborhood basis for $\xi$. $\square$

\

\begin{figure}[htbp]
    \centering 
    \includegraphics[width=1.1\textwidth]{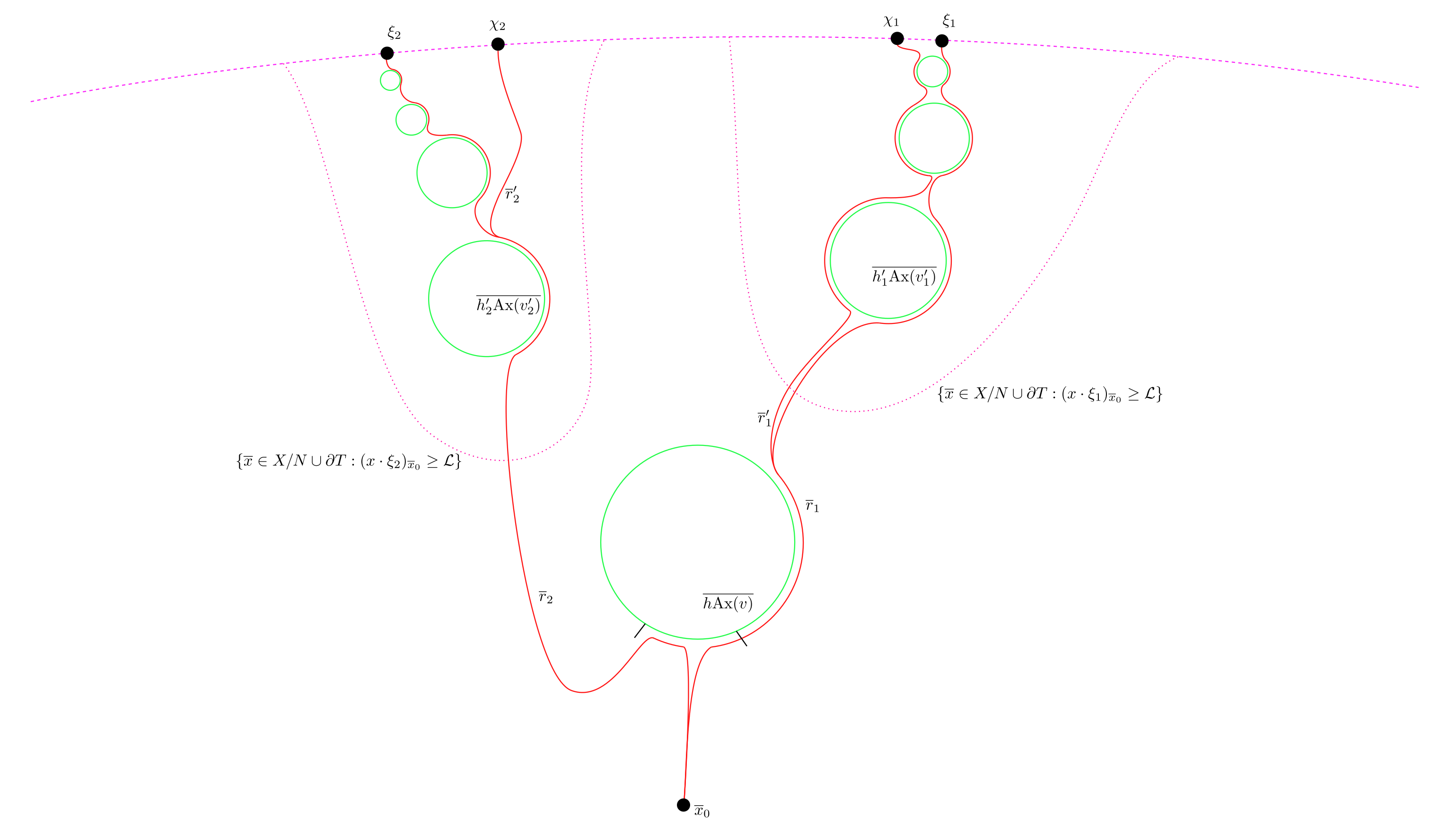} 
    \caption{A schematic of the construction used to prove Lemma \ref{lem:lockin}, showing the relevant objects under projection to $X/N$. The key idea is that the struts $\overline{h_1'\Ax(v_1')}$ and $\overline{h_2'\Ax(v_2')}$ `lock in' projection angles for (respectively) boundary points near $\xi_1$ and $\xi_2$ to $\overline{h\Ax(v)}$. Here, $\xi_1\in \partial T_{\overline{h\Ax(v)}}$, whereas $\xi_2\notin \partial T_{\overline{h\Ax(v)}}$}
    \label{fig:lockin}
\end{figure}

As every point of $\partial T$ has a neighborhood basis of clopen sets of the form $\partial T_{\bar{h\Ax(v)}}$.

\begin{corollary}
\label{cor:infinitestrutsDim}

The space $\partial T\subset \partial (X/N)$ equipped with the subspace topology has topological dimension $0$.

\end{corollary}

\

At last, we may prove our main result.

\

\textbf{Proof of Lemma \ref{lem:onedim}.} Recall the definitions
$$\partial S=\bigcup_{\bar{z}\in X/N}\Lambda(S_{\bar{z}}),\ \partial T=\partial S^C.$$
As the action of $G$ on $X/N$ is transitive and $G$ is finitely generated, the union is countable. From Lemma \ref{lem:countableclosure} and Corollary \ref{cor:quotientquasitree}, we obtain $\dim(\partial S)\leq \dim(\partial X)$, with the subspace topology on $\partial S$. As $\dim(\partial T)=0$ by Corollary \ref{cor:infinitestrutsDim}, we may apply Lemma \ref{lem:urymengformula} to conclude that as $\partial (X/N)=\partial S\cup\partial T$,
$$\dim(\partial (X/N))\leq \dim(\partial S)+\dim(\partial T)+1=\dim(\partial S)+1. $$
Lastly, by Theorem \ref{thm:maintheorem}, $\partial (X/N)$ is linearly connected, and hence has dimension at least $1$. $\square$

\section{Boundary Quotients and the N\"obeling Curve}

\subsection{The N\"obeling Curve}

We now begin the proof of Theorem \ref{thm:nobeling}, motivated largely by the results of \cite{HP11}. First, we provide a brief introduction to the N\"obeling Curve.

\begin{definition}
\label{def:nobeling}

The \emph{N\"obeling Curve} is the topological space defined by removing from $\reals^3$ all points with at least two rational coordinates and taking the subspace topology.

\end{definition}

It is reasonable to think of the N\"obeling curve as a non-compact analogue of the Menger curve (also known as the Menger sponge).

The N\"obeling curve admits a characterization theorem. It may also be defined via a universal property, although we will not state this definition. Before stating the theorem, we must define a somewhat peculiar property.

\begin{definition}
\label{def:locfinite}

A topological space $X$ satisfies the \emph{locally finite $1$-disks property} if the following conditions are satsfied. For any family of continuous maps $f_m:[0,1]\to X$ and any open cover $\mathcal{U}$ of $X$, there is a family of continuous maps $g_m:[0,1]\to X$ such that:
\begin{enumerate}

\item For each $x\in X$, there is an open neighborhood $U$ of $x$ (not necessarily in $\mathcal{U}$) such that the image of $g_n$ does not intersect $U$ for all sufficiently large $n$.
\item For each $t\in[0,1]$ and $n\in\ntrl$, there is an open set $U\in\mathcal{U}$ such that $f_n(t)$ and $g_n(t)$ both lie in $U$.

\end{enumerate}

\end{definition}

For an exposition on why the N\"obeling curve possesses the locally finite $1$-disks property, see Prop. 6.1 in \cite{HP11}. The following result was proven by Andrej Andrzej Nagórko based on results in \cite{KLT97}, but is stated in \cite{HP11}.

\begin{theorem}
\label{thm:nobchar}

A topological space $U$ is homeomorphic to the N\"obeling curve if and only if the following properties are satisfied.
\begin{enumerate}
\item $U$ is Polish (separable and complete metrizable).
\item $U$ is connected and locally path-connected.
\item $U$ has topological dimension $1$.
\item $U$ possesses the locally finite $1$-disks property.

\end{enumerate}

\end{theorem}

Theorem \ref{thm:nobeling} will require additional hypotheses, which we now cover. The first hypothesis is a strengthening of Hypothesis \ref{hyp:main}.

\subsection{Hypotheses}

We now introduce our main hypothesis for the N\"obeling curve result. Note that relative hierarchical hyperbolicity is a strictly weaker condition than hierarchical hyperbolicity.

Note that any HHG $G$ comes with an action $G\curvearrowright \mathcal{X}$ on a \emph{hierarchically hyperbolic space} $\mathcal{X}$, a set $\mathcal{D}$ of domains with a unique maximal domain $X$, and an action $G\curvearrowright \mathcal{D}$. For any $U\in \mathcal{D}\setminus\{X\}$, we have a projection map $\rho^X_U:X\to U$ and a coarse point $\rho^U_X\subset X$. There is also an orthogonality relationship $\perp$ on elements of $\mathcal{D}$. Furthermore, for any $U\in\mathcal{D}$, there is a projection map $\pi_U:\mathcal{X}\to U$; these maps satisfy a coarse functoriality property.

\begin{hypothesis}
\label{hyp:nob}

Let $G$ be a finitely generated relatively hierarchically hyperbolic group (RHHG) such that the RHHG structure of $G$ is realized by an action $G\curvearrowright \mathcal{X}$, where $\mathcal{X}$ is a relatively hierarchically hyperbolic space. Let $\mathcal{D}$ be the collection of domains for $\mathcal{X}$, and recall that by definition there is an action $G\curvearrowright\mathcal{D}$. Let $E$ be the maximum across all constants appearing in the RHHS/RHHG axioms for $G$ and $\mathcal{X}$. Let $X$ be the top level domain of $\mathcal{X}$. By definition, $X$ is a Gromov hyperbolic space, and there is an action $G\curvearrowright X$ of $G$ on $X$ by hyperbolic isometries. We let $w_1,...,w_k$ be elements of $G$. We fix an arbitrary basepoint $x_0\in X$, and consider the normal subgroup $N=\langle\langle w_1,...,w_k\rangle\rangle$. We assume the following:
\begin{enumerate}
\item The space $\mathcal{X}$ is the Cayley graph of $G$ with respect to some finite generating set, and is equipped with the simplicial metric.
\item The space $X$ is a simplicial graph, and its Gromov hyperbolicity constant $\delta$ is at least $1000$.
\item The action of $G$ on $X$ is vertex-transitive and non-elementary.
\item The maximal finite normal subgroup of $G$ is trivial.
\item For each non-maximal $U\in\mathcal{D}$, $\rho^U_X$ is a point in $X$. Furthermore, for $g\in G$, if $U'=gU$, then $\rho^{U'}_X=g\rho^U_X$.
\item There exists $M\geq 10^{20}$ such that
$$(w_{1},...,w_{k})\in\mathcal{E}_{10^{20}M,M}$$
where $\mathcal{E}_{10^{20}M,M}$ is defined as in Definition \ref{def:walkpermissible}.

\end{enumerate}

Given these hypotheses, we set $L=\max_iL(w_i)$ and $\mathcal{M}:=50\min_i L(w_i)/M$.

\begin{remark}

It is straightforward to extend the reductions applied to Hypothesis \ref{hyp:main} to this expanded setup; see Appendix $A$ of \cite{ABMNR26} for the needed facts about quotients of $\mathcal{X}$ by finite normal subgroups of $G$. As before, the other claims are standard.

\end{remark}

\end{hypothesis}




We note that the action of $G$ on $X$ is acylindrical, and hence WPD, by Theorem K of \cite{BHS17b} (the proof is only for HHGs, but the argument works for RHHGs as well), and the action is likewise action path-connected by Lemma \ref{lem:hierarchypaths}. Vertex-transitivity and conditions four and five may likewise be assumed without loss of generality. Possibly replacing $E$ with a larger constant, we can assume that $E$ also serves as the action path-connectedness constant.

We will use `subdomain' to refer to non-maximal domains of $\mathcal{X}$; note that this is not standard terminology. 

We make the following definition, which is the main requirement for showing that $\partial (X/N)$ has the locally finite $1$-disks property.

\begin{definition}
\label{def:largedense}

Let $G$ and $\mathcal{X}$ be as in Hypothesis \ref{hyp:nob}. Recall We say that $\mathcal{X}$ has the \emph{large dense domains} property if there exists a subset $\mathcal{D}_0\subset \mathcal{D}\setminus\{X\}$ such that the following conditions are satisfied for some constant $H$:
\begin{enumerate}
\item The subset $\mathcal{D}_0$ is closed under the $G$-action on $\mathcal{D}$.
\item Every $U\in\mathcal{D}_0$ contains an quasi-isometrically embedded copy of $\reals$.
\item Every point $x\in X$ lies within distance $H$ of some $U\in\mathcal{D}_0$.
\item For every geodesic $[x,y]$ in $X$, there is a sequence $U_1,...,U_k$ of domains in $\mathcal{D}_0$ so that $\rho^{U_1}_X$ lies within $H$ of $x$, $\rho^{U_k}_X$ lies within $H$ of $y$, every $\rho^{U_i}_X$ lies within $H$ of $[x,y]$, and $U_i\perp U_{i+1}$ and $d_X(\rho^{U_i}_X,\rho^{U_{i+1}}_X)\leq 10H$ for all $U_i$ with $1\leq i\leq k-1$.

\end{enumerate}

\end{definition}

We make the following hypothesis.

\begin{hypothesis}
\label{hyp:nobplus}

Let $G$, $\mathcal{X}$, and $X$ be as in Hypothesis \ref{hyp:nob}. We additionally assume that $G$, $\mathcal{X}$, and $X$ satisfy the large dense domains property.

\end{hypothesis}

If needed, we may replace $E$ with a larger constant so that $E\geq H$.

Lastly, we make the following hypothesis on $X$. Note that quasi-trees necessarily have boundaries with dimension $0$.

\begin{hypothesis}
\label{hyp:tree}

Let $G$, $\mathcal{X}$, and $X$ be as in Hypothesis \ref{hyp:nob}. We additionally assume that $X$ is a quasi-tree.
    
\end{hypothesis}

\section{Locally Finite $1$-Disks}

In this section, we prove that under Hypotheses \ref{hyp:nob} and \ref{hyp:nobplus}, $\partial (X/N)$ satisfies the locally finite $1$-disks property. Before we begin, recall the triangle inequality for Gromov products: for any $\xi_1,\xi_2,\xi_3$ that are individually either points in $X/N$ or points in $\partial (X/N)$ and any $\bar{x}\in X/N$, we  by Lemma \ref{lem:quotienthyp} that
$$(\xi_1\cdot\xi_3)_{\bar{x}_0}\geq\min((\xi_1\cdot\xi_2)_{\bar{x}_0},(\xi_2\cdot\xi_3)_{\bar{x}_0})-20L$$

\subsection{Subdomain Projections in the Quotient}

Following Abbott-Berlyne-Mangioni-Ng-Rasmussen in \cite{ABMNR26}, we could define a full RHHS structure on $\mathcal{X}/N$, and obtain subdomain projections accordingly. However, this is more work than is necessary in our case, and we will instead define and prove only what is needed to show the locally finite $1$-disks property. 

Recall that for reach

\begin{definition}
\label{def:quotientdomains}

Define $\bar{\mathcal{D}}$ as the set-theoretic quotient $\mathcal{D}/N$ under the action $G\curvearrowright \mathcal{D}$, except that $X$ is replaced by $X/N$. Define $\bar{\mathcal{D}_0}:=\mathcal{D}_0/N\subset \mathcal{D}/N$ accordingly. For each $\bar{U}\in\bar{\mathcal{D}}\setminus\{ X/N\}$, we define $\bar{U}$ as an isometric copy of any $U\in\mathcal{D}$ lifting $\bar{U}$. Given $\bar{U}\in\bar{\mathcal{D}}$, define $\rho^{\bar{U}}_{\bar{X}}$ as the projection of  $\rho^U_X$ to $\bar{X}$ for any lift $U$ of $\bar{U}$. Note that this point is well-defined by Hypothesis \ref{hyp:nob}.

\end{definition}

Recall that by definition, if $U'=gU$ for distinct $U,U'\in\mathcal{D}$, then the map from $U$ to $U'$ induced by $g$ is an isometry. Furthermore, $\rho^{gU}_X=g\rho^U_X$, and as all $g\in N$ have minimal translation length at least $0.98\min_iL(v_i)$ by Lemma \ref{lem:CM2.1}, we may conclude $N$ acts freely on $\mathcal{D}\setminus\{X\}$. Therefore, if $U'=gU$ for $g\in N$, the induced isometry between $U'$ and $U$ is unique. As a result, we may define a projection map
$$\pi_N:\bigsqcup_{U \text{ lifts }\bar{U}} U\to \bar{U}$$
by identifying $\bar{U}$ with a fixed lift $U_0$ of $\bar{U}$, mapping all other lifts $U$ to $U_0$ via the unique maps described above, and taking a union of images.

Recall that for each $U\in\mathcal{D}\setminus\{X\}$, there exists a `projection' map $\rho^X_U:X\to U$.

\begin{definition}
\label{def:subdomquot}

Given $\bar{x}\in X/N$ and $\bar{U}\in \bar{\mathcal{D}}\setminus\{X/N\}$, define
$$\rho^{\bar{X}}_{\bar{U}}(\bar{x})=\left(\bigsqcup_{\substack{x,U:\\ x\text{ lifts }\bar{x}\\U\text{ lifts }\bar{U}\\x,U\text{ are $0.85$-near-minimal}}}\rho^X_U(x)\right)/N$$
where the quotient by $N$ is defined via the projection map $\pi_N$ described in Definition \ref{def:quotientdomains}.

\end{definition}

Before we go on, recall the Bounded Geodesic Image Axiom for RHHSes.

\begin{lemma}
\label{lem:subdomBGI}

Let $G$, $\mathcal{X}$, and $X$ be as in Hypothesis \ref{hyp:nob}. Let $x$ and $y$ be points of $X$ such that $[x,y]$ does not intersect $\N_E(\rho^U_X)$ for some subdomain $U\in\mathcal{D}\setminus\{X\}$. Then $d_U(\rho^X_U(x),\rho^X_U(y))\leq E$.

\end{lemma}

We now show projections have bounded diameter. A helper lemma is required; this lemma is essential Remark 4.17 in \cite{ABMNR26}.

\begin{lemma}
\label{lem:axisprojection}

There exists a constant $E'$ depending only on $E$ and the action $G\curvearrowright X$  with the following property. Let $h\Ax(v)$ be an axis, let $g=hvh^{-1}$, and let $x$ be a point lying on $h\Ax(v)$. Then for any $U\in \mathcal{D}$, $d_U(\rho^X_U(x),\rho^X_U(gx))\leq E'$.

\end{lemma}

\textbf{Proof.} The existence of $\hat{x}$ follows from transitivity of the action of $G$ on $X$. We note that $\langle hvh^{-1}\rangle$ acts geometrically on the quasi-convex subset $h\Ax(v)$, which implies the existence of constants $(\lambda,c)$ such that the restriction of the map $\pi_X:\mathcal{X}\to X$ (which is the orbit map for $G$) to $\langle hvh^{-1}\rangle$ is a $(\lambda,c)$-quasi-isometry. It now follows from the Bounded Geodesic Image Axiom that the image $\rho^X_U(h\Ax(v))$ has diameter bounded by some function $D(\lambda,c)$ of $\lambda$ and $c$. As $G$ acts on the set $\A$ of axes with finitely many orbits, $\lambda$ and $c$ are uniformly bounded across all axes. We therefore obtain a bound on the diameter of $\rho^X_U(h\Ax(v))$ independent of $h\Ax(v)\in \A$. $\square$

\

We will assume without loss of generality that $E'\leq E$, which may require replacing $E$ with a larger constant.

\begin{lemma}
\label{lem:subdomboundeddiam}

With the notation in Definition \ref{def:subdomquot}, the set $\rho^{\bar{X}}_{\bar{U}}(\bar{x})$ has diameter at most $O(E)$.

\end{lemma}

\textbf{Proof.} Consider pairs $x_1,U_1$ and $x_2,U_2$, both lifting the pair $\bar{x},\bar{U}$, and such that $x_1$ and $U_1$ are $0.85$-near-minimal and $x_2$ and $U_2$ are $0.85$-near-minimal. By definition of $\rho^{\bar{X}}_{\bar{U}}(\bar{x})$, we may translate by $N$ and thus assume without loss of generality that $U_1=U_2$. Call this subdomain $U$. If $x_1=x_2$, we are done. Otherwise, we may apply Lemma \ref{lem:CM2.1} and conclude that there is some axis $h\Ax(v)$ such that $d_{h\Ax(v)}(x_1,x_2)\geq 0.99L(v)$. 

Now, if $d_{h\Ax(v)}(x_i,\rho^U_X)\leq 0.13L(v)$ for $i\in\{1,2\}$, we may conclude by the triangle inequality that either $x_1$ or $x_2$ is not $0.85$-near-minimal with $U$. As a consequence, defining $y_1$ as an arbitrary point in $\pi_{h\Ax(v)}(x_1)$, any geodesic from $x_1$ to $\rho^U_X$ must first pass within $O(\delta)$ of $y_1$ and then travel along $h\Ax(v)$ for distance at least $0.1L(v)$. We may define $y_2$ analogously. By Lemma \ref{lem:subdomBGI}, as long as $L(v)$ is sufficiently large relative to $E$ (which we may assume), $d_U(\rho^X_U(x_i),\rho^X_U(y_i))\leq E$ for $i\in\{1,2\}$.

It is now straightforward to verify that there exists an element $g=hv^mh^{-1}$ with $m\in\{-1,0,1\}$ such that $[gy_1,y_2]$ does not pass within $E$ of $\rho^U_X$, and so $d_U(\rho^X_U(gy_1),\rho^X_U(y_2))\leq E$. Furthermore, by Lemma \ref{lem:axisprojection}, $d_U(\rho^X_U(y_1),\rho^X_U(gy_1))\leq E$. We may now apply the Bounded Geodesic Image axiom a final time and conclude by the triangle inequality that
$$d_U(\rho^X_U(x_1),\rho^X_U(x_2))\leq O(E). \square$$

\

\begin{lemma}
\label{lem:subdomquotBGI}

Let $\bar{x},\bar{y}\in X/N$, and consider $\bar{U}\in\bar{\mathcal{D}}$. Assume that $[\bar{x},\bar{y}]$ does not pass within $E$ of $\rho^{\bar{U}}_{\bar{X}}$. Then $d_{\bar{U}}(\rho^{\bar{X}}_{\bar{U}}(\bar{x}),\rho^{\bar{X}}_{\bar{U}}(\bar{y}))\leq O(E)$.

\end{lemma}

\textbf{Proof.} Let $x,U_1$ be near-minimal lifts of $\bar{x},\bar{U}$, and let $y,U_2$ be near-minimal lifts of $\bar{y},\bar{U}$. By equivariance and a minor modification of Lemma \ref{lem:nearmintriple}, we may assume that $U_1=U_2$ (call this subdomain $U$) and that $x$ and $y$ are minimal. By Lemma \ref{lem:subdomboundeddiam}, it suffices to prove that $d_U(\rho^X_U(x),\rho^X_U(y))\leq E$. But this bound now holds from our construction and the Bounded Geodesic Image axiom for $X$. $\square$

\

\begin{definition}
\label{def:raytruncation}

Given a quasi-geodesic ray $\bar{r}$ in $X/N$ and a point $\bar{y}$ on $\bar{r}$, define $\bar{r}_{\bar{y}}$ to be the `tail' of $\bar{r}$ obtained by removing all points of $\bar{r}$ before $\bar{y}$.

\end{definition}

\begin{definition}
\label{def:boundarysubdomproj}

Fix $\xi\in \partial(X/N)$ and $\bar{U}\in\bar{\mathcal{D}}\setminus \{X/N\}$. Let $R$ be the set of all $(1,20L)$-quasi-geodesic rays $\bar{r}$ in $X/N$ originating from $\bar{x}_0$ and representing $\xi$. Then
$$\rho^{\partial (X/N)}_{\bar{U}}(\xi)=\bigcup_{\bar{r}\in R}\bigcup_{\substack{\bar{y}\in\bar{r}\\ d_{X/N}(\bar{r}_{\bar{y}},\rho^{\bar{U}}_{\bar{X}})\geq 10^{10}L}}\rho^{\bar{X}}_{\bar{U}}(\bar{y})$$

\end{definition}

The following lemma follows directly from Lemmas \ref{lem:subdomboundeddiam} and \ref{lem:subdomquotBGI}.

\begin{lemma}
\label{lem:boundarysubdomprojfinite}

For any $\xi\in\partial (X/N)$ and any $\bar{U}\in\bar{\mathcal{D}}/\{X/N\}$, the set $\rho^{\partial (X/N)}_{\bar{U}}(\xi)$ has diameter at most $O(E)$.

\end{lemma}

\begin{definition}
\label{def:quotientorthog}

Given $\bar{U},\bar{V}\in\bar{\mathcal{D}}\setminus\{X/N\}$, we write $\bar{U}\perp\bar{V}$ if there exist near-minimal respective lifts $U$ and $V$ of $\bar{U}$ and $\bar{V}$ with $U\perp V$.

\end{definition}

We record a useful consequence of partial realization.

\begin{lemma}
\label{lem:boundaryrealization}

Fix $\bar{U}\in\bar{\mathcal{D}}\setminus\{X/N\}$. Then the map $\rho^{\partial (X/N)}_{\bar{U}}\to \bar{U}$ is $O(E)$-coarsely surjective.

\end{lemma}

\textbf{Proof.} Fix $\bar{u}\in\bar{U}$, and let $U$ be an arbitrary lift of $\bar{U}$. By partial realization, there exists a point $\hat{u}\in\mathcal{X}$ and a point $u'\in U$ projecting to $\bar{u}'\in\bar{U}$ such that $u'$ is contained in the projection of $\hat{u}$ to $U$ and $d_{\bar{U}}(\bar{u}',\bar{u})\leq E$. Let $x_u$ be a point in the projection of $\hat{u}$ to $X$, and let $\bar{x}_u$ be its image in $X/N$. Possibly after adjusting $\hat{u}$ and $u'$ in a way that moves $u'$ by distance at most $E$, we may assume without loss of generality that $d_X(x_u,\rho^U_X)\leq E$, and that as a consequence $x_u$ and $U$ are near-minimal.

Say $x_u=hx_0$. As in the proof of Lemma \ref{sublem:axisextender}, we will choose a group element $c\in G$ with the following properties:
\begin{enumerate}
\item There exists a $c'\in G$ and a $v_i$ so that $[x_0,cx_0]\subset\N_{20\delta}((c')^{-1}\Ax(v_i))$.
\item We have $d_X(x_0,cx_0)=_{100\delta} \frac{1}{3}L_i$.
\item We have $d_G(e,c)=O(n)$.

\end{enumerate}

We then take the sequence $z_j=hc^jx_0$, so $z_0=x_u$. By essentially the same proof as in Lemma \ref{sublem:sequencelimits}, the sequence $\bar{z}_j$ defines a boundary point $\xi\in\partial (X/N)$, and all $z_j$ are near-minimal with respect to $U$. Furthermore, we may concatenate geodesic segments $[\bar{z}_j,\bar{z}_{j+1}]$ to obtain a $(1,1000\M)$-quasi-geodesic ray $\bar{s}$. We may subdivide this ray into at most three sections $\bar{s}_1$, $\bar{s}_2$, and $\bar{s}_3$, each section made of some number $[z_j,z_{j+1}]$ segments, and satisfying the following properties. We require that $\bar{s}_1$ and $\bar{s}_3$ do not pass within $10L$ of $\rho^{\bar{U}}_{\bar{X}}$, and that $\bar{s}_2$ consists of at most $40$ segments. Let $z_{j_1}$ be the terminal point of $\bar{s}_1$, let $z_{j_2}$ be the terminal point of $\bar{s}_2$, and let $\bar{z}_{j_3}$ be some arbitrary point on $\bar{s}_3$ with distance at least $10^{20}L$ from $\rho^{\bar{U}}_{\bar{X}}$. By the word-length bound above, the coarse Lipschitz projections property, and Lemma \ref{lem:subdomquotBGI}, we have
$$d_{\bar{U}}(\rho^{\bar{U}}_{\bar{X}}(\bar{z}_0),\rho^{\bar{U}}_{\bar{X}}(\bar{z}_{j_1}))\leq O(E),\ d_{\bar{U}}(\rho^{\bar{U}}_{\bar{X}}(\bar{z}_{j_1}),\rho^{\bar{U}}_{\bar{X}}(\bar{z}_{j_2}))\leq O(E),$$
$$d_{\bar{U}}(\rho^{\bar{U}}_{\bar{X}}(\bar{z}_{j_2}),\rho^{\bar{U}}_{\bar{X}}(\bar{z}_{j_3}))\leq O(E)$$

Let $\bar{r}$ be a $(1,20L)$-quasi-geodesic ray in $X/N$ originating from $\bar{x}_0$ and representing $\xi$. By the Morse Lemma, $\bar{r}$ fellow-travels the sequence $\bar{z}_j$ for sufficiently large $j$. We may now choose a sufficiently large $j_3$ and conclude by the triangle inequality. $\square$

\

We now recall a specific case of the Partial Realization Axiom for HHSes.

\begin{lemma}

Let $U,V\in\mathcal{D}\setminus\{X\}$ with $U\perp V$. Then the map
$$\pi_U\times\pi_V:\mathcal{X}\to U\times V$$
is $E$-coarsely surjective.

\end{lemma}

The following lemma now follows immediately from the proof of Lemma \ref{lem:boundaryrealization} and the Partial Realization Axiom.

\begin{lemma}
\label{def:pairboundaryrealization}

Fix $\bar{U},\bar{V}\in\bar{\mathcal{D}}\setminus\{X/N\}$, and assume that $\bar{U}\perp\bar{V}$. Then the map
$$\rho^{\partial (X/N)}_{\bar{U}}\times\rho^{\partial (X/N)}_{\bar{V}}:\partial (X/N)\to \bar{U}\times\bar{V}$$
is $O(n)$-coarsely surjective.

\end{lemma}

We now state an extension of the Bounded Geodesic Image Axiom that applies to boundary points in the quotient. This lemma follows directly from unraveling definitions and applying Lemma \ref{lem:subdomquotBGI}.

\begin{lemma}
\label{lem:boundaryBGI}

There exists a constant $J=O(L)$ with the following property. Let $\bar{U}\in\bar{\mathcal{D}}\setminus\{X/N\}$ and $\xi_1,\xi_2\in\partial (X/N)$. If $d_{\bar{U}}(\rho^{\partial (X/N)}_{\bar{U}}(\xi_1),\rho^{\bar{X}}_{\bar{U}}(\bar{x}_0))$ is at least $J$, then any $(1,20L)$-quasi-geodesic ray originating at $\bar{x}_0$ representing $\xi$ passes within $J$ of $\rho^{\bar{U}}_{\bar{X}}$. Furthermore, if $d_{\bar{U}}(\rho^{\partial (X/N)}_{\bar{U}}(\xi_1),\rho^{\partial (X/N)}_{\bar{U}}(\xi_2))$ is at least $J$, then $(\xi_1\cdot\xi_2)_{\rho^{\bar{U}}_{\bar{X}}}\leq E$.

\end{lemma}

We now reach the most technical claim, which will require that we rerun the construction from Lemma \ref{lem:mainconstruction} with a few modifications.

\begin{lemma}
\label{lem:narrowpaths}

For any constant $\Psi_1>0$, there exists a constant $\Psi_2>0$ with the following property. Let $\xi_1,\xi_2\in\partial (X/N)$ and $\bar{U}\in\bar{\mathcal{D}}\setminus\{X/N\}$, and assume that $d_{\bar{U}}(\rho^{\partial (X/N)}_{\bar{U}}(\xi_1),\rho^{\partial (X/N)}_{\bar{U}}(\xi_2))\leq \Psi_1$. Then there exists a path $\gamma:[0,1]\to \partial (X/N)$ from $\xi_1$ to $\xi_2$ such that for all $t\in[0,1]$, $d_{\bar{U}}(\rho^{\partial (X/N)}_{\bar{U}}(\gamma(t)),\rho^{\partial X/N}_{\bar{U}}(\xi_1))\leq \Psi_2$.

\end{lemma}

In order to prove this lemma, we make the following claim:

\begin{lemma}
\label{lem:subdommainconstruction}

Assume Hypothesis \ref{hyp:nob}. Fix $\xi_1,\xi_2\in\partial (X/N)$ with $(\xi_1\cdot \xi_2)_{\bar{x}_0}=D$ and $(\xi_1\cdot\xi_2)_{\rho^{\bar{U}}_{\bar{X}}}=D'$, and a constant $F$. Also fix $\bar{U}\in\bar{\mathcal{D}}\setminus\{X/N\}$, and assume that $d_{\bar{U}}(\rho^{\partial (X/N)}_{\bar{U}}(\xi_1),\rho^{\partial (X/N)}_{\bar{U}}(\xi_2))\leq \Psi_1$. Then for $M$ sufficiently large in terms of $F$, the following property holds: there exists a finite sequence of boundary points $\xi_1=\chi_1,\chi_2,...,\chi_Q=\xi_2$ such that for $1\leq i\leq Q-1$, $(\chi_i\cdot\chi_{i+1})_{\bar{x}_0}\geq D+\frac{1}{2}F$, and for all $2\leq i\leq Q-1$, and $j\in\{1,Q\}$, $(\chi_i\cdot\chi_j)_{\bar{x}_0}\geq D-100L$. Furthermore, for all $\chi_i$ with $1\leq i\leq Q-1$ and all $\chi_j$ with $j\in\{1,Q\}$, we have $(\chi_i\cdot\chi_{i+1})_{\rho^{\bar{U}}_{\bar{X}}}\geq D'+\frac{1}{2}F$ as well as $d_{\bar{U}}(\rho^{\partial (X/N)}_{\bar{U}}(\chi_i),\rho^{\partial (X/N)}_{\bar{U}}(\chi_j))\leq \Psi_1+O(L)$.

\end{lemma}

\textbf{Proof.} The proof of this lemma will be substantially derived from the proof of Lemma \ref{lem:mainconstruction}, so we will gloss over the aspects that remain unchanged. We will also borrow key notation from this proof. As before, we let $\bar{r}_1,\bar{r}_2:\intgr_{\geq 0}\to X/N$ be $(1,20L)$-quasi-geodesic rays originating at $\bar{x}_0$ and respectively representing $\xi_1$ and $\xi_2$. We choose points $\bar{x}_1$ and $\bar{x}_2$ on these rays such that $d_{X/N}(\bar{x}_0,\bar{x}_1)\geq \mathcal{Z}$ and $(\bar{x}_1,\bar{x}_2)_{\bar{x}_0}=_{10}D$. We now also require that $d_{X/N}(\bar{x}_i,\rho^{\bar{U}}_{\bar{X}})\geq \mathcal{Z}+10L$. Here, $\mathcal{Z}$ is some large multiple of $L$, say $10^{12}L$; it is straightforward that $\bar{x}_i$ exist satisfying this additional bound.

We let $x_0,x_1,x_2$ be pairwise near-minimal respective lifts of $\bar{x}_0,\bar{x}_1,\bar{x}_2$ to $X$, chosen so that $x_1$ and $x_2$ are minimal. Let $\gamma$ be a gedoesic from $x_1$ to $x_2$, and set $\bar{\gamma}=\pi_{X/N}(\gamma)$. We now use Theorem \ref{thm:hierarchypathsexist} to construct a hierarchy path from $x_1$ to $x_2$ (this construction replaces the use of action path-connectedness). We obtain a sequence of group elements $a_1,...,a_k\in G$ such that $a_1x_0=x_1$, $a_kx_0=x_2$, such that $d_G(a_{i-1},a_i)\leq W$ for all $2\leq i\leq k$, and such that $x_1=a_1x_0,a_2x_0,...,a_kx_0=x_2$ is an unparameterized $(E,E)$-quasigeodesic. As before, we define $\bar{w}_i:=a_i\bar{x}_0$. By Theorem $\ref{thm:hierarchypathsexist}$, the Bounded Geodesic Image Axiom, and our choice of $\mathcal{Z}$, the assumption that $$d_{\bar{U}}(\rho^{\partial (X/N)}_{\bar{U}}(\xi_1),\rho^{\partial (X/N)}_{\bar{U}}(\xi_2))\leq \Psi_1$$
implies that
$$d_{\bar{U}}(\rho^{\bar{X}}_{\bar{U}}(\bar{x}_1),\rho^{\bar{X}}_{\bar{U}}(\bar{w}_i))\leq \Psi_1+O(E^3).$$

Recall that in Lemma \ref{sublem:insertaxes} we constructed points $u_i:=a_is_ix_0$ avoiding certain `enemy prongs'; using the notation of Lemma \ref{lem:enemyprongs}, two of these prongs were $p_1(w_i,\bar{x}_0)$ and $p_2(w_i,\bar{x}_0)$. We add $p_1(w_i,\rho^{\bar{U}}_{\bar{X}})$ and $p_2(w_i,\rho^{\bar{U}}_{\bar{X}})$ to the list of enemy prongs we wish to avoid with $u_i$. Recall that $d_G(e,s_i)\leq E_M$ by the Schottky set construction. As a consequence of the coarse Lipschitz property of projections from $\mathcal{X}$ and the functoriality axiom for projections,
$$d_{\bar{U}}(\rho^{\bar{X}}_{\bar{U}}(\bar{w}_i),\rho^{\bar{X}}_{\bar{U}}(\bar{u}_i))\leq O(EE_M)+O(E)=O(B)$$

Here, $B$ is the constant used in Lemma \ref{lem:mainconstruction}.

While the statement was not needed before, it may be derived from the proof of the second case of Lemma \ref{sublem:innergain} that for any $(1,20L)$-quasi-geodesic ray $\rho$ representing some $\chi_{i,j}$ and any point $\bar{s}$ sufficiently far along on $\rho$, any geodesic $[\bar{s},\bar{x}_0]$ passes through some point $\bar{y}$ lying within distance $100\delta$ of either $\bar{u}_i$ or $\bar{u}_{i+1}$, and furthermore the subgeodsic $[\bar{y},\bar{x}_0]$ passes within $1200B$ of $\bar{w}_i$. We apply essentially the same argument in our new setup, using both the enemy prongs constructed from $\rho^{\bar{U}}_{\bar{X}}$ and the enemy prongs constructed from $\bar{x}_0$. We obtain that any geodesic $[\bar{s},\rho^{\bar{U}}_{\bar{X}}]$ passes through some point $\bar{y}$ lying with distance $100\delta$ of either $\bar{u}_i$ or $\bar{u}_{i+1}$, and likewise $[\bar{y},\rho^{\bar{U}}_{\bar{X}}]$ passes within $1200B$ of $\bar{w}_i$. 

Recall that $ d_X(\bar{w}_i,\bar{u}_i)\geq E_D\geq 10000B$, while $d_G(e,s_i)\leq E_M$, where $E_D$ and $E_M$ are independent of $n$. Therefore, the construction of $\bar{y}$ above is sufficient to conclude that $[\bar{y},\bar{s}]$ and $[\bar{u}_i,\bar{y}]$ do not pass within $E$ of $\rho^{\bar{U}}_{\bar{X}}$. From Lemma \ref{lem:subdomquotBGI},
$$d_{\bar{U}}(\rho^{\bar{X}}_{\bar{U}}(\bar{y}),\rho^{\bar{X}}_{\bar{U}}(\bar{u}_i))\leq O(E),\ d_{\bar{U}}(\rho^{\bar{X}}_{\bar{U}}(\bar{u}_i),\rho^{\bar{X}}_{\bar{U}}(\bar{s}))\leq O(E).$$

We may now apply the triangle inequality and conclude that
\begin{equation}\label{eq:domainshift}d_{\bar{U}}(\rho^{\partial(X/N)}_{\bar{U}}(\chi_{i,j}),\rho^{\bar{X}}_{\bar{U}}(\bar{w}_i))\leq O(B).\end{equation}

The rest of the argument follows as in Lemma \ref{lem:mainconstruction}. At its close, we use the new enemy prongs that were used to construct the $\bar{u}_i$ to argue as in Lemma \ref{sublem:innergain} and obtain that, after relabeing to $\chi_1,\chi_2,...,\chi_Q$, we have
$$(\chi_i\cdot\chi_{i+1})_{\rho^{\bar{U}}_{\bar{X}}}\geq D'+\frac{1}{2}F.$$
Lastly, recall Equation \ref{eq:domainshift}. Applying the triangle equality a final time with $\bar{w}_i$ and either $\bar{x}_1$ or $\bar{x}_2$, and recalling the construction of these points, yields for all $\chi_i$ and $\chi_j$ that
$$d_{\bar{U}}(\rho^{\partial (X/N)}_{\bar{U}}(\chi_i),\rho^{\partial (X/N)}_{\bar{U}}(\chi_j))\leq \Psi_1+O(E).$$
We replace $O(B)$ with $O(L)$, which is suboptimal but harmless for later arguments. The proof is now complete. $\square$

\

\textbf{Proof of Lemma \ref{lem:narrowpaths}.} Let $J$ be the constant in Lemma \ref{lem:boundaryBGI}, and let $\mathcal{C}'$ be the constant in the proof of Theorem \ref{thm:maintheorem}. By repeatedly applying Lemma \ref{lem:subdommainconstruction} approximately $\frac{10(\mathcal{C}'+L)}{F}$ times, we arrive at a sequence $\xi_1=\chi_1,...,\chi_Q=\xi_2$ of points in $\partial (X/N)$ with the following properties:
\begin{enumerate}

\item For $1\leq i\leq Q-1$, $(\chi_i\cdot\chi_{i+1})_{\rho^{\bar{U}}_{\bar{X}}}\geq 10\mathcal{C'}+10L$.
\item For all $2\leq i\leq Q$ and $j\in\{1,2\}$, we have
$$d_{\bar{U}}(\rho^{\partial (X/N)}_{\bar{U}}(\chi_i),\rho^{\partial (X/N)}_{\bar{U}}(\xi_j))\leq \Psi_1+O\left(\frac{10B(\mathcal{C'}+L)}{F}\right).$$

\end{enumerate}

We now repeat the argument of Theorem \ref{thm:maintheorem} for each pair $\chi_i,\chi_{i+1}$, substituting Lemma \ref{lem:subdommainconstruction} for Lemma \ref{lem:mainconstruction}. The result is a sequence of paths $\gamma_1,\gamma_2,...,\gamma_{Q-1}$ in $\partial (X/N)$ with domain $[0,1]$, each satisfying the property that for each $\gamma_i$ and each $t\in[0,1]$, 
$$(\gamma(t)\cdot\chi_i)_{\rho^{\bar{X}}_{\bar{U}}}\geq (\chi_i\cdot\chi_{i+1})_{\rho^{\bar{U}}_{\bar{X}}}-\mathcal{C}'\geq10L$$
Applying Lemma \ref{lem:subdomquotBGI}, we conclude that
$$\diam(\rho^{\partial (X/N)}_{\bar{U}}(\gamma_i))\leq O(E)$$
for each $\gamma_i$
Concatenating these paths together and using the second property above, we may take $\Psi_2=O\left(\frac{100B(\mathcal{C'}+L)}{F}\right)$ and obtain the desired result. $\square$

\subsection{Height, Slack, and Finiteness}

From now on, assume we have a fixed open cover $\mathcal{U}$ of $\partial (X/N)$. Recall that the topology on $\partial (X/N)$ is induced by the visual metric.

\begin{definition}
\label{def:height}

Given a point $\xi\in \partial (X/N)$, we define the \emph{height} $H(\xi)$ of $\xi$ to be the smallest nonnegative integer so that the set
$$\{\chi\in\partial (X/N):(\xi\cdot\chi)_{\bar{x}_0}\geq H(\xi)\}$$
is contained in an element of $\mathcal{U}$.

Given $\bar{U}\in\mathcal{\bar{D}}\setminus\{X/N\}$, we say $\bar{U}$ is \emph{sufficiently high} if there exists a minimal integer $H(\bar{U})$ satisfying $0\leq H(\bar{U})\leq d_{X/N}(\bar{x}_0,\rho^{\bar{U}}_{\bar{X}})-100L$ so that the set
$$\{\chi\in\partial (X/N):(\chi\cdot \rho^{\bar{U}}_{\bar{X}})_{\bar{x}_0}\geq H(\bar{U})\}$$
is contained in an element of $\mathcal{U}$. If $\bar{U}$ is sufficiently high, we define its \emph{height} to be this quantity $H(\bar{U})$, and its \emph{slack} $A(\bar{U})$ to be the difference $A(\bar{U}):=d_{X/N}(\bar{x}_0,\rho^{\bar{U}}_{\bar{X}})-H(\bar{U})$. 

\end{definition}

\begin{remark}

Given any finite-diameter set $K\subset X$, one can define the \emph{shadow} of $K$ from $\bar{x}_0$ as the set of all boundary points $\xi$ with at least one $(1,20L)$-quasi-geodesic representative ray $\bar{r}$ originating from $\bar{x}_0$ and passing through $K$. (This definition is roughly equivalent to the usual definition.) Then the slack $A(\bar{U})$ of $\bar{U}$ is coarsely the radius of the largest ball we can draw around $\rho^{\bar{U}}_{\bar{X}}$ such that the shadow of this ball is contained inside some element of $\bar{U}$.

\end{remark}

Note that by definition of the visual metric topology, all points in $\partial (X/N)$ have well-defined heights.

It will be easier to work with `balls' defined in terms of Gromov products, rather than balls defined in terms of the visual metric.

\begin{definition}

Given $\xi\in \partial (X/N)$ and $D\in\reals_{\geq 0}$, define
$$b(\xi,D)=\{\chi\in\partial X:(\xi\cdot\chi)_{\bar{x}_0}\geq D\}$$

\end{definition}

Why is the height/slack terminology needed? Given $\xi\in \partial (X/N)$, the maximum 
$$\max_{\bar{U}\in\bar{\mathcal{D}}\setminus\{X/N\}}d_{\bar{U}}(\rho^{\bar{X}}_{\bar{U}}(\bar{x}_0),\rho^{\partial (X/N)}_{\bar{U}}(\xi))$$
may be unbounded, which would block a key `large rotations' step in a later argument. Slack lets us get around this problem.

\begin{lemma}
\label{lem:slack}

Fix a quantity $A_M$, and let $\bar{\mathcal{D}}_0(A_M)$ be the set of all $\bar{U}\in\bar{\mathcal{D}}_0$ such that the slack $A(\bar{U})$ of $\bar{U}$ satisfies $A(\bar{U})<A_M$. Then for any $\xi\in\partial X$, 
$$\max_{\bar{U}\in\bar{\mathcal{D}}_0(A_M)}d_{\bar{U}}(\bar{x}_0,\xi)<\infty.$$

\end{lemma}

\textbf{Proof.} By Lemma \ref{lem:subdomquotBGI} and the Morse Lemma, the domains $\mathcal{U}$ such that some $(1,20L)$-quasi-geodesic ray $\bar{r}$ does not pass within distance, say, $1000L$ of $\rho^{\bar{U}}_{\bar{X}}$ may be excluded. Fixing a specific $\bar{r}$ defining a boundary point $\xi$, and assume $\rho^{\bar{U}}_{\bar{X}}$ does lie within $1000L$ of $\bar{r}$. Let $V\in\mathcal{U}$ be such that
$$\{\chi\in\partial (X/N):(\xi\cdot\chi)_{\bar{x}_0}\geq H(\xi)\}\subset V.$$
For $\chi\in\partial (X/N)$, we have by analysis above that
$$(\chi\cdot\xi)_{\bar{x}_0}\geq\min((\chi\cdot\rho^{\bar{U}}_{\bar{X}})_{\bar{x}_0},(\xi\cdot\rho^{\bar{U}}_{\bar{X}})_{\bar{x}_0})-20L$$
$$\geq \min((\chi\cdot\rho^{\bar{U}}_{\bar{X}})_{\bar{x}_0},d_{X/N}(\bar{x}_0,\rho^{\bar{U}}_{\bar{X}}))-2000L$$
Therefore, as long as $d_{X/N}(\rho^{\bar{U}}_{\bar{X}},\bar{x}_0)\geq H(\xi)+2000L$, we have the implication
$$(\chi\cdot\rho^{\bar{U}}_{\bar{X}})_{\bar{x}_0}\geq H(\xi)+2000L\leadsto (\chi\cdot\xi)_{\bar{x}_0}\geq H(\xi).$$
By definition of $H(\xi)$, we therefore obtain 
$$H(\bar{U})\leq H(\xi)+2000L.$$ 
With the additional assumption that $\bar{U}\in\bar{\mathcal{D}}_0(A_M)$, we have by definition that
$$d_{X/N}(\bar{x}_0,\rho^{\bar{U}}_{\bar{X}})\leq A_M+H(\bar{U})\leq A_M+H(\xi)+2000L.$$
Furthermore, this bound holds uniformly for all $\bar{U}\in \bar{\mathcal{D}}_0(A_M)$. 

Now, let $\bar{y}$ be a point on $\bar{r}$ such that $d_{X/N}(\bar{x}_0,\bar{y})\geq A_M+H(\xi)+10^{13}L$, so that by definition
$$d_{\bar{U}}(\rho^{\bar{X}}_{\bar{U}}(\bar{x}_0),\rho^{\partial (X/N)}_{\bar{U}}(\xi))\leq d_{\bar{U}}(\rho^{\bar{U}}_{\bar{X}}(\bar{x}_0),\rho^{\bar{U}}_{\bar{X}}(\bar{y}))+O(n).$$
Note that we may choose $\bar{y}$ indepenently of $\bar{U}$. Now, by our hypotheses, there exist points $\hat{x}_0$ and $\hat{y}$ in $\mathcal{X}$ projecting respectively to $\bar{x}_0$ and $\bar{y}$. From functoriality and coarse Lipschitzness of projections, $d_{\mathcal{X}}(\hat{x}_0,\hat{y})$ provides a coarse upper bound for $d_{\bar{U}}(\rho^{\bar{X}}_{\bar{U}}(\bar{x}_0),\rho^{\bar{X}}_{\bar{U}}(\bar{y}))$. 

To conclude, we note that $d_{\mathcal{X}}(\hat{x}_0,\hat{y})$ is fully independent of the $\bar{U}\in \bar{\mathcal{D}}_0(A_M)$ chosen, and so coarsely upper bounds $d_{\bar{U}}(\rho^{\bar{X}}_{\bar{U}}(\bar{x}_0),\rho^{\bar{X}}_{\bar{U}}(\bar{y}))$ for all such $\bar{U}$. This bound completes the proof. $\square$

\

\subsection{Constructing Suitable Paths}

We now move on to verifying the locally finite $1$-disks property. Throughout the proof, we will use a constant $\mathcal{B}$. This constant will not depend on $f_n$ or $\mathcal{U}$; $\mathcal{B}:=1000L$ should suffice.

The following lemma is a straightforward consequence of the triangle inequality for Gromov products.

\begin{lemma}
\label{lem:heightcontinuity}

Fix $\xi\in\partial (X/N)$, and let $\chi$ be either another point in $\partial (X/N)$ or a point in $X/N$. Assume $\chi \in b(\xi,H(\xi)+\mathcal{B})$. Then
$$H(\chi)=_{100L}H(\xi)$$

\end{lemma}

We now prove the following lemma, which will be broken down into several subclaims. The lemma will rely on a running construction; this construction is depicted in Figures \ref{fig:onedisk} and \ref{fig:onediskclose}; note that Figure \ref{fig:onediskclose} is a close-up of part of \ref{fig:onedisk}.

\begin{figure}[htbp]
    \centering 
    \includegraphics[width=1\textwidth]{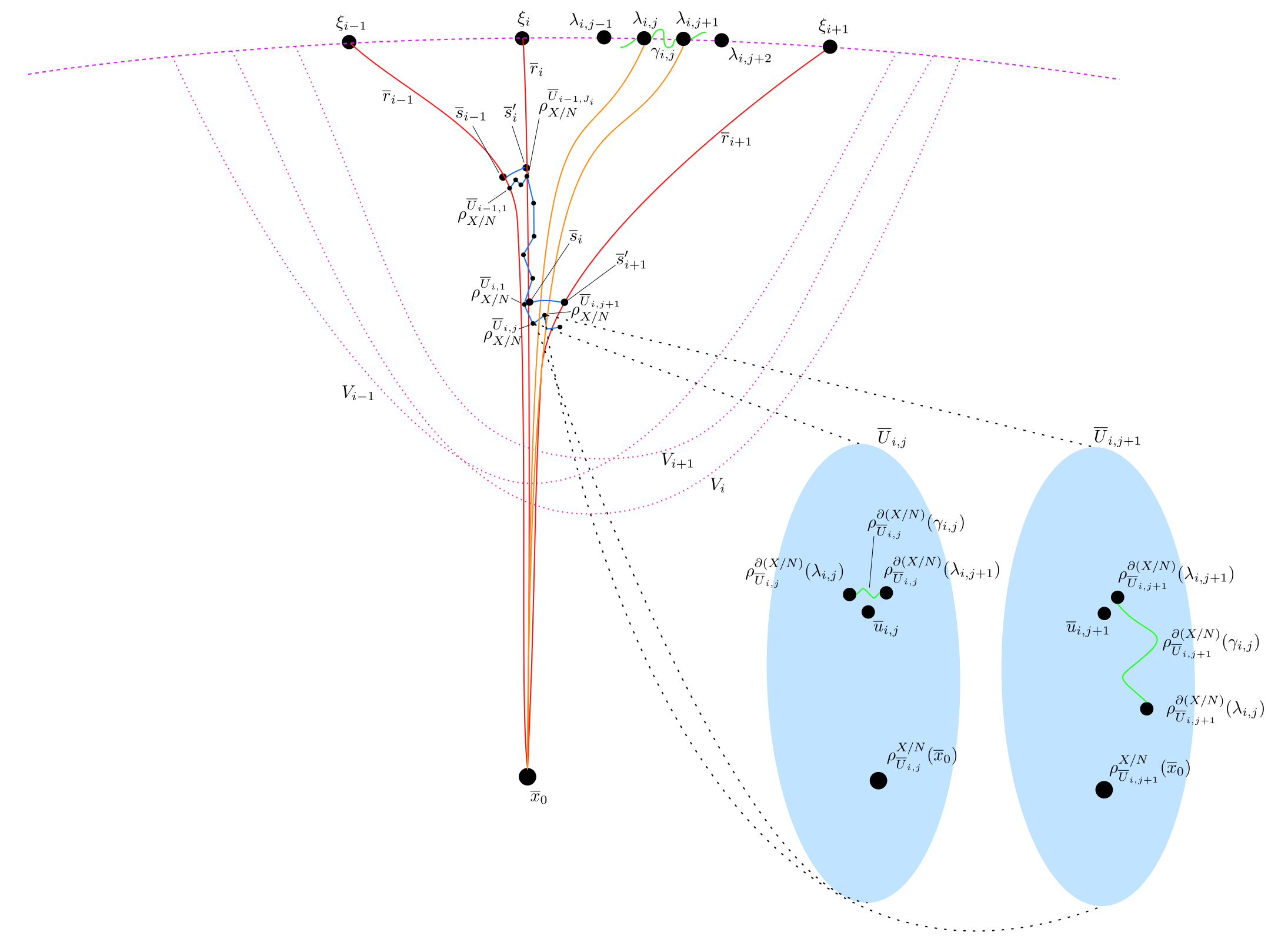} 
    \caption{A schematic of the construction used to prove Lemma \ref{lem:onedisks}, showing the relevant objects under projection to $X/N$, along with windows showing relevant objects in the domains $\bar{U}_{i,j}$ and $\bar{U}_{i,j+1}$. The chains of projection points of domains are such that sequential domains are pairwise orthogonal.}
    \label{fig:onedisk}
\end{figure}

\begin{figure}[htbp]
    \centering 
    \includegraphics[width=1\textwidth]{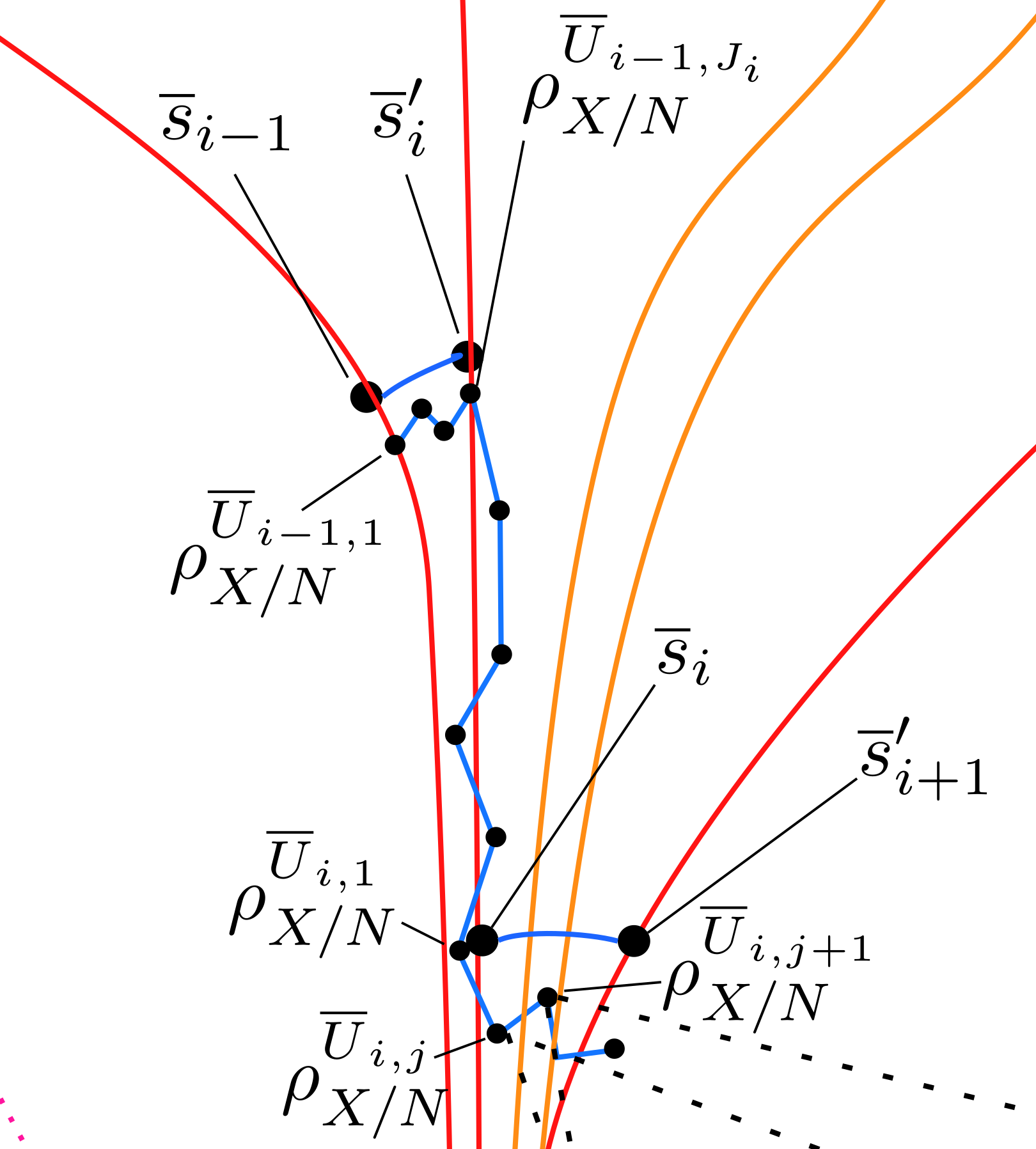} 
    \caption{A close-up of part of Figure \ref{lem:onedisks}.}
    \label{fig:onediskclose}
\end{figure}

\begin{lemma}
\label{lem:onedisks}

Assume $G$, $X$, $\mathcal{X}$, and $X/N$ satisfying Hypotheses \ref{hyp:nob} and \ref{hyp:nobplus}. Then $\partial (X/N)$ satisfies the locally finite $1$-disks property.

\end{lemma}

\textbf{Proof.} We will describe the construction of the paths $g_m$, performing verification as we go. We begin with a path $f_m:[0,1]\to\partial (X/N)$ and our fixed open cover $\mathcal{U}$.

The first step is a variation of the standard Lebesgue number construction. Specifically, we choose a subdivision $0=t_0,t_1,...,t_k=1$ of $[0,1]$ and define $\xi_i=f_n(t_i)$ and $\mathcal{I}_i=f_n([t_i,t_{i+1}])$ so that for all $\mathcal{I}_i$, we have
$$\mathcal{I}_i\subset b(\xi_i,H(\xi_i)+1000\mathcal{B})$$
The existence of such a subdivision follows from compactness of $f_n([0,1])$ and the definition of the $H(\cdot)$, plus an elementary topological argument. Notice the consequence that for all $1\leq i\leq k-1$, $H(\xi_{i+1})=_{100L} H(\xi_i)$. Since each $b(\xi_i,H(\xi_i))$ is contained in some element of $\mathcal{U}$ by construction, let $V_i$ be this element. We also choose for each $\xi_i$ a $(1,20L)$-quasi-geodesic ray $\bar{r}_i$ originating from $\bar{x}_0$ representing $\xi_i$. By removing unnecessary $t_i$, we may also assume without loss of generality that 
$$(\xi_i\cdot\xi_{i+1})_{\bar{x}_0}\leq H(\xi_i)+10000\mathcal{B}$$
We now focus on a specific $\xi_i$, $\mathcal{I}_i$, and $V_i$. 

\begin{lemma}
\label{sublem:orthogonalstring}

Given the sequences $(\xi_i)$, $(\mathcal{I}_i)$, and $(V_i)$ as above, there exists for each $1\leq i\leq k$ a sequence $\bar{U}_{i,1},...,\bar{U}_{i,K_i}$ of domains in $\bar{\mathcal{D}}_0\setminus\{X/N\}$ satisfying the following properties:
\begin{enumerate}
\item The point $\rho^{\bar{U}_{i,1}}_{\bar{X}}$ lies within $2E$ of $\bar{r}_i$, and the point $\rho^{\bar{U}_{i,K_i}}_{\bar{X}}$ lies within $2E$ of $\bar{r}_{i+1}$.
\item For each $\bar{U}_{i,j}$, $\bar{U}_{i,j}\perp \bar{U}_{i,j+1}$.
\item Every $\bar{U}_{i,j}$ lies within $\mathcal{B}$ of $\bar{r}_i$.
\item For each $\bar{U}_{i,j}$, we have
$$H(\xi_i)+100\mathcal{B}=_{10\mathcal{B}} d_{X/N}(\rho^{\bar{U}_{i,j}}_{\bar{X}},\bar{x}_0)$$
and
$$H(\xi_i)+100\mathcal{B}=_{10\mathcal{B}}(\xi_i\cdot \rho^{\bar{U}_{i,j}}_{\bar{X}})_{\bar{x}_0}$$
\item We have that $\bar{U}_{i,K_i}=\bar{U}_{i+1,1}$.

\end{enumerate}

\end{lemma}

\textbf{Proof.} By construction, there exists a point $\bar{s}_i$ on $\bar{r}_i$ such that $d_{X/N}(\bar{s}_i,\bar{x}_0)=_{0.1\mathcal{B}} H(\xi_i)+100\mathcal{B}$. We choose $\bar{s}_{i}'$ analogously on $\bar{r}_{i+1}$. Let $\bar{\eta}=[\bar{s}_i,\bar{s}_{i}']$ be a geodesic in $X/N$ between these two points, and let $\eta=[s_i,s_{i}']$ be an arbitrary lift of $\bar{\eta}$ to $X$.

Since $(\xi_i\cdot\xi_{i+1})_{\bar{x}_0}\geq H(\xi_i)+999\mathcal{B}$, $\bar{r}_i$ and $\bar{r}_{i+1}$ must fellow-travel within the range where $\bar{s}_i$ and $\bar{s}_{i}'$ are chosen, and we may conclude that $\bar{\eta}$, and hence $\eta$, have length at most $0.2\mathcal{B}$. In particular, every point of $\eta$ lies within $\mathcal{B}$ of $\bar{r}_i$ and $\bar{r}_{i+1}$, and has distance $H(\xi_i)+100\mathcal{B}\pm \mathcal{B}$ from $\bar{x}_0$. Applying Hypothesis \ref{hyp:nobplus}, there exists a sequence $U_{i,1},...,U_{i,J_i}$ of domains in $\mathcal{D}_0$ such that $\rho^{U_{i,1}}_X$ lies within $E$ of $s_i$, $\rho^{U_{i,J_i}}_X$ lies within $E$ of $s_{i+1}$, every $\rho^{U_{i+1}}_X$ lies within $E$ of $\eta$, and for all $1\leq i\leq J_i-1$, we have $U_i\perp U_{i+1}$ and $d_X(\rho^{U_i}_X,\rho^{U_{i+1}}_X)\leq 10E$. It is now straightforward to verify that the projected sequence $\bar{U}_{i,1},...,\bar{U}_{i,J_i}$ satisfies the desired properties. 

Given $\bar{s}_{i}'$ and $\bar{s}_{i+1}$ for $1\leq i\leq k-1$, we now  build a new geodesic $\bar{\eta}$ between them. As $H(\xi_{i+1})=_{0.1\mathcal{B}}H(\xi_i)$, this geodesic has length at most $\mathcal{B}$. We now build a sequence $\bar{U}_{i,J_i},\bar{U}_{i,J_i+1},...,\bar{U}_{i,K_i}$ as in the previous paragraph from $\bar{\eta}$; verification of the desired coarse equalities is likewise straightforward. $\square$

\

\begin{lemma}
\label{lem:slackness}

Each $\bar{U}_{i,j}$ constructed in Lemma \ref{sublem:orthogonalstring} has sufficient height, and its slack $A(\bar{U}_{i,j})$ is at most $200\mathcal{B}$.

\end{lemma}

\textbf{Proof.} By construction, we have for each $\bar{U}_{i,j}$ that
$$H(\xi_i)+100\mathcal{B}=_{10\mathcal{B}}d_{X/N}(\rho^{\bar{U}_{i,j}}_{\bar{X}},\bar{x}_0)$$
and 
$$H(\xi_i)+100\mathcal{B}=_{10\mathcal{B}}(\xi_i\cdot\rho^{\bar{U}_{i,j}}_{\bar{X}})_{\bar{x}_0}$$
From Lemma \ref{lem:heightcontinuity}, $H(\bar{U}_{i,j})=_{0.1\mathcal{B}}H(\xi_i)$. From the control on $d_{X/N}(\rho^{\bar{U}_{i,j}}_{\bar{X}},\bar{x}_0)$ given above, certainly $\bar{U}_{i,j}$ has sufficient height. On the other hand, 
$$d_{X/N}(\bar{x}_0,\rho^{\bar{U}_{i,j}}_{\bar{X}})-H(\bar{U}_{i,j})\leq (H(\xi_i)+110\mathcal{B})-(H(\xi_i)-\mathcal{B})\le 200\mathcal{B}.\square$$

\

We now build the paths we will concatenate to construct $g_n$.  Choose a constant $\mathcal{T}$; the exact value will be chosen in a moment. Given $\bar{U}_{i,j}$, recall that $\bar{U}_{i,j}$ contains a quasi-embedded copy of $\reals$. As a result, there exists some point $\bar{u}_{i,j}\in \bar{U}_{i,j}$ such that $d_{\bar{U}_{i,j}}(\bar{u}_{i,j},\rho^{\bar{X}}_{\bar{U}_{i,j}}(\bar{x}_0))\geq 100\mathcal{T}m$. Say that the maps in Lemma \ref{lem:boundaryrealization} are $\Psi_1$-coarsely surjective for $\Psi_1=O(E)$. Also let $J$ be the constant in Lemma \ref{lem:subdomquotBGI}. Using this lemma, we may construct $\lambda_{i,j}\in \partial (X/N)$ so that $\lambda_{i,j}$ $\mathcal{T}$-coarsely projects to $(\bar{u}_{i,j},\bar{u}_{i,j+1})\in \bar{U}_{i,j}\times \bar{U}_{i,j+1}$. We may then apply Lemma \ref{lem:narrowpaths} with $\Psi_1$, letting $\Psi_2$ be the resulting constant. This lemma yields a path $\gamma_{i,j}:[0,1]\to \partial (X/N)$ from $\lambda_{i,j}$ to $\lambda_{i,j+1}$ such that for all $t\in[0,1]$, 
$$d_{\bar{U}_{i,j}}(\rho^{\partial (X/N)}_{\bar{U}_{i,j}}(\gamma_{i,j}(t)),\rho^{\partial (X/N)}_{\bar{U}_{i,j}}(\lambda_{i,j}))\leq \Psi_2$$
We obtain finitely many such $\gamma_{i,j}$, and so may concatenate them together into a single path $g_n:[0,1]\to\partial (X/N)$ so that $g_n|_{[t_i,t_{i+1}]}$ is equal to the concatenation of $\gamma_{i,1}$ through $\gamma_{i,K_i-1}$.

We now require that $\mathcal{T}\geq 100(\Psi_1+\Psi_2+J)$.

It remains to verify that $g_m$ satisfies the two properties of Definition \ref{def:locfinite}. Recall that for $t\in[t_i,t_{i+1}]$, by definition $f_m(t)\in \mathcal{I}_i\subset V_i$.

\begin{lemma}
\label{sublem:onedisksfellowtravel}

For each $\gamma_{i,j}$, $\gamma_{i,j}([0,1])\subset V_i$.

\end{lemma}

\textbf{Proof.} For each $t\in[0,1]$, we have
$$d_{\bar{U}_{i,j}}(\rho^{\partial (X/N)}_{\bar{U}_{i,j}}(\gamma_{i,j}(t)),\rho^{\partial (X/N)}_{\bar{U}_{i,j}}(\lambda_{i,j}))\leq \Psi_2$$
as well as 
$$d_{\bar{U}_{i,j}}(\rho^{\partial (X/N)}_{\bar{U}_{i,j}}(\lambda_{i,j}),\rho^{\partial (X/N)}_{\bar{U}_{i,j}}(\bar{u}_{i,j}))\leq \Psi_1.$$
But also
$$d_{\bar{U}_{i,j}}(\bar{u}_{i,j},\rho^{\bar{X}}_{\bar{U}_{i,j}}(\bar{x}_0))\geq 100\mathcal{T}m.$$
By choice of $\mathcal{T}$ and the triangle inequality, we obtain 
$$d_{\bar{U}_{i,j}}(\rho^{\partial (X/N)}_{\bar{U}_{i,j}}(\gamma_{i,j}(t)),\rho^{\bar{X}}_{\bar{U}_{i,j}}(\bar{x}_0))\geq 10J.$$
Therefore, Lemma \ref{lem:subdomBGI} implies that any $(1,20L)$-quasi-geodesic ray originating at $\bar{x}_0$ representing $\gamma_{i,j}(t)$ passes within $J$ of $\rho^{\bar{U}_{i,j}}_{\bar{X}}$. The fact that $(\xi_i\cdot\rho^{\bar{U}_{i,j}}_{\bar{X}})_{\bar{x}_0}=_{10\mathcal{B}} H(\xi_i)+100\mathcal{B}$ implies that, for $\mathcal{B}\geq 10J+1000L$, we have
$$(\xi_i\cdot\gamma_{i,j}(t))_{\bar{x}_0}\geq H(\xi_i)+50\mathcal{B}$$
By definition of $H(\xi_i)$, this inequality is sufficient to conclude that $\gamma_{i,j}(t)\in V_i$. $\square$

\

Note for the following lemma that any $b(\chi,R)$ with $\chi\in \partial (X/N)$ and $R\in\reals_{>0}$ contains some open neighborhood of $\xi$, essentially by definition of the visual metric topology.

\begin{lemma}
\label{sublem:onedisksavoidance}

For each point $\xi\in\partial (X/N)$, there is some $R_0$ and $M_0$ so that for $m\geq M_0$, $g_m$ does not intersect $b(\xi,R_0)$, where
$$b(\xi,R_0)=\{\chi\in\partial (X/N):(\xi\cdot\chi)_{\bar{x}_0}\geq R_0\}$$

\end{lemma}

\textbf{Proof.} Define
$$\mu:=\max_{\bar{U}\in \bar{\mathcal{D}}_0(10^5\mathcal{B})}d_{\bar{U}}(\rho^{\partial (X/N)}_{\bar{X}}(\xi),\rho^{\bar{X}}_{\bar{U}}(\bar{x}_0))$$
By Lemma \ref{lem:slack}, $\mu$ is finite. Choose $M_0\geq \mu+10E$, and $R_0=H(\xi)+10^{10}\mathcal{B}$.

For every $t\in[0,1]$ and $m\geq M_0$, by construction there exists some $\bar{U}\in\bar{\mathcal{D}_0}(200\mathcal{B})$ (one of the $\bar{U}_{i,j}$ for that $g_m$, with slack at most $200\mathcal{B}$) such that 
$$d_{\bar{U}}(\rho^{\bar{X}}_{\bar{U}}(\bar{x}_0),\rho^{\partial (X/N)}_{\bar{X}}(g_n(t)))\geq 100\mathcal{T}M_0.$$
On the other hand, by definition of $M_0$,
$$d_{\bar{U}}(\rho^{\bar{X}}_{\bar{U}}(\bar{x}_0), \rho^{\partial (X/N)}_{\bar{U}}(\xi))\leq M_0.$$
From the triangle inequality, we may conclude that
$$d_{\bar{U}}(\rho^{\partial (X/N)}_{\bar{U}}(\xi),\rho^{\partial (X/N)}_{\bar{X}}(g_m(t)))\geq \mathcal{T}$$.
As $\mathcal{T}$ was chosen sufficiently large, we may apply Lemma \ref{lem:subdomquotBGI} and obtain
$$(g_m(t)\cdot\xi)_{\rho^{\bar{U}}_{X/N}}\leq E.$$

If $g_m(t)\in b(\xi,R_0)$, we would have $$(g_m(t)\cdot\xi)_{\bar{x}_0}\geq H(\xi)+10^{10}\mathcal{B}$$
by definition. By standard properties of the Gromov product, we must therefore have that
$$d_{X/N}(\bar{x}_0,\rho^{\bar{U}}_{\bar{X}})\geq H(\xi)+10^9\mathcal{B}$$
and that
$$(\rho^{\bar{U}}_{\bar{X}}\cdot\xi)_{\bar{x}_0}\geq H(\xi)+10^9\mathcal{B}.$$
From Lemma \ref{lem:heightcontinuity}, we have that $H(\bar{U})\leq H(\xi)+\mathcal{B}$, and so by definition the slack $A(\bar{U})$ of $\bar{U}$ is at least $10^8\mathcal{B}$. This slack bound contradicts the assumption that $\bar{U}\in\bar{\mathcal{D}}_0(200\mathcal{B})$, so in fact $g_m(t)$ cannot intersect $b(\xi,R_0)$. $\square$

\

We conclude that the $g_m$ we have constructed satisfy the desired conditions, and so $\partial (X/N)$ has the locally finite $1$-disks property, thus proving Lemma \ref{lem:onedisks}. $\square$

\section{The N\"obeling Curve and Applications}

Assembling the results of the last few sections, we have the following Theorem.

\begin{theorem}
\label{thm:nobelingbody}

Let $G$, $\mathcal{X}$, $X$, and $N$ be as in Hypothesis \ref{hyp:nobplus}. Additionally assume that $\dim(\partial X)=0$. Then $\partial (X/N)$ is homeomorphic to the N\"obeling curve.

\end{theorem}

Note that all quasi-trees have zero-dimensional boundaries.

We now identify two examples of classes of groups and spaces satisfying these conditions.

\subsection{Graphs of Groups}

The first class comes from \cite{BR22}, although the reader should note that this work builds on ideas developed around graphs of groups and around cube complexes by many researchers (see, e.g., \cite{BHS17b}).

\begin{definition}

Let $\Gamma$ be a finite simplicial graph with vertex set $V_\Gamma$ and edge set $E_\Gamma$. For each $v\in V_\Gamma$, let $G_v$ be a nontrivial finitely generated group. We define the \emph{graph product} $G_\Gamma$ as the group
$$G_\Gamma:=\left(\bigast_{v\in V_\Gamma} G_v\right)/\langle\langle[g_v,g_w]: g_v\in G_v, g_w\in G_w, \{v,w\}\in E_\Gamma\rangle\rangle.$$

Given an induced subgraph $\Lambda\subset \Gamma$, we have that $G_\Lambda$ is a subgroup of $G_\Gamma$; we call such groups \emph{graph subgroups}.

\end{definition}

With this notation, right-angled Artin groups (RAAGs) are exactly the graph products $G_\Gamma$ where $G_v=\intgr$ for all vertices $v$.

Berlyne and Russell prove that graph products are relative HHGs, with explicitly constructed HHS spaces realizing the HHG actions.

\begin{theorem}[Berlyne-Russell \cite{BR22}, Theorem A]
\label{thm:BR}

Let $G_\Gamma$ be a graph product of groups. Then $G_\Gamma$ is relatively hierarchically hyperbolic relative to its vertex groups, with explicit maximal domain $X_G$. In particular, if all vertex groups of $G_\Gamma$ are hyperbolic or hierarchically hyperbolic, then $G_\Gamma$ is hierarchically hyperbolic.

\end{theorem}

We will now identify the class of graph products to which the N\"obeling curve result will apply.

\begin{definition}
\label{def:collapsegraph}

Let $\Gamma$ be a finite simplicial graph, and let $\Lambda_1,..,\Lambda_k$ be nonempty induced subgraphs satisfying $\Lambda_i\neq\Gamma$. Set $\hat{\Lambda}=\{\Lambda_1,...,\Lambda_k\}$  Recall that the link $\text{lk}(\Lambda)$ of an induced subgraph $\Lambda\subset\Gamma$ is the set of all vertices of $\Gamma$ that are adjacent or equal to every vertex of $\Lambda$. The \emph{collapse graph} $\Gamma_{\hat{\Lambda}}$ of $\Gamma$ relative to $\hat{\Lambda}$ is the graph with vertex set $\Lambda$ and $\Lambda_i$ and $\Lambda_j$ connected by an edge if $\Lambda_i\subset\text{lk}(\Lambda_j)$.
 
\end{definition}

\begin{definition}
\label{def:admissiblegraphs}

Let $G_\Gamma$ be a graph product of groups. We say that $G_\Gamma$ is \emph{admissible} if there exists a collection $\hat{\Lambda}=\{\Lambda_1,...,\Lambda_k\}$ of induced proper subgraphs such that the following conditions are satisfied:
\begin{enumerate}
\item $\Gamma$ is not a nontrivial join.
\item Each $\Lambda_i\in\hat{\Lambda}$ is either a single vertex $v$ such that $G_v$ possesses some finitely generated Cayley graph containing a quasi-isometrically embedded copy of $\reals$, or is a subgraph with at least two vertices that is not a join.
\item Every vertex of $\Gamma$ lies in some $\Lambda_i$.
\item The collapse graph $\Gamma_{\hat{\Lambda}}$ is connected.
\end{enumerate}

\end{definition}

\begin{remark}
Note that in Definition \ref{def:admissiblegraphs}, if all $G_v$ satisfy the embedded-$\reals$ condition, then it suffices for $G_\Gamma$ to be connected and to not be a join, as $\hat{\Lambda}$ can be set to be the vertex set of $\Gamma$.

\end{remark}

\begin{lemma}
\label{lem:graphproduct}

Let $G_\Gamma$ be a graph product of groups with vertex set $V_\Gamma$. Let $N$ be as in hypothesis $\ref{hyp:main}$, and let $X_G$ be the top-level domain obtained in Theorem \ref{thm:BR}. Assume that $G_\Gamma$ is admissible.

Then $\partial (X_G/N)$ is homeomorphic to the N\"obeling curve.

\end{lemma}

To prove this lemma, we need to prove that the setup satisfies Hypotheses \ref{hyp:nobplus} and \ref{hyp:tree}, which require digging into how the relative HHG structure for $G_\Gamma$ is constructed. For each $v\in \Gamma$, let $S_v$ be a finite generating set of $G_v$, and define $S=\bigcup_{v\in \Gamma}S_v$.

\begin{definition}[\cite{BR22}, Defs. 3.1 and 3.3.]
\label{def:graphdomains}

Let $G_\Gamma$ be a graph product. The domains of $G_\Gamma$ are indexed by cosets of the form $gG_\Lambda$, where $g\in G_\Gamma$ and $G_\Lambda$ is a graph subgroup, up to the equivalence relation that $[gG_\Lambda]=[g'G_{\Lambda'}]$ if $\Lambda=\Lambda'$ and $g^{-1}g'\in \langle G_v:v\in \Lambda\cup \text{lk}(\Lambda)\rangle$. Inclusion is defined such that $\Gamma$ is the unique maximal domain.

For each induced subgraph $\Lambda\subset\Gamma$, define its domain $\mathcal{C}(\Lambda)$ to be the graph with vertex set $G_\Lambda$ and edges between vertices $g,h\in G_\Lambda$ exactly when either $g^{-1}h\in \bigcup_{v\in \Lambda} S_v$ or there exists a proper induced subgraph $\Lambda'\subsetneq\Lambda$ such that $g^{-1}h\in G_{\Lambda'}$. Then define $\mathcal{C}([g\Lambda])=g\mathcal{C}(\Lambda)$.

\end{definition}

\begin{definition}[\cite{BR22}, Def. 3.17]

Given a non-maximal domain $g\Lambda$, we define
$$\rho^{g\Lambda}_{\Gamma}=\bigcup_{[h\Lambda]=[g\Lambda]}h G_\Lambda$$
We will often somewhat cheat on notation and use $[g\Lambda]$ as a standin for $\rho^{[g\Lambda]}_\Gamma$ when it is clear from context that we are discussing the image $\rho^{[g\Lambda]}_{\Gamma}$.
\end{definition}

\begin{definition}[\cite{BR22}, Def. 3.19]

Given two domains $[g_1\Lambda_1],[g_2\Lambda_2]$, we define $[g_1\Lambda_1]\perp [g_2\Lambda_2]$ if there exists $h\in G_\Gamma$ with $[h\Lambda_1]=[g_1\Lambda_1]$, $[h\Lambda_2]=[g_2\Lambda_2]$, and $\Lambda_2\in\text{lk}(\Lambda_1)$. 

\end{definition}

We now prove Lemma \ref{lem:graphproduct}.

\

\textbf{Proof of Lemma \ref{lem:graphproduct}.} We need to establish Hypotheses \ref{hyp:nobplus} and \ref{hyp:tree}. Recall that the top-level domain for $G_\Gamma$ is $\mathcal{C}(\Gamma)$. From Theorems 4.4 and 4.10 in \cite{BR22}, $\mathcal{C}(\Gamma)$ is an infinite-diameter quasi-tree, establishing Hypothesis \ref{hyp:tree}.

We now turn to \ref{hyp:nobplus}. From \ref{thm:BR} and standard facts about relative HHGs, it suffices to verify the large dense domains property. Let $\hat{\Lambda}=\{\Lambda_1,...,\Lambda_k\}$ be a set verifying the assumption that $G_\Gamma$ is admissible. We define
$$\mathcal{D}_0=\{[g\Lambda_i]:g\in G_\Gamma, \Lambda_i\in\hat{\Lambda}\}.$$
As the $G_\Lambda$-action on domains satisfies $g[h\Lambda]=[gh\Lambda]$, certainly $\mathcal{D}_0$ is closed under this action. Furthermore, for every $\mathcal{C}([g\Lambda_i])$, either $\Lambda_i=\{v\}$ and $\mathcal{C}([g\Lambda_i])$ is isometric to the Cayley graph $\text{Cay}(G_v,S_v)$, or $\Lambda_i$ is not a join. In the first case, $\text{Cay}(G_v,S_v)$ contains a quasi-isometrically embedded copy of $\reals$ by assumption, and in the second case, we may apply Theorems 4.4 and 4.10 in \cite{BR22} to conclude that $\mathcal{C}([g\Lambda_i])$ is an infinite-diameter quasi-tree, which necessarily contains a quasi-isometrically embedded copy of $\reals$.

Next, for any $g\in \mathcal{C}(\Gamma)$, $g$ lies within $gG_v$ for any $v\in\Gamma$, and hence $g$ lies within $[g\Lambda_i]$ for any $[g\Lambda_i]$.

Lastly, it suffices to show that there is some constant $\mathcal{R}$ with the following property: for any adjacent vertices $g,h\in \mathcal{C}(\Gamma)$ and any $[g_1\Lambda_{j_1}]$ containing $g$, there exists a sequence $[g_1\Lambda_{j_1}],...,[g_q\Lambda_{j_q}]$ so that $[g_i\Lambda_{j_i}]\perp[g_{i+1}\Lambda_{j_{i+1}}]$ for all $[g_i\Lambda_{j_i}]$, all $[g_i\Lambda_{j_i}]$ have distance at most $\mathcal{R}$ from $g$, and $[g_q\Lambda_{j_q}]$ contains $h$. 

One possibility is that $g^{-1}h=s_v$ for some $s_v\in S_v$, where $S_v$ is the generating set of some vertex group $G_v$. Assume without loss of generality that $v\in\Lambda_1$. Then $h\in [g\Lambda_1]$, as in fact $h\in gG_{\Lambda_1}$. On the other hand, $g\in [g_1\Lambda_{j_1}]$ by assumption, so by definition there exists $g_1'$ such that $[g_1'\Lambda_{j_1}]=[g_1\Lambda_{j_1}]$ and $g\in g_1'G_{\Lambda_{j_1}}$. So in fact $[g_1\Lambda_{j_1}]=[g\Lambda_{j_1}]$. 

Now let $\gamma$ be a path in the collapse graph $\Gamma_{\hat{\Lambda}}$ from $\Lambda_{j_1}$ to $\Lambda_1$; here we need the assumption that the collapse graph $\Gamma_{\hat{\Lambda}}$ is connected. Enumerate $\gamma$ as $\gamma=\{\Lambda_{j_1},\Lambda_{j_2},...,\Lambda_{j_q}\}$ with $\Lambda_{j_q}=\Lambda_1$. Considering the sequence $[g\Lambda_{j_1},g\Lambda_{j_2},...,g\Lambda_{j_q}]$, by definition of the collapse graph, we have $[g\Lambda_{j_i}]\perp[g\Lambda_{j_{i+1}}]$ for all $[g\Lambda_{j_i}]$. Furthermore, all these domains clearly have distance $0$ from $g$, while the last domain has distance $0$ from $h$ by choice of $\Lambda_{j_q}=\Lambda_1$.

The other case we must consider is when $g^{-1}h\in G_{\Lambda'}$ for some $\Lambda'\subsetneq \Gamma$. Say $g^{-1}h=s$, and write $s$ as a product of generators $s_1s_2...s_p$, where each $s_i\in G_{\Lambda'}$. Then $gs_1s_2...s_{p'}$ is adjacent to $h$ for all $1\leq p'\leq p$, and we may rerun the previous argument for each $s_i$ and then concatenate the resulting sequences of domains. $\square$



\subsection{Non-Geometric Three-Manifold Groups}

The second class comes from \cite{HRSS25}, which concerns $3$-manifold groups. Our starting point is the following theorem:

\begin{theorem}[Hagen-Russell-Sisto-Spriano, Corollary 1.3]
\label{thm:3man}

If $M$ is a $3$-dimensional non-geometric graph manifold, then $\pi_1(M)$ is a hierarchically hyperbolic group, with explicit top-level domain $X_M$.

\end{theorem}

We have the following result.

\begin{lemma}
\label{lem:3mannob}

Let $M$ be a $3$-dimensional non-geometric graph manifold, and let $X_M$ be the top-level domain obtained in Theorem \ref{thm:3man}. Let $N$ be as in Hypothesis \ref{hyp:main}. 

Then $\partial (X_M/N)$ is homeomorphic to the N\"obeling Curve.
    
\end{lemma}

As before, we need to check Hypotheses \ref{hyp:nobplus} and \ref{hyp:tree}. We note the following facts.

\begin{lemma}
\label{lem:3manprops}
Let $M$ be a $3$-dimensional non-geometric graph manifold, and let $X_M$ be the top-level domain obtained in Theorem \ref{thm:3man}. Then the following properties hold.

\begin{enumerate}

\item There exists a natural quasi-isometry $\sigma:T\to X_M$, where $T$ is an infinite-diameter tree (the Bass-Serre tree for $\pi_1(M)$).
\item There exists a subcollection of non-maximal domains corresponding to vertices $v$ of $T$, and for each such vertex $v$, the corresponding hyperbolic space is a quasi-line $L_v$.
\item Furthermore, the projection of the domain corresponding to $v\in T$ to $X_M$ contains the point $\sigma(v)$.
\item For vertices $v,s\in T$, the domains corresponding to $v$ and $s$ are orthogonal if and only if $v$ and $s$ are adjacent in $T$.

\end{enumerate}

\end{lemma}

\textbf{Proof.} The first three results follow from Lemma 6.8 in \cite{HRSS25}, although the reader should note that the authors use a combinatorial definition of hierarchical hyperbolicity derived from the combinatorial HHG machinery developed in \cite{BHMS24}, and some care must be taken to unravel the relevant definitions. The fourth result follows from the definition of orthogonality in \cite{BHMS24} and Lemma 6.2 in \cite{HRSS25}. $\square$

\

Lemma \ref{lem:3mannob} now follows directly from Lemma \ref{lem:3manprops}.

\printbibliography

\end{document}